\documentclass[reqno,11pt]{amsart}
\usepackage{relsize}
\usepackage{scalerel}
\usepackage{stackengine,wasysym}
\usepackage{todonotes}
\usepackage{xcolor}
\usepackage{mathrsfs}
\usepackage{dsfont}
\usepackage{mathtools}
\usepackage{amsmath}
\mathtoolsset{showonlyrefs}
\usepackage[hyperfootnotes=false,colorlinks=true,linkcolor = blue,
urlcolor  = blue, citecolor = blue]{hyperref}
\usepackage[sort,nocompress]{cite}
\usepackage{cite}
\usepackage{float}

\numberwithin{equation}{section}

\usepackage{amsmath,amssymb} 
\usepackage[toc,page]{appendix}
\usepackage{bbm}
\usepackage{verbatim}
\allowdisplaybreaks
\theoremstyle{plain}
\newtheorem{lemma}{Lemma}[section]
\newtheorem{theorem}[lemma]{Theorem}
\newtheorem{corollary}[lemma]{Corollary}

\newtheorem{proposition}[lemma]{Proposition}
\newtheorem{definition}[lemma]{Definition}

\newtheorem{example}{Example}

\theoremstyle{remark}
\newtheorem{remark}{Remark}

\newcommand*  {\N} {{\mathbb N}}

\newcommand{\Nb}{\mathbb{N}}

\newcommand{\Eb}{\mathbb{E}}
\newcommand{\Pb}{\mathbb{P}}

\newcommand{\Fb}{\mathbb{F}}
\newcommand{\Fc}{\mathcal{F}}
\newcommand{\Fct}{\left( \mathcal{F}_t \right)_{t \geq 0}}

\newcommand{\Uc}{\mathscr{U}}
\newcommand{\Xc}{\mathcal{X}}

\newcommand{\Sc}{\mathcal{S}}

\newcommand{\lnorm}{\left\|}
\newcommand{\rnorm}{\right\|}

\newcommand{\hook}{\hookrightarrow}	

\makeatletter
\newcommand*{\rom}[1]{\expandafter\@slowromancap\romannumeral #1@}
\makeatother

\def\p{\partial}

\begin{document}

\vskip 0.125in

\title[Stochastic Hydrostatic Euler equations]
{Pathwise Solutions for Stochastic Hydrostatic Euler Equations under the Local Rayleigh Condition} 

\date{\today}

\author[R. Hu]{Ruimeng Hu}
\address[R. Hu]
{Department of Mathematics \\
Department of Statistics and Applied Probability \\
     University of California  \\
	Santa Barbara, CA 93106, USA.} \email{rhu@ucsb.edu}

\author[Q. Lin]{Quyuan Lin*}
\address[Q. Lin]
{	School of Mathematical and Statistical Sciences \\
Clemson University\\
Clemson, SC 29634, USA.} \email{quyuanl@clemson.edu}
\thanks{*Corresponding author. School of Mathematical and Statistical Sciences, Clemson University, Clemson, SC 29634, USA. E-mail address: quyuanl@clemson.edu}

\begin{abstract}

The hydrostatic Euler equations are important in the study of atmospheric and oceanic dynamics in the planetary scale. While its deterministic version has been widely studied in the literature, its stochastic version is far less understood. In this paper, we consider the two-dimensional stochastic hydrostatic Euler equations with initial data that are random variables in a suitable Sobolev space satisfying the local Rayleigh condition. We establish local-in-time existence and uniqueness of maximal pathwise solutions. Our work provides the first result on existence and uniqueness in Sobolev spaces, and establishes the first existence of pathwise solutions to the stochastic hydrostatic Euler equations.

\end{abstract}

\maketitle




\section{Introduction} \label{section:introduction}
In this paper, we study the following two-dimensional
stochastic hydrostatic Euler equations:
\noeqref{PE-inviscid-1}
\noeqref{PE-inviscid-2}
\noeqref{PE-inviscid-3}
\noeqref{PE-inviscid-ic}
\begin{subequations}\label{PE-inviscid-system}
\begin{align}
    &d u + (u\partial_x u + w \partial_z u  + \p_x p)dt = \sigma (u) d W , \label{PE-inviscid-1}  
    \\
    &\partial_z p  =0 , \label{PE-inviscid-2}
    \\
    &u_x + w_z =0,  \label{PE-inviscid-3} 
    \\
    &u(0) = u_0. \label{PE-inviscid-ic}
\end{align}
\end{subequations}
This model is also known as the stochastic inviscid primitive equations (PEs), and it is widely used in the study of atmospheric and oceanic dynamics in the planetary scale.
We consider \eqref{PE-inviscid-system} in a periodic channel 
$$\mathbb D:=\mathbb T \times (0,1) = \{(x,z): \, x\in \mathbb R/\mathbb Z,\, 0<z<1\},
$$
with the boundary conditions
\begin{equation}\label{BC}
    u, w, p \text{ are periodic in } x \text{ with period }  1, \quad\text{and}\quad w|_{z=0,1}=0.
\end{equation}
Here $(u,w)$ are the horizontal and vertical velocity, respectively, and $p$ stands for the pressure. The term $\sigma (u) d W$ stands the external forcing driven by white noise.
The corresponding deterministic system of \eqref{PE-inviscid-system} can be formally derived as the hydrostatic limit by taking $\epsilon\rightarrow 0$ in the 2D Euler equations in an $\epsilon$-narrow periodic channel $\mathbb T \times (0,\epsilon)$ and considering the leading order behavior. This formal hydrostatic limit has been rigorously proved in \cite{grenier1999derivation,brenier2003remarks,masmoudi2012h} for the inviscid case under the local Rayleigh condition:
\begin{equation}\label{local-rayleigh}
    \p_{zz} u(t,x,z) \geq \kappa >0
\end{equation}
for some positive constant $\kappa$. See also \cite{azerad2001mathematical,li2019primitive,li2022primitive} for the rigorous derivation of viscous PEs from the Navier-Stokes equations.

{\noindent\bf Main Results.} Our main result, stated in the following theorem, concerns the local-in-time existence and uniqueness of maximal pathwise solutions (see Definition \ref{def:pathwise.sol}). We defer more explanations on the notations used in the theorem to Section \ref{section:preliminary}.
\begin{theorem}\label{thm:main-1}
    Let $\Sc = \left(\Omega, \Fc, \Fct, \Pb \right)$ be a given stochastic basis, and let $s\geq 6$ and $0<\kappa<\frac12$ be fixed. Suppose the initial data $u_0\in L^2\left(\Omega; \mathcal D_{s,2\kappa}\right)$ is $\mathcal F_0$-measurable, and assume the noise $\sigma$ satisfies some proper conditions \eqref{noise-ine}. Then there exists a unique maximal local pathwise solution $(u, \{\eta_n\}_{n\in\mathbb N},\xi)$ to the system \eqref{PE-inviscid-system} in the sense of Definition \ref{def:pathwise.sol}.
\end{theorem}

{\noindent \bf Related literature.}
In the deterministic case, the three-dimensional (3D) PEs with full viscosity were shown to be globally well-posed in Sobolev spaces \cite{cao2007global,kobelkov2006existence,kukavica2007regularity,hieber2016global}. The same result holds when the PEs have only horizontal viscosity \cite{cao2016global,cao2017strong,cao2020global}. The PEs with only vertical viscosity, also called the hydrostatic Navier-Stokes equations, were shown to be ill-posed in Sobolev spaces \cite{renardy2009ill}. In order to obtain the well-posedness, one can consider additional weak dissipation \cite{cao2020well}, or assume the initial data to be Gevrey regular and satisfying condition \eqref{local-rayleigh} \cite{gerard2020well}, or be analytic in the horizontal direction \cite{paicu2020hydrostatic,lin2022effect}. It still remains open  whether the smooth solutions exist globally or form singularity in finite time. For the inviscid PEs (the hydrostatic Euler equations), it has been shown that such a system is ill-posed in Sobolev spaces \cite{renardy2009ill,han2016ill,ibrahim2021finite}. Such ill-posedness can be overcome in the following two situations: 1) In the 2D case, the local well-posedness can be obtained by assuming the initial data satisfying the local Rayleigh condition \eqref{local-rayleigh}\cite{brenier1999homogeneous,masmoudi2012h}; 2) By assuming real analyticity in all directions for general initial data in both 2D and 3D, \cite{ghoul2022effect,kukavica2011local} established the local well-posedness in the space of analytic functions with the radius of analyticity shrinking in time. Unlike the case with horizontal viscosity where the strong solutions exist globally in time, the smooth solutions to the inviscid PEs have been shown to form singularity in finite time \cite{cao2015finite,collot2023stable,ibrahim2021finite,wong2015blowup}.

In the stochastic setting, one considers the system driven by white noise with random initial data in some proper spaces. Having white noise terms in the system and assuming initial data to be random can take into account numerical and empirical uncertainties. It can also offer predictions of not only a realistic trajectory but also the associated uncertainties. Along this line of research, the stochastic PEs with full viscosity was investigated in 2D \cite{glatt2008stochastic,glatt2011pathwise} and in 3D \cite{brzezniak2021well,debussche2011local,debussche2012global,agresti2024stochastic,agresti2025stochastic}. With only horizontal viscosity, the global existence and uniqueness of strong solutions have been established in \cite{saal2021stochastic}. These results on global well-posedness in Sobolev spaces were based on the results from the deterministic case. The inviscid model is far less investigated in the literature. In the authors' previous work \cite{hu2023local}, they established the existence of local martingale solutions (weak solution in the stochastic sense) and pathwise uniqueness of solutions provided that the initial data is analytic. In addition, with some specific noise (random damping and random diffusion), the well-posedness in Gevrey class of the hydrostatic Euler equations was shown in \cite{hu2025regularization}.

As discussed in \cite{hu2023local}, due to the difference between the nonlinear estimates in the analytic framework and in Sobolev spaces, the existence of pathwise solutions to the stochastic hydrostatic Euler equations still remains open. Moreover, there are no works concerning the existence or uniqueness of either martingale solutions or pathwise solutions in Sobolev spaces. This paper aims to fill this gap by studying the 2D system \eqref{PE-inviscid-system} in Sobolev spaces with initial data being a random variable satisfying condition \eqref{local-rayleigh} almost surely. Remark that in this paper, as well as in most related literature, the stochastic term $\sigma(u) d W$ in \eqref{PE-inviscid-1} is understood in the It\^o integral sense. If the model were understood in the Stratonovich sense, usually denoted by $\sigma(u) \circ d W$, the approach performed below may still apply, after converting it back to the It\^o formalism under suitable conditions 
\cite{duan2014effective}. This will add an additional $dt$ term to \eqref{PE-inviscid-1} which contains the Fr\'{e}chet derivative of $\sigma$ on $u$, leading to more involved assumptions and analysis. 

{\noindent\bf Main contribution.} 
1. As far as we know, there have been only studies on the existence of martingale solutions \cite{hu2023local} for the stochastic hydrostatic Euler equations subject to general multiplicative noise. However, a strong solution in the stochastic sense is preferable, since one may need pathwise information or construct solutions on a given filter probability space that has ``meanings'' when modeling a real-world process. This paper is the first work to show the existence of pathwise solutions to these models.

    2. Our previous study of the stochastic hydrostatic Euler equations \cite{hu2023local} relied on the analytic framework. In general, the analyticity of the initial data puts too much restriction, and milder requirements on the initial condition are preferred if possible. This work gives the first result on the existence and uniqueness of pathwise solutions in Sobolev spaces.
    
    3. Motivated by \cite{brenier1999homogeneous,masmoudi2012h,gerard2020well} which overcome the ill-posedness in Sobolev spaces of the deterministic hydrostatic Euler equations  \cite{renardy2009ill,ibrahim2021finite,han2016ill}, we shall solve the stochastic system \eqref{PE-inviscid-system} with the local Rayleigh condition \eqref{local-rayleigh}. To achieve this, we work with a modified version of system \eqref{PE-inviscid-system} by adding a well-chosen cut-off function, which will reconcile the original system before a proper stopping time, and preserve \eqref{local-rayleigh}. 
    We believe that the developed analytical tools can be applied to other stochastic partial differential equations (SPDEs) where the well-posedness in Sobolev spaces can be established with structured initial data, e.g., the Prandtl equations with initial data satisfying monotonicity condition (similar to the local Rayleigh condition) \cite{kukavica2014local,oleinik1966mathematical,alexandre2015well,masmoudi2015local,xin2004global}.

    4. In order to obtain the pathwise solution with desired regularity, we prove an abstract Cauchy theorem (Lemma~\ref{lemma:cauchy}). Different from this type of results developed in the past \cite{glatt2009strong,glatt2014local}, our analysis is much more involved due to  
    the presence of the local Rayleigh condition \eqref{local-rayleigh} and the use of more complicated functional spaces (see Section \ref{sec:prelim-funcspace}). Therefore, we establish our version from scratch. 
    We believe this tool will be directly useful or shed light on the study of certain SPDEs when the required functional spaces have some special structure.

For the purpose of clarity, we describe the organization and contents of the paper by a flowchart in Figure~\ref{fig:my_label}. 

\begin{figure}[htpb]
    \centering
    \includegraphics[width=\textwidth, keepaspectratio = TRUE]{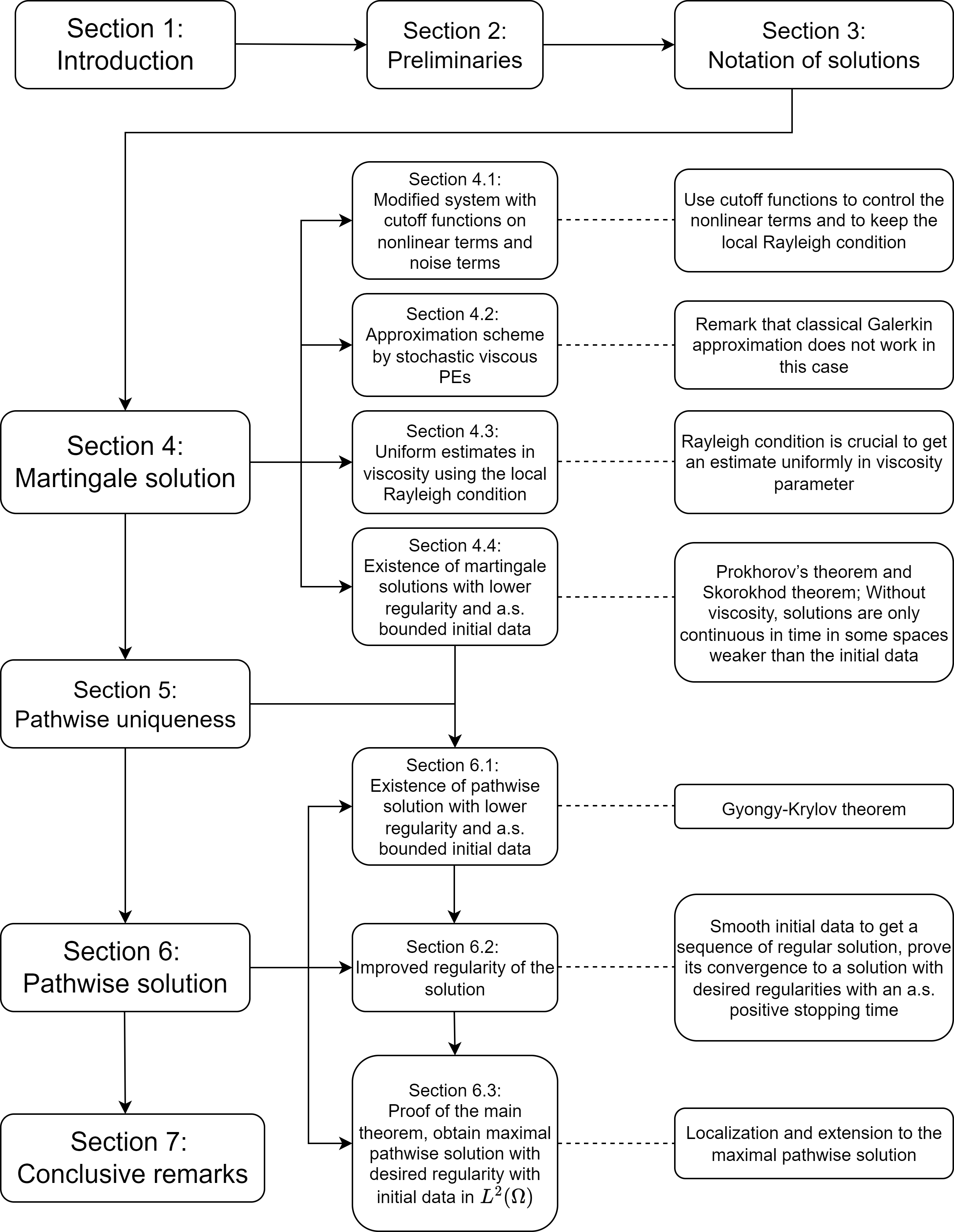}
    \caption{The flowchart for the organization of the paper.}
    \label{fig:my_label}
\end{figure}

\section{Preliminaries}\label{section:preliminary}
In this section, we introduce notations and preliminaries that will be used throughout the paper. The universal constants $c$ and $C$ appearing below may change from line to line. When necessary, we use subscripts to indicate the dependence of the constant on certain parameters, {\it e.g.}, write $C_r$ to emphasize that the constant depends on $r$.  

\subsection{Function spaces and stochastic settings.}\label{sec:prelim-funcspace} 
Let $L^p$ and $W^{s,p}$ with $s\in \mathbb N$ and $p\geq 1$ be the standard Sobolev spaces (\cite{adams2003sobolev}). We denote by $H^s = W^{s,2}$, 
$\|\varphi\|:= \|\varphi\|_{L^2(\mathbb D)}$ and $\|\varphi\|_{H^s} = \|\varphi\|_{H^s(\mathbb D)}$ for simplicity, where the $H^s$ norm is defined by
\begin{equation*}
\begin{split}
    &\|\varphi \|_{H^s}= \sqrt{ \sum\limits_{|\alpha|\leq s} \|D^\alpha \varphi\|^2}.
\end{split}
\end{equation*} 
For $0<\kappa<1$, denote a subspace of $H^s$ as 
\begin{equation*}
    H^s_{\kappa} :=\left\{ \varphi \in H^s: \kappa\leq \p_z\varphi \leq \frac1\kappa \; \text{for all } (x,z)\in \mathbb D \right\},
\end{equation*}
and denote the corresponding norm as
\begin{equation}\label{equiv-hs-norm}
    \|\varphi\|_{\widetilde H^s} := \sqrt{ \sum\limits_{|\alpha|\leq s, D^\alpha \neq \partial_x^s} \|D^\alpha \varphi\|^2 + \lnorm \frac{\p_x^s \varphi}{\sqrt{\p_z\varphi}}\rnorm^2}.
\end{equation}
Notice that when $\varphi \in H^s_{\kappa}$, one has the following equivalence in norms:
\begin{equation}\label{equiv-beta}
    c_{\kappa}\|\varphi\|_{H^s}\leq\frac{c}{1+\lnorm\sqrt{\p_z\varphi}\rnorm_{L^\infty}}\|\varphi\|_{H^s} \leq \|\varphi\|_{\widetilde H^s} \leq C\left(1+\lnorm\frac1{\sqrt{\p_z\varphi}}\rnorm_{L^\infty}\right)\|\varphi\|_{H^s}\leq C_{\kappa}\|\varphi\|_{H^s}.
\end{equation}
Due to the boundary condition \eqref{BC}, one has
\begin{equation}\label{w}
    \int_0^1 \p_x u dz = 0, \quad w(x,z) = -\int_0^z \p_x u(x,\tilde{z})d\tilde{z},
\end{equation}
and we shall consider $u\in H$ where
\begin{equation}\label{space:H}
    H:=\left\{\varphi\in L^2(\mathbb D): \int_0^1 \p_x \varphi dz = 0 \right\}.
\end{equation}

 As shown in \cite{masmoudi2012h}, the vorticity $\p_z u$ plays an important role in the study of system \eqref{PE-inviscid-system}. For this reason, we define
\begin{equation}\label{def:v}
    v= \p_z u,
\end{equation}
for convenience. By taking $\p_z$ to \eqref{PE-inviscid-system}, we get 
\begin{align*}
    &d v + (u\partial_x v + w\partial_z v  )dt = \partial_z \sigma (u) dW , 
    \\
    &v(0) = \partial_z u_0. 
\end{align*}
Given the relation \eqref{w}, we have the following Poincar\'e inequalities:
\begin{equation}\label{poincare}
    \|\p_x^{k+1} u\|\leq C \|\p_x^{k+1} v\|,\qquad \|\p_x^{k} w\|\leq C \|\p_x^{k+1} u\| \leq C \|\p_x^{k+1} v\|, \qquad k=0,1,2,\cdots.
\end{equation}
Moreover, for $|\alpha|=k$ with $0<k\leq s$, and $D^\alpha\neq (k,0)$, i.e., not all the derivatives hit on $x$, we have the following inequalities:
\begin{equation}\label{ine:controls}
   \|D^\alpha u\|\leq \|D^\alpha v\|\leq \|v\|_{H^k}, \quad \|D^\alpha w\|\leq \|v\|_{H^k},     \quad \|u\|_{H^s} \leq \|u\| + \|v\|_{H^s}.
\end{equation}

Let $\mathcal D_s$ and $\mathcal D_{s, \kappa}$ be the Hilbert spaces defined by
\begin{equation}\label{domain:D}
    \begin{split}
        \mathcal D_s:=\{\varphi\in H\cap H^s: \p_z \varphi \in H^s \}, \qquad
        \mathcal D_{s,\kappa}:=\{\varphi\in H\cap H^s: \p_z \varphi \in H_\kappa^s \},
    \end{split}
\end{equation}
with the corresponding norms
\begin{equation}\label{norm:D}
    \begin{split}
        \|\varphi\|_{s} \equiv \|\varphi\|_{\mathcal D_s} :=\|\varphi\| + \|\p_z \varphi\|_{H^s}, \qquad
        \|\varphi\|_{\tilde s} \equiv \|\varphi\|_{\mathcal D_{s,\kappa}} :=\|\varphi\| + \|\p_z \varphi\|_{\widetilde H^s}.
    \end{split}
\end{equation}
In this paper, we shall consider solutions at each time $t$ as random variables taking values in $\mathcal D_s$ or $\mathcal D_{s,\kappa}$, for $0<\kappa<\frac12$ and $s\geq 6$. 
Notice that due to \eqref{ine:controls} and \eqref{equiv-beta}, one has the following equivalences: $$\|\varphi\|_{s}\sim \|\varphi\|_{H^s} + \|\p_z \varphi\|_{H^s}, \quad \quad \|\varphi\|_{\tilde s}\sim \|\varphi\|_{H^s} + \|\p_z \varphi\|_{\widetilde H^s},$$
and there exist some constants ${\tilde c}_\kappa$ and ${\tilde C}_\kappa$ such that
\begin{equation}\label{ctildekappa}
    {\tilde c}_\kappa \|\varphi\|_{s} \leq \|\varphi\|_{\tilde s} \leq {\tilde C}_\kappa \|\varphi\|_{s}.
\end{equation}

Let $\Sc = \left(\Omega, \Fc, \Fb = \Fct, \Pb\right)$ be a filtered stochastic basis. Denote by $\Uc$ a separable Hilbert space and by $W$ an $\Fb$-adapted cylindrical Wiener process with reproducing kernel Hilbert space $\Uc$ on $\Sc$. By picking a complete orthonormal basis $\lbrace e_k \rbrace_{k = 1}^\infty$ for $\Uc$, $W$ may be formally written as $W = \sum_{k=1}^\infty e_kW^k$, where $W^k$ are independent one-dimensional (1D) Wiener processes on $\Sc$.

Consider another separable Hilbert space $X$ and denote by $L_2(\Uc, X)$ the collection of Hilber-Schmidt operators from $\Uc$ into $X$. For a predictable process $\Phi \in L^2\bigl(\Omega; L^2\left(0, T; L_2\left(\Uc, X \right)\right)\bigr)$, one may define the It\^o stochastic integral with respect to the cylindrical Wiener process by
\[
	\int_0^T \Phi \, dW = \sum_{k = 1}^\infty \int_0^T \Phi e_k \, dW^k.
\]
Such integrals can also be extended to $\Phi$ with $\int_0^T \| \Phi \|^2_{L_2\left( \Uc, X \right)} \, dt < \infty$, $\mathbb{P}$-almost surely, and we refer readers to \cite[Section 4]{da2014stochastic} for more details. 

Next, we recall the definitions of Sobolev spaces with fractional time derivative, see {\it e.g.}\ \cite{simon1990sobolev}. Let $X$ be a separable Hilbert space as before, and let $t > 0$, $p > 1$ and $\alpha \in (0, 1)$. We define
\[
	W^{\alpha, p}(0, t; X) :=  \left\{ u \in L^p(0, t; X) \mid \int_0^t \int_0^t \frac{\vert u(s) - u(r) \vert_X^p}{|s - r|^{1+\alpha p}} \, dr \, ds < \infty \right\}
\]
and equip it with the norm
\[
	\| u \|^{p}_{W^{\alpha, p}(0, t; X)} := \int_0^t \vert u(s) \vert^p_X \, dt' + \int_0^t \int_0^t \frac{\vert u(s) - u(r) \vert_X^p}{|s - r|^{1+\alpha p}} \, dr \, ds.
\]

We conclude this subsection by recalling the following two versions of the Burkholder-Davis-Gundy (BDG) inequality, which will be repeatedly used in the sequel. For $\Phi \in L^2\left(\Omega; L^2\left(0, T; L_2\left(\Uc, X\right)\right)\right)$, one has
\begin{equation}
	\label{eq:bdg}
	\Eb \sup_{t \in \left[0, T\right]} \left| \int_0^t \Phi \, dW \right|^r_X \leq C_{r} \, \Eb \left( \int_0^T \| \Phi \|_{L_2(\Uc, X)}^2 \, dt \right)^{r/2}.
\end{equation}
In addition, if $p \geq 2$ and $\Phi \in L^p\left(\Omega; L^p\left(0, T; L_2\left(\Uc, X\right)\right)\right)$, then
\begin{equation}
	\label{eq:bdg.frac}
	\Eb \left| \int_0^\cdot \Phi \, dW \right|^p_{W^{\alpha, p}(0, T; X)} \leq C_{p} \, \Eb \int_0^T \| \Phi \|_{L_2(\Uc, X)}^p \, dt,
\end{equation}
for $\alpha \in [0, 1/2)$. For proofs, see, for instance, \cite{karatzas2012brownian}  and \cite[Lemma 2.1]{flandoli1995martingale}.

\subsection{Assumption and example of the noise term \texorpdfstring{$\sigma(u)$}{Lg}}
In this work, we assume that the noise is multiplicative and satisfies that, for $|\alpha|\leq s$,
\begin{equation}\label{noise-ine}
\begin{split}
    \|\sigma(u)\|_{L_2(\mathscr U, L^2)} &\leq C(1+\|u\|), 
    \\
    \|D^\alpha \p_z \sigma(u)\|_{L_2(\mathscr U, L^2)}&\leq C(1+\|u\|_{s}),
    \\
    \|\sigma(u) - \sigma(u^\#)\|_{L_2(\mathscr U, L^2)} &\leq C\|u-u^\#\|,
    \\
    \|D^\alpha \p_z \sigma(u) - D^\alpha \p_z \sigma(u^\#)\|_{L_2(\mathscr U, L^2)} &\leq C\|u-u^\#\|_{s},
\end{split}
\end{equation}
for $u, u^\# \in \mathcal D_{s}$. 
We present the following example for further understanding of condition~\eqref{noise-ine}. 
\begin{example}[Noise satisfies \eqref{noise-ine}]
\label{example.1}
Let $\psi_k, \chi_k \in W^{s+1,\infty}(\mathbb D)$ satisfy
\begin{equation*}
\sum_{k=1}^\infty \|\psi_k\|_{W^{s+1,\infty}}^2 = \kappa_1^2, \quad \sum_{k=1}^\infty \|\chi_k\|_{W^{s+1,\infty}}^2 = \kappa_2^2,
\end{equation*}
for some $\kappa_1, \kappa_2 \geq 0$. Define
\begin{equation}
	\label{example:sigma.maximal}
	\sigma(u) \zeta = \sum_{k=1}^\infty \zeta_k \left[  \psi_k u + \chi_k \right], \qquad \zeta = \sum\limits_{k=1}^\infty \zeta_k e_k \in \Uc.
\end{equation}
When $u, v\in H^s$, $\sigma$ satisfies
\begin{align*}
    \Vert \sigma(u) \Vert^2_{L_2(\Uc, L^2)} &= \sum_{k=1}^\infty \|   \psi_k u + \chi_k \|^2 \leq C(\kappa_1^2 \|u\|^2+ \kappa_2^2) \leq C_{\kappa_1,\kappa_2} (1+\|u\|^2), \\
    \Vert D^\alpha \p_z \sigma(u) \Vert^2_{L_2(\Uc, L^2)} &= \sum_{k=1}^\infty \| D^\alpha \p_z ( \psi_k u + \chi_k) \|^2 \leq C(\kappa_1^2 (\|u\|^2_{H^s} + \|v\|^2_{H^s}) + \kappa_2^2)
    \\
    &\qquad \leq C_{\kappa_1,\kappa_2} (1+\|u\|_{s}^2).
\end{align*}
Replacing $u$ by $u^\#$ in \eqref{example:sigma.maximal}, the Lipschitz continuity can be verified by:
\begin{align*}
        \Vert \sigma(u) - \sigma(u^\#)\Vert^2_{L_2(\Uc, L^2)} &= \sum_{k=1}^\infty \|  \psi_k (u-u^\#) \|^2
        \leq  C_{\kappa_1} \|u-u^\# \|^2, \\
        \Vert D^\alpha \p_z\sigma(u) - D^\alpha \p_z \sigma(u^\#)\Vert^2_{L_2(\Uc, L^2)} &= \sum_{k=1}^\infty \| D^\alpha \p_z \psi_k (u-u^\#) \|^2
        \leq  C_{\kappa_1} \|u-u^\# \|_{s}^2 .
\end{align*}
\end{example}

\subsection{Preliminary estimates} 
The spaces $\mathcal D_{s}$ and $\mathcal D_{s,\kappa}$ have the following property.
\begin{lemma}\label{lemma:cpt-emb}
The embedding $\mathcal D_{s+1} \hook \hook \mathcal D_{s}$ and $\mathcal D_{s+1,\kappa}\hook \hook \mathcal D_{s,\kappa}$ are compact.
\end{lemma}
\begin{proof}
Let $\{\varphi_n\}$ be a bounded sequence in $\mathcal D_{s+1}$. Then $\{\|\varphi_n\|_{H^{s+1}}\}$ and $\{\|\p_z\varphi_n\|_{H^{s+1}}\}$ are both bounded. Since $H^{s+1} \hook \hook H^s$ is compact, there exists a subsequence $\{\varphi_{n_k}\}$ such that $\{\|\varphi_{n_k}\|_{H^s}\}$ is a Cauchy sequence. Since $\{\|\p_z\varphi_{n_k}\|_{H^{s+1}}\}$ is also bounded, there exists a further subsequence $\{\varphi_{n_{k_l}}\}$ such that $\{\|\p_z\varphi_{n_{k_l}}\|_{H^{s}}\}$ is also Cauchy. Therefore the subsequence $\{\varphi_{n_{k_l}}\}$ is Cauchy in $\mathcal D_{s}$, and thus the embedding $\mathcal D_{s+1} \hook \hook \mathcal D_{s}$ is compact.

The embedding $\mathcal D_{s+1, \kappa} \hook \hook \mathcal D_{s, k}$ can be proved similarly due to the equivalence of norms in \eqref{equiv-beta}.
\end{proof}

The following lemma will be useful in nonlinear estimates.
\begin{lemma}\label{lemma:nonlinear-ine}
Assume $f, g \in H^N$. Then for any multi-index $\alpha=\left(\alpha_1, \cdots\right.$, $\left.\alpha_n\right),|\alpha| \leq N$, we have
\begin{enumerate}
    \item $\left\|D^\alpha(f g)\right\|_{L^2} \leq C_N\left(\|f\|_{L^{\infty}} \left\|D^N g\right\|_{L^2}+\|g\|_{L^{\infty}} \left\|D^N f\right\|_{L^2}\right)$,
    \item $\left\|D^\alpha(f g)-f D^\alpha g\right\|_{L^2} \leq C_N\left(\|D f\|_{L^{\infty}} \left\|D^{N-1} g\right\|_{L^2}+\|g\|_{L^{\infty}} \left\|D^N f\right\|_{L^2}\right)$,
    \item $\left\|D^\alpha(f g)-f D^\alpha g - g D^{\alpha}f \right\|_{L^2} \leq C_N\left(\|D f\|_{L^{\infty}} \left\|D^{N-1} g\right\|_{L^2}+\|Dg\|_{L^{\infty}} \left\|D^{N-1} f\right\|_{L^2}\right)$
\end{enumerate}
\end{lemma}
\begin{proof}
Parts (i)--(ii) follow from \cite[Appendix A.1]{klainerman1981singular}. Although the results stated in \cite[Appendix A.1]{klainerman1981singular} is for torus, the proof their works for general bounded domain. For part (iii), noticing that 
\begin{equation*}
    \begin{aligned}
\left\|D^\alpha(f g)-f D^\alpha g - gD^\alpha f\right \|_{L^2} & \leq C_{(\alpha)} \sum_{\beta+\gamma=\alpha, \beta,\gamma \neq 0}\left\|D^\beta f \cdot D^\gamma g\right\|_{L^2} \\
&=C_{(\alpha)} \sum_{\left|\beta^{\prime}\right|+|\gamma^{\prime}| \leqq N-2}\left\|D^{\beta^{\prime}}(D f) \cdot D^{\gamma^{\prime}} (Dg)\right\|_{L^2},
\end{aligned}
\end{equation*}
then the result follows by applying (i).
\end{proof}

We recall the following compactness results which are needed in this paper. The proofs can be found in \cite[Theorem 5]{simon1986compact} and \cite[Theorem 2.1]{flandoli1995martingale}, respectively.

\begin{lemma}
\label{lemma:aubin-lions}
a) \emph{(Aubin-Lions-Simon Lemma).} Let $X_2 \subset X \subset X_1$ be Banach spaces so that the embedding $X_2 \hook \hook X$ is compact and the embedding $X \hook X_1$ is continuous. Suppose $p \in (1, \infty)$ and $\alpha \in (0, 1)$, the following embedding is compact
\[
	L^p(0, t; X_2) \cap W^{\alpha, p}(0, t; X_1) \hook \hook L^p(0, t; X).
\]

b) Let $X_2 \subset X$ be a Banach space so that $X_2$ is reflexive and the embedding $X_2 \hook \hook X$ is compact. Assume $\alpha \in (0, 1]$, $p \in (1, \infty)$, and such that $\alpha p > 1$, then the following embedding is compact
\[
	W^{\alpha, p}(0, t; X_2) \hook \hook C([0, t], X).	
\]
\end{lemma}

\section{Notion of Solutions}

In this paper, we consider both martingale solutions ({i.e.}, weak solutions in the stochastic sense) and pathwise solutions ({i.e.}, strong solutions in the stochastic sense). 

For martingale solutions, we consider the initial data to be given by a Borel measure $\mu_0$ in $\mathcal D_{s,2\kappa}$ such that, for some $M> 0$ and $0<\kappa<\frac12$,
\begin{equation}\label{condition:mu-zero}
    \mu_0\left(\left\{ u\in \mathcal D_{s,2\kappa}: \Vert u  \Vert_{\tilde s} \geq \frac M2 \right\} \right) = 0.
\end{equation}
In particular, this implies that for any 
$p \geq 1$,
\begin{equation}
	\label{eq:mu.zero}
	\int_{\mathcal D_{s,2\kappa}} \Vert u  \Vert_{\tilde s}^p  \, d\mu_0(u) < \infty.
\end{equation}

\begin{definition}[Martingale solution]\label{definition:martingale-sol}
Let $s\geq 6$ and $\mu_0$ satisfy \eqref{condition:mu-zero} with some positive constant $M>0$. Assume that $\sigma$ satisfies \eqref{noise-ine}. We call a quadruple $(\Sc, W, u, \eta)$ a \emph{local martingale solution} to system \eqref{PE-inviscid-system} if $\Sc = \left(\Omega, \Fc, \Fb, \Pb\right)$ is a stochastic basis, $W$ is an $\Fb$-adapted cylindrical Wiener process with reproducing kernel Hilbert space $\Uc$, $\eta$ is an $\Fb$-stopping time and $u\left(\cdot \wedge \eta\right)$ is a progressively measurable process such that $\eta > 0$ $\Pb$-a.s. and for all $T>0$,
\begin{equation}\label{eq:solution.regularity}
	\mathds{1}_{[0, \eta]}(\cdot)  u(\cdot)  \in L^2\left(\Omega; L^\infty\left( 0, T; \mathcal D_{s,\kappa} \right) \right), \quad  u\left( \cdot \wedge \eta \right) \in L^2\left(\Omega; C\left([0, T], \mathcal D_{s-1,\kappa}\right)\right), 
\end{equation}
where the law of $u(0)$ is $\mu_0$ and $u$ satisfies the following equality in $H$ for all $t\geq 0$:
\begin{equation}\label{eq:solution.def}
	u\left(t \wedge \eta\right) + \int_0^{t \wedge \eta} \big[u\p_x u + w\p_z u\big] dt' = u(0) + \int_0^{t \wedge \eta} \sigma(u) \, dW.
\end{equation}
Moreover, if $\eta = \infty$ $\Pb$-a.s.\ we call the triple $(\Sc, W, u)$ a \emph{global martingale solution}.
\end{definition}
\begin{remark}
The pressure gradient $\p_x p$ vanishes in \eqref{eq:solution.def} as we consider the equality in the space $H.$
\end{remark}

\begin{definition}[Pathwise solution]
\label{def:pathwise.sol}
Suppose that $s\geq 6$ and $\sigma$ satisfies \eqref{noise-ine}. Let $\Sc = \left(\Omega, \Fc, \Fb, \Pb \right)$ be a stochastic basis and let $W$ be a given $\Fb$-adapted cylindrical Wiener process with reproducing kernel Hilbert space $\Uc$. For $0<\kappa<\frac12$, let $u_0 \in L^2\left(\Omega;  \mathcal D_{s,2\kappa}\right)$ be an $\Fc_0$-measurable random variable.
\begin{enumerate}
	\item A pair $(u, \eta)$ is called a \emph{local pathwise solution} to system \eqref{PE-inviscid-system} if $\eta$ is an $\Fb$-stopping time and $u(\cdot \wedge \eta)$ is a progressively measurable stochastic process satisfying \eqref{eq:solution.def}, and $$u(\cdot \wedge \eta)\in L^2\left(\Omega; C\left([0, T], \mathcal D_{s,\kappa}\right)\right).$$
	\item A triple $\left(u, \{\eta_n\}_{n\in\mathbb N},\xi\right)$ is called a \emph{maximal pathwise solution} if each pair $(u,\eta_n)$ is a local pathwise solution, $\eta_n$ is increasing with $\lim\limits_{n\to \infty} \eta_n = \xi$ $\Pb$-a.s., and 
 \[
 \sup\limits_{t\in[0,\xi)} \|u\|_{\tilde s}^2 = \infty, \text{ or }\, \|\partial_{zz} u - \partial_{zz} u_0\|_{L^\infty} = \frac{\kappa}4 \quad \text{ on the set } \{\xi<\infty\}.
 \]
\end{enumerate}
\end{definition}

\begin{remark}
    The condition $\|\partial_{zz} u - \partial_{zz} u_0\|_{L^\infty} = \frac{\kappa}4$ is due to the choice of stopping time constructed in the proof of Theorem \ref{thm:main-1}. For more details see Section \ref{sec:proof-thm}. 
\end{remark}

\section{Martingale Solutions}\label{sec:MartSoln}

This section aims to establish the local existence of martingale solutions. To this end, we first define a modified system~\ref{PE-modified-system} via multiplying the $dt$ and $dW$ terms by some carefully-chosen cut-off functions. Then by considering its full viscous version as the approximation scheme 
and deriving uniform estimates (cf. Proposition~\ref{proposition:estimate-inviscid}), we obtain the global existence of martingale solutions to the modified system~\eqref{PE-modified-system} using compactness arguments (cf. 
 Proposition~\ref{prop:global.martingale.existence}). Finally, by defining a proper stopping time, we obtain the local existence of martingale solutions to the original system \eqref{PE-inviscid-system} (cf. Corollary~\ref{cor:loc.mart.sol}). 

\subsection{The modified system}
Let $\rho>0$  and consider the function $\theta_\rho(x)\in C^\infty([0,\infty))
$ to be a non-increasing cut-off function such that 
\begin{equation}\label{eqn:rho}
    \mathbbm{1}_{[0, \frac{\rho}2]} \leq \theta_\rho(x) \leq \mathbbm{1}_{[0, \rho]}.
\end{equation}
For $\rho>0$, $0<\kappa<\frac12$, and $s\geq 6$, we consider the modified version of system \eqref{PE-inviscid-system}:
\noeqref{PE-modified-1}
\noeqref{PE-modified-2}
\noeqref{PE-modified-3}
\noeqref{PE-modified-ic}
\begin{subequations}\label{PE-modified-system}
\begin{align}
    &d u + \theta_\rho(\|u\|_{s-1})\theta_\kappa(\|\partial_z v - \p_z v_0\|_{L^\infty})(u\partial_x u + w\partial_z u + \p_x p )dt \nonumber
    \\
    & \qquad\qquad\qquad\qquad\qquad= \theta_\rho(\|u\|_{s-1})\theta_\kappa(\|\partial_z v - \p_z v_0\|_{L^\infty})\sigma (u) dW , \label{PE-modified-1}  
    \\
    & \p_z p =0 , \label{PE-modified-2}
    \\
    & \p_x u + \p_z w = 0,  \label{PE-modified-3} 
    \\
    &u(0) = u_0, \label{PE-modified-ic}
\end{align}
\end{subequations}
where $v=\p_z u$ satisfies
\begin{align*}
    &d v + \theta_\rho(\|u\|_{s-1})\theta_\kappa(\|\partial_z v - \p_z v_0\|_{L^\infty})(u\partial_x v + w\partial_z v )dt \nonumber
    \\
    &\qquad\qquad\qquad\qquad\qquad= \theta_\rho(\|u\|_{s-1}) \theta_\kappa(\|\partial_z v - \p_z v_0\|_{L^\infty})\partial_z \sigma (u) dW , 
    \\
    &v(0) = \partial_z u_0. 
\end{align*}
Remark that the cut-off function $\theta_\kappa$ is to keep the local Rayleigh condition valid.

\subsection{The approximation scheme}\label{section:approximation}
When considering the approximation scheme of system \eqref{PE-modified-system}, the Galerkin approximation is usually a good choice \cite{hu2023local,brzezniak2021well,debussche2011local}. However, in our situation, the projection of certain nonlinear terms will destroy the essential cancellation, and thus prohibits us from closing the energy estimates in Sobolev spaces. See Remark \ref{rmk:cancellation} for more details.

Instead, we consider the viscous version of system \eqref{PE-modified-system} as the approximation scheme, for $n \in \Nb$:
\noeqref{PE-modified-system-approximation-1}
\noeqref{PE-modified-system-approximation-2}
\noeqref{PE-modified-system-approximation-3}
\noeqref{PE-modified-system-approximation-ic}
\begin{subequations}\label{PE-modified-system-approximation}
\begin{align}
    &d u_n + \theta_\rho(\|u_n\|_{s-1})\theta_\kappa(\|\partial_z v_n - \p_z v_0\|_{L^\infty})\left(u_n\partial_x u_n + w_n\partial_z u_n + \p_x p_n -\frac1n \partial_{xx} u_n  \right)dt \nonumber
    \\
    & \qquad\qquad\qquad\qquad\qquad= \theta_\rho(\|u_n\|_{s-1})\theta_\kappa(\|\partial_z v_n - \p_z v_0\|_{L^\infty})\sigma (u_n) dW , \label{PE-modified-system-approximation-1}  
    \\
    & \p_z p_n =0 , \label{PE-modified-system-approximation-2}
    \\
    & \p_x u_n + \p_z w_n = 0,  \label{PE-modified-system-approximation-3} 
    \\
    &u_n(0) = u_0, \label{PE-modified-system-approximation-ic}  
\end{align}
\end{subequations}
where $v_n=\p_z u_n$. Consequently, by taking $\p_z$ to system \eqref{PE-modified-system-approximation}, we have $v_n$ satisfying
\noeqref{PE-vorticity-modified-system-approximation-ic}
\begin{subequations}\label{PE-vorticity-modified-system-approximation}
\begin{align}
    &d v_n + \theta_\rho(\|u_n\|_{s-1})\theta_\kappa(\|\partial_z v_n - \p_z v_0\|_{L^\infty})\left(u_n\partial_x v_n + w_n\partial_z v_n -\frac1n \partial_{xx} v_n \right )dt \nonumber
    \\
    & \qquad\qquad\qquad\qquad\qquad= \theta_\rho(\|u_n\|_{s-1})\theta_\kappa(\|\partial_z v_n - \p_z v_0\|_{L^\infty})\p_z \sigma (u_n) dW , \label{PE-vorticity-modified-system-approximation-1}  
    \\
    &v_n(0) = \p_z u_0. \label{PE-vorticity-modified-system-approximation-ic}  
\end{align}
\end{subequations}

Analogue to Definition~\ref{def:pathwise.sol} and from \cite{brzezniak2021well,debussche2011local}, one can define pathwise solutions to system \eqref{PE-modified-system-approximation}, with \eqref{eq:solution.regularity} replaced by 
\begin{equation}\label{eq:solution.regularity-fullviscous}
	\mathds{1}_{[0, \eta]}(\cdot) u_n(\cdot)  \in L^2\left( \Omega; L^2\left( 0, T; \mathcal D_{s+1} \right) \right), \quad  u_n\left( \cdot \wedge \eta \right) \in L^2\left(\Omega; C\left([0, T], \mathcal D_{s,\kappa}\right)\right), 
\end{equation}
and \eqref{eq:solution.def} by
\begin{equation}\label{eq:solution.def-fullviscous}
	u_n\left(t \wedge \eta\right) + \int_0^{t \wedge \eta} \big[u_n\p_x u_n + w_n\p_z u_n - \frac1n \partial_{xx} u_n\big] dt' = u(0) + \int_0^{t \wedge \eta} \sigma(u_n) \, dW.
\end{equation}

For each $n\in\mathbb N$, the existence and uniqueness of the global pathwise solution to the stochastic PEs with either full viscosity or horizontal viscosity is a standard result, see for example, \cite{brzezniak2021well,debussche2011local,debussche2012global,glatt2008stochastic,glatt2011pathwise,saal2021stochastic}. Here a maximal pathwise solution $\left(u, \{\eta_n\}_{n\in\mathbb N}, \xi\right)$ is called global if $\xi = \infty$ $\Pb$-almost surely.
The procedure goes as follows. 
\begin{enumerate}
    \item Consider a modified version of the viscous stochastic PEs by adding a cut-off function before the nonlinear terms.
    \item Employ the Galerkin approximation to the modified system and establish the energy estimates.
    \item Use the compactness theorem (Lemma \ref{lemma:aubin-lions}) to derive the existence of martingale solutions to the modified system, and then apply a suitable stopping time to obtain the existence of martingale solutions to the original viscous stochastic PEs.
    \item Prove the pathwise uniqueness and apply the Gy\"ongy-Krylov criterion to prove the existence of the pathwise  solution to the original system.
    \item Finally, prove it is indeed a global pathwise solution.
\end{enumerate}

The following proposition concerns the existence and uniqueness of pathwise solutions to the modified viscous system \eqref{PE-modified-system-approximation}.

\begin{proposition}\label{prop:viscous-PEs}
 Suppose that $\sigma$ satisfies \eqref{noise-ine}, $p\geq 2$, and $0 <\kappa < \frac12$. Let $\mathcal S = \left(\Omega, \Fc, \Fb, \Pb \right)$ be a stochastic basis, and suppose that $u_0 \in L^p(\Omega; \mathcal D_{s,2\kappa})$ is an $\Fc_0$-measurable random variable. Then, for each $n \in \Nb$ fixed,  there exists a unique global pathwise solution to the modified system \eqref{PE-modified-system-approximation}. Moreover, for any given time $T>0$, we have the following estimates
 \begin{equation}\label{eq:estimate-viscous}
    \mathbb E \sup\limits_{t'\in[0,T]} \|u_n\|_{s}^p  \leq C_{n,p,\rho}\Big(1+ \mathbb E \|u_0\|_{s}^p\Big) e^{C_{n,p,\rho} T},
\end{equation}
where $C_{n,p,\rho}\to\infty$ as $n\to \infty.$
\end{proposition}

\begin{proof}
The existence of the pathwise solution follows from \cite{saal2021stochastic}. Although our modified system \eqref{PE-modified-system-approximation} is slightly different from \cite{saal2021stochastic} in terms of the cut-off functions and our initial condition is more regular, the procedure (i)--(v) mentioned above is still valid. As there is no vertical viscosity, the boundary conditions for system \eqref{PE-modified-system-approximation} coincide with those of the original system.
For brevity, we omit the details and only focus on the estimate \eqref{eq:estimate-viscous} and on showing that $u_n$ satisfies the convexity condition. These proofs will be useful when we prove Proposition \ref{proposition:estimate-inviscid} later. 

\medskip

\noindent\underline{\textbf{Estimate of $\|u_n\|$:}} By It\^o's formula, from system \eqref{PE-modified-system-approximation} we get
\begin{equation}\label{est:viscous-0}
    \begin{split}
        d \|u_n\|^p  = &-\frac pn \theta_\rho(\|u_n\|_{s-1}) \theta_\kappa(\|\partial_z v_n - \p_z v_0\|_{L^\infty}) \|\partial_x u_n\|^2 \|u_n\|^{p-2} dt 
        \\
   & + \theta_\rho^2(\|u_n\|_{s-1}) \theta_\kappa^2(\|\partial_z v_n - \p_z v_0\|_{L^\infty})\frac p2 \|\sigma(u_n)\|^2_{L_2(\mathscr U, L^2)} \|u_n\|^{p-2} dt 
   \\
   &+ \theta_\rho^2(\|u_n\|_{s-1})\theta_\kappa^2(\|\partial_z v_n - \p_z v_0\|_{L^\infty})\frac{p(p-2)}{2} \left\langle  \sigma(u_n), u_n \right\rangle^2 \|u_n\|^{p-4} dt 
   \\
   &+ \theta_\rho(\|u_n\|_{s-1})\theta_\kappa(\|\partial_z v_n - \p_z v_0\|_{L^\infty}) p\left\langle   \sigma(u_n), u_n \right\rangle \|u_n\|^{p-2} dW,
    \end{split}
\end{equation}
where the nonlinear term vanishes due to integration by parts.
By the Cauchy-Schwarz inequality, Young's inequality, and Assumption~\eqref{noise-ine}, we have
\begin{equation}\label{est:viscous-1}
    \begin{split}
        &\Big|\theta_\rho^2(\|u_n\|_{s-1}) \theta_\kappa^2(\|\partial_z v_n - \p_z v_0\|_{L^\infty})\frac p2 \|\sigma(u_n)\|^2_{L_2(\mathscr U, L^2)} \|u_n\|^{p-2}  \Big|
   \\
   &+ \Big|\theta_\rho^2(\|u_n\|_{s-1})\theta_\kappa^2(\|\partial_z v_n - \p_z v_0\|_{L^\infty})\frac{p(p-2)}{2} \left\langle  \sigma(u_n), u_n \right\rangle^2 \|u_n\|^{p-4}\Big|
   \\
   \leq & C_{p} (1+ \|u_n\|^p).
    \end{split}
\end{equation}
Using the Burkholder-Davis-Gundy inequality, the Cauchy-Schwartz inequality, Young's inequality, and the property of $\sigma$ in \eqref{noise-ine} gives
\begin{equation}\label{est:viscous-2}
\begin{split}
    &\mathbb E \sup\limits_{t'\in[0,t]} \left|\int_0^{t'} \theta_\rho(\|u_n\|_{s-1})\theta_\kappa(\|\partial_z v_n - \p_z v_0\|_{L^\infty})  p\left\langle   \sigma(u_n), u_n \right\rangle \|u_n\|^{p-2} dW\right| 
    \\
    \leq &C_{p} \mathbb E\Big(\int_0^t \| u_n\|^{2(p-1)}  (1+\|u_n\|)^2 dt'\Big)^{\frac12}
    \\
    \leq &\frac12 \mathbb E \sup\limits_{t'\in[0,t]} \|u_n\|^p 
    + C_{p} \mathbb E \int_0^t (1+ \|u_n\|^p ) dt'.
\end{split}
\end{equation}

\medskip

\noindent\underline{\textbf{Estimates of $\|D^\alpha v_n\|$ for $0\leq |\alpha|\leq s$:}} For $0\leq |\alpha|\leq s$, applying It\^o's formula to system \eqref{PE-vorticity-modified-system-approximation} brings
\begin{equation*}
\begin{split}
   d \|D^\alpha v_n\|^p = & -\frac pn \theta_\rho(\|u_n\|_{s-1})\theta_\kappa(\|\partial_z v_n - \p_z v_0\|_{L^\infty}) \|\partial_{x} D^\alpha v_n\|^{2} \|D^\alpha v_n\|^{p-2}dt
   \\
   &-p \theta_\rho(\|u_n\|_{s-1})\theta_\kappa(\|\partial_z v_n - \p_z v_0\|_{L^\infty}) \|D^\alpha v_n\|^{p-2} \left\langle D^\alpha(u_n \p_x v_n + w_n\partial_z v_n), D^\alpha v_n \right\rangle dt
   \\
   & + \theta_\rho^2(\|u_n\|_{s-1})\theta_\kappa^2(\|\partial_z v_n - \p_z v_0\|_{L^\infty})\frac p2 \|D^\alpha\partial_z  \sigma(u_n)\|^2_{L_2(\mathscr U, L^2)} \|D^\alpha v_n\|^{p-2} dt 
   \\
   &+ \theta_\rho^2(\|u_n\|_{s-1})\theta_\kappa^2(\|\partial_z v_n - \p_z v_0\|_{L^\infty})\frac{p(p-2)}{2} \left\langle D^\alpha\partial_z  \sigma(u_n), D^\alpha v_n \right\rangle^2 \|D^\alpha v_n\|^{p-4} dt 
   \\
   &+ \theta_\rho(\|u_n\|_{s-1})\theta_\kappa(\|\partial_z v_n - \p_z v_0\|_{L^\infty}) p\left\langle D^\alpha\partial_z  \sigma(u_n), D^\alpha v_n \right\rangle \|D^\alpha v_n\|^{p-2} dW.
\end{split}
\end{equation*}

Notice that the first term on the right-hand side is the dissipation term. For the nonlinear terms, since $\left\langle u_n\partial_x D^\alpha v_n + w_n\partial_z D^\alpha v_n, D^\alpha v_n \right\rangle =0$, we can write
\begin{equation*}
    \begin{split}
        &\left\langle D^\alpha(u_n \p_x v_n + w_n\partial_z v_n), D^\alpha v_n \right\rangle
        \\
        =& 
        \langle D^\alpha(u_n \p_x v_n + w_n\partial_z v_n) - (u_n\partial_x D^\alpha v_n + w_n\partial_z D^\alpha v_n) - (D^\alpha u_n\partial_x  v_n + D^\alpha w_n\partial_z  v_n), D^\alpha v_n \rangle
        \\
        &+ \langle  D^\alpha u_n\partial_x  v_n + D^\alpha w_n\partial_z  v_n, D^\alpha v_n\rangle =: I_1 + I_2.
    \end{split}
\end{equation*}
Let us consider the highest order terms, namely, when $|\alpha|=s$. The lower-order terms can be handled readily. For $I_1$, thanks to \eqref{ine:controls} and Lemma \ref{lemma:nonlinear-ine}, we have
\begin{equation*}
    \begin{split}
       |I_1| \leq &\ C\Big( (\|Du_n\|_{L^\infty} + \|Dw_n\|_{L^\infty}) \|D^{s-1} v_n\| + \|D v_n\|_{L^\infty}(\|D^{s-1}u_n\| + \|D^{s-1} w_n\|) \Big) \|v_n\|_{H^{s}}
       \\
       \leq & C \|u_n\|_{W^{2,\infty}} \|v_n\|_{H^{s}}^2.
    \end{split}
    \end{equation*}
For $I_2$, when $D^\alpha \neq \p_x^s$ we have
\begin{equation*}
    \begin{split}
        \Big|\langle  D^\alpha u_n\partial_x  v_n + D^\alpha w_n\partial_z  v_n, D^\alpha v_n\rangle \Big| &\leq C(\|D^\alpha u_n\| \| \partial_x v_n\|_{L^\infty} + \|D^\alpha w_n\| \|\partial_z v_n\|_{L^\infty}) \|v_n\|_{H^{s}} 
        \\
        &\leq C \|u_n\|_{W^{2,\infty}} \|v_n\|_{H^{s}}^2,
    \end{split}
    \end{equation*}
and when $D^\alpha = \partial_x^s$, we have
\begin{equation*}
    \begin{split}
        \left|\langle \p_x^s u_n \p_x v_n  + \p_x^s w_n\partial_z v_n, \p_x^s v_n\rangle \right| &\leq C(\|\p_x^s u_n\| \|\p_x v_n\|_{L^\infty} + \|\p_x^s w_n\| \|\p_z v_n\|_{L^\infty})\|\p_x^s v_n\|
        \\
        &\leq C(\|v_n\|_{H^s}\|u_n\|_{W^{2,\infty}} + \|\p_x^{s+1} v_n\|\|u_n\|_{W^{2,\infty}})\|v_n\|_{H^s}
        \\
        &\leq Cn (1+\|u_n\|_{W^{2,\infty}}^2) \|v_n\|_{H^s}^2 + \frac1{2n} \|\p_x^{s+1} v_n\|^2.
    \end{split}
\end{equation*}
Combining the estimates of $I_1$ and $I_2$ and summing over $|\alpha| \leq s$ gives
\begin{equation}\label{est:viscous-3}
    \begin{split}
       &\sum\limits_{|\alpha| \leq s}\Big| p \theta_\rho(\|u_n\|_{s-1})\theta_\kappa(\|\partial_z v_n - \p_z v_0\|_{L^\infty}) \|D^\alpha v_n\|^{p-2} \left\langle D^\alpha(u_n \p_x v_n + w_n\partial_z v_n), D^\alpha v_n \right\rangle \Big|
        \\
        \leq &  \theta_\rho(\|u_n\|_{s-1})\theta_\kappa(\|\partial_z v_n - \p_z v_0\|_{L^\infty})\left(C_{p}n (1+\|u_n\|_{W^{2,\infty}}^2) \|v_n\|_{H^{s}}^p + \frac p{2n} \|\p_x^s v_n\|^{p-2} \|\p_x^{s+1} v_n\|^2 \right)
        \\
        \leq & C_{p,\rho}n \|v_n\|_{H^{s}}^p + \frac p{2n}  \theta_\rho(\|u_n\|_{s-1})\theta_\kappa(\|\partial_z v_n - \p_z v_0\|_{L^\infty})\|\p_x^s v_n\|^{p-2} \|\p_x^{s+1} v_n\|^2 .
    \end{split}
\end{equation}
Next, by Cauchy-Schwarz inequality, Young's inequality, and Assumption~\eqref{noise-ine}, we have for any $|\alpha|\leq s,$
\begin{equation}\label{est:viscous-4}
    \begin{split}
        &\Big|\theta_\rho^2(\|u_n\|_{s-1})\theta_\kappa^2(\|\partial_z v_n - \p_z v_0\|_{L^\infty})\frac p2 \|D^\alpha\partial_z  \sigma(u_n)\|^2_{L_2(\mathscr U, L^2)} \|D^\alpha v_n\|^{p-2} \Big|
   \\
   &+ \Big|\theta_\rho^2(\|u_n\|_{s-1})\theta_\kappa^2(\|\partial_z v_n - \p_z v_0\|_{L^\infty})\frac{p(p-2)}{2} \left\langle D^\alpha\partial_z  \sigma(u_n), D^\alpha v_n \right\rangle^2 \|D^\alpha v_n\|^{p-4}\Big|  
   \\
   \leq & C_{p} (1+ \|u_n\|^p + \|v_n\|_{H^s}^p).
    \end{split}
\end{equation}
Finally, using the Burkholder-Davis-Gundy inequality, the Cauchy-Schwartz inequality, Young's inequality, and the property of $\sigma$ in \eqref{noise-ine} produces
\begin{equation}\label{est:viscous-5}
\begin{split}
    &\mathbb E \sup\limits_{t'\in[0,t]} \left|\int_0^{t'} \theta_\rho(\|u_n\|_{s-1})\theta_\kappa(\|\partial_z v_n - \p_z v_0\|_{L^\infty}) p\left\langle D^\alpha\partial_z  \sigma(u_n), D^\alpha v_n \right\rangle \|D^\alpha v_n\|^{p-2} dW \right| 
    \\
    \leq &C_{p} \mathbb E\Big(\int_0^t  \|D^\alpha v_n\|^{2(p-1)}  (1+\|u_n\| + \|v_n\|_{\widetilde H^s})^2 dt'\Big)^{\frac12}
    \\
    \leq &\frac12 \mathbb E \sup\limits_{t'\in[0,t]}  \|D^\alpha v_n\|^p 
    + C_{p} \mathbb E \int_0^t (1+ \|u_n\|^p + \|v_n\|_{ H^s}^p ) dt'.
\end{split}
\end{equation}

\medskip

\noindent\underline{\textbf{Combining the estimates:}} Combining the estimates \eqref{est:viscous-1}--\eqref{est:viscous-5} yields
\begin{equation}
    \begin{split}
         \mathbb E \sup\limits_{t'\in[0,t]} \|u_n\|_{s}^p
        \leq &C_p \Big( \mathbb E \sup\limits_{t'\in[0,t]} \|u_n\|^p + \sum\limits_{|\alpha|\leq s}\mathbb E \sup\limits_{t'\in[0,t]} \|D^\alpha v_n\|^p \Big) 
        \\
        \leq & C_{p,\rho}\mathbb E \Big[ 1 +  \|u_0\|^p+\|v_0\|_{H^s}^p + \int_0^t \big(1 +  \|u_n\|^p+n\|v_n\|_{ H^s}^p   \big) dt' \Big]
         \\
        \leq & C_{p,\rho}\mathbb E \Big[ 1 +  \|u_0\|_{\tilde s}^p + \int_0^t \big(1 +  n\|u_n\|_{s}^p   \big) dt' \Big].
    \end{split}
\end{equation}
Now, by the Gr\"onwall inequality \cite[Lemma 5.3]{glatt2009strong}, for $p\geq 2$ and any $T>0$, one has
\begin{equation}
    \mathbb E \sup\limits_{t'\in[0,T]} \|u_n\|_{s}^p  \leq C_{n,p,\rho}\Big(1+ \mathbb E \|u_0\|_{s}^p\Big) e^{C_{n,p,\rho} T}.
\end{equation}

\medskip

\noindent\underline{\textbf{Convexity condition:}} Finally, we need to verify that $\kappa\leq\p_z v_n \leq \frac1\kappa$ for all $(t,x,z) \in [0,T]\times \mathbb D$ and a.s., and thus $u_n$ takes value in $\mathcal D_{s,\kappa}.$  First, notice that $\partial_z v_n$ satisfies
\begin{equation}\label{eqn:pzv}
\begin{split}
      &d (\p_z v_n - \p_z v_0) + \theta_\rho(\|u_n\|_{s-1})\theta_\kappa(\|\partial_z v_n - \p_z v_0\|_{L^\infty})
      \\
      &\hspace{1.2in}\times(\p_z u_n \p_x v_n + u_n \p_{xz} v_n + \p_z w_n \p_z v_n + w_n\p_{zz} v_n - \frac1n \p_z \partial_{xx} v_n)dt
    \\
    = &\theta_\rho(\|u_n\|_{s-1})\theta_\kappa(\|\partial_z v_n - \p_z v_0\|_{L^\infty})\partial_{zz} \sigma (u_n) dW.
\end{split}
\end{equation}
By continuity in time of the global pathwise solution, for almost every  $\omega\in\Omega$, either there exists a time $t_1 = t_1(\omega)\leq T$ such that $\|\partial_z v_n - \p_z v_0\|_{L^\infty} = \kappa$ at $t_1$ and $\|\partial_z v_n - \p_z v_0\|_{L^\infty} < \kappa$ for $t\in[0,t_1)$, or such $t_1$ does not exist. In the first scenario, by the definition of $\theta_\kappa$ we know $\theta_\kappa(\|\partial_z v_n - \p_z v_0\|_{L^\infty}) = 0$ at $t_1$, and therefore $d (\p_z v_n - \p_z v_0) =0$. This implies that for $t\in[t_1,T]$ one has $\|\partial_z v_n - \p_z v_0\|_{L^\infty} = \kappa$. Thus 
\begin{align}\label{convexity-justify}
    \|\partial_z v_n - \p_z v_0\|_{L^\infty}\leq \kappa, \quad \forall t \in [0,T].
\end{align}
In the second scenario, \eqref{convexity-justify} automatically holds. Combining both yields that \eqref{convexity-justify} holds $\Pb$-a.s.. Since $0<\kappa<\frac12$ and $2\kappa \leq \p_z v_0 \leq \frac1{2\kappa}$ for all $(x,z)\in \mathbb D$ and a.s., and thanks to the regularity of $\partial_z v_n$, we have $\kappa\leq\p_z v_n \leq \frac1\kappa$ for all $(t,x,z) \in [0,T]\times \mathbb D$ and a.s.. 
\end{proof}

\begin{remark}
Notice that the estimate \eqref{eq:estimate-viscous} is not uniform in $n$, in particular, the right-hand side will diverge as $n\to \infty.$ Therefore, we need a more delicate estimate that is uniform in $n$ in order to show the convergence of $u_n$, which we present in the next subsection.
\end{remark}

\subsection{Uniform estimates}\label{section:uniform-estimate}
In this section, we establish the estimates that are uniform in $n$ for the approximation system~\eqref{PE-modified-system-approximation}. For this purpose, we need to employ the Rayleigh condition.

\begin{proposition}\label{proposition:estimate-inviscid}
Let $p\geq 2$ and $0<\kappa<\frac12$. Suppose that for a given stochastic basis $\mathcal S = \left(\Omega, \Fc, \Fb, \Pb \right)$, $u_0 \in L^p(\Omega; \mathcal D_{s,2\kappa})$ is an $\Fc_0$-measurable random variable. Suppose that $\sigma$ satisfies \eqref{noise-ine}.
Let $u_n$ be the unique solution to system \eqref{PE-modified-system-approximation}, and denote by $v_n=\p_z u_n$. Then, for any positive time $T>0$, the followings hold:
\begin{enumerate}
	\item $
	  \sup\limits_{n\in \N}	\Big(\mathbb E  \sup\limits_{t'\in[0,T]} \|u_n\|_{\tilde s}^p  \Big) \leq  C_{p,\rho,\kappa}\Big(1+ \mathbb E \|u_0\|_{\tilde s}^p\Big) e^{C_{p,\rho,\kappa} T};
	$
	
	\item For all $(t,x,z) \in [0,T]\times \mathbb D$ and almost surely, $\kappa\leq\p_z v_n \leq \frac1\kappa$, and $u_n\in L^p(\Omega; L^\infty(0,T;\mathcal D_{s,\kappa}))$;
	
		\item For  $\alpha\in [0,1/2)$, $\int_0^{\cdot}\theta_\rho(\|u_n\|_{s-1})\theta_\kappa(\|\partial_z v_n - \p_z v_0\|_{L^\infty}) \sigma(u_n) dW$ is bounded in
		
    \noindent	$
		L^p\left( \Omega; W^{\alpha, p}(0, T; \mathcal D_{s-1}) \right);	
	$
	\item Moreover, if $p\geq 4$, then $ u_n - \int_0^\cdot \theta_\rho(\|u_n\|_{s-1})\theta_\kappa(\|\partial_z v_n - \p_z v_0\|_{L^\infty}) \sigma(u_n) dW$ is bounded in \\
	$
		L^2\left( \Omega; W^{1, 2}(0, T; \mathcal D_{s-2}))\right).
	$
\end{enumerate}
\end{proposition}
\begin{proof}
In the following estimates, we do not take advantage of the viscosity as the coefficient of viscosity vanishes to zero as $n\to \infty.$

\medskip

\noindent\underline{\textbf{Estimate of $\|u_n\|$:}} The argument repeats the same as in \eqref{est:viscous-0}--\eqref{est:viscous-2}.

\medskip

\noindent\underline{\textbf{Estimates of $\|D^\alpha v_n\|$ for $D^\alpha\neq \p_x^s$:}} 
First, by the definition of $\theta_\kappa$ and $2\kappa \leq \p_z v_0 \leq \frac1{2\kappa}$ for all $(x,z)\in \mathbb D$ and almost surely, we have
\begin{equation}\label{cutoff-kappa}
    \begin{split}
        &\theta_\kappa(\|\partial_z v_n - \p_z v_0\|_{L^\infty}) \|v_n\|^p_{H^s} = \theta_\kappa(\|\partial_z v_n - \p_z v_0\|_{L^\infty})  \|v_n\|^p_{H^s} \frac{(1+\lnorm\sqrt{\p_z v_n}\rnorm_{L^\infty})^p}{(1+\lnorm\sqrt{\p_z v_n}\rnorm_{L^\infty})^p}
        \\
        \leq & C_p \|v_n\|_{\widetilde H^{s}}^p \theta_\kappa(\|\partial_z v_n - \p_z v_0\|_{L^\infty})\left(1+\lnorm\sqrt{\p_z v_n}\rnorm_{L^\infty}^p\right) 
        \\
        \leq &C_p \|v_n\|_{\widetilde H^{s}}^p \left(1+\theta_\kappa(\|\partial_z v_n - \p_z v_0\|_{L^\infty})\lnorm\p_z v_n\rnorm_{L^\infty}^{\frac p2}\right)
        \\
        \leq &C_p \|v_n\|_{\widetilde H^{s}}^p \left(1+ \| \p_z v_0\|^{\frac p2}_{L^\infty} +\theta_\kappa(\|\partial_z v_n - \p_z v_0\|_{L^\infty})\lnorm\p_z v_n - \p_z v_0\rnorm_{L^\infty}^{\frac p2}\right)
        \leq C_{p,\kappa} \|v_n\|_{\widetilde H^{s}}^p.
    \end{split}
\end{equation}
Now repeat the computations as in Proposition~\ref{prop:viscous-PEs}, except for $D^\alpha = \p_x^s$. The nonlinear estimate \eqref{est:viscous-3} now becomes
\begin{equation}\label{est:inviscid-1}
    \begin{split}
        &\Big| p \theta_\rho(\|u_n\|_{s-1})\theta_\kappa(\|\partial_z v_n - \p_z v_0\|_{L^\infty}) \|D^\alpha v_n\|^{p-2} \left\langle D^\alpha(u_n \p_x v_n + w_n\partial_z v_n), D^\alpha v_n \right\rangle \Big|
        \\
        \leq & p \theta_\rho(\|u_n\|_{s-1})\theta_\kappa(\|\partial_z v_n - \p_z v_0\|_{L^\infty}) \|u_n\|_{W^{2,\infty}} \|v_n\|_{H^{s}}^p
        \\
        \leq & C_{p,\rho} \theta_\kappa(\|\partial_z v_n - \p_z v_0\|_{L^\infty})\|v_n\|_{H^{s}}^p
        \leq  C_{p,\rho,\kappa} \|v_n\|_{\widetilde H^{s}}^p.
    \end{split}
\end{equation}
Similar to \eqref{est:viscous-4}, one has
\begin{equation}\label{est:inviscid-2}
    \begin{split}
        &\Big|\theta^2_\rho(\|u_n\|_{s-1})\theta^2_\kappa(\|\partial_z v_n - \p_z v_0\|_{L^\infty})\frac p2 \|D^\alpha\partial_z  \sigma(u_n)\|^2_{L_2(\mathscr U, L^2)} \|D^\alpha v_n\|^{p-2} \Big|
   \\
   &+ \Big|\theta^2_\rho(\|u_n\|_{s-1})\theta^2_\kappa(\|\partial_z v_n - \p_z v_0\|_{L^\infty})\frac{p(p-2)}{2} \left\langle D^\alpha\partial_z  \sigma(u_n), D^\alpha v_n \right\rangle^2 \|D^\alpha v_n\|^{p-4}\Big|  
   \\
   \leq & C_{p,\kappa} (1+ \|u_n\|^p + \|v_n\|_{\widetilde H^s}^p).
    \end{split}
\end{equation}
The estimate analogue to \eqref{est:viscous-5} reads
\begin{equation}\label{est:inviscid-3}
\begin{split}
    &  \mathbb E \sup\limits_{t'\in[0,t]} \left|\int_0^{t'} \theta_\rho(\|u_n\|_{s-1})\theta_\kappa(\|\partial_z v_n - \p_z v_0\|_{L^\infty}) p\left\langle D^\alpha\partial_z  \sigma(u_n), D^\alpha v_n \right\rangle \|D^\alpha v_n\|^{p-2} dW \right| 
    \\
    \leq &C_{p,\kappa} \mathbb E\Big(\int_0^t  \|D^\alpha v_n\|^{2(p-1)}  (1+\|u_n\| + \|v_n\|_{\widetilde H^s})^2 d t'\Big)^{\frac12}
    \\
    \leq &\frac12 \mathbb E \sup\limits_{t'\in[0,t]}  \|D^\alpha v_n\|^p 
    + C_{p,\kappa} \mathbb E \int_0^t (1+ \|u_n\|^p + \|v_n\|_{\widetilde H^s}^p ) d t'.
\end{split}
\end{equation}

\medskip

\noindent\underline{\textbf{Estimate of $\lnorm\frac{\p_x^s v_n}{\sqrt{\p_z v_n}}\rnorm$:}} When $\alpha = (s,0),$ i.e., $D^\alpha = \partial_x^s$, we need to apply the Rayleigh condition in order to avoid the dependence of $n$ through viscosity.
Applying $\partial_x^s$ to \eqref{PE-vorticity-modified-system-approximation-1} yields
\begin{equation}\label{PE-modified-1-derivative}
\begin{split}
    d \partial_x^s v_n = &- \theta_\rho(\|u_n\|_{s-1})\theta_\kappa(\|\partial_z v_n - \p_z v_0\|_{L^\infty})\Big[u_n\partial_x^{s+1}v_n+w_n\partial_x^s \partial_z v_n+ \partial_x^s w_n \partial_z v_n \\
    & \qquad\qquad\qquad\qquad\qquad+ \sum\limits_{k=0}^{s-1} \binom{s}{k} \partial_x^{s-k} u_n \partial_x^{k+1} v_n + \sum\limits_{k=1}^{s-1} \binom{s}{k} \partial_x^{s-k} w_n \partial_x^{k} \partial_z v_n -\frac 1n \p_x^s \partial_{xx} v_n \Big]dt
    \\
    &+ \theta_\rho(\|u_n\|_{s-1})\theta_\kappa(\|\partial_z v_n - \p_z v_0\|_{L^\infty})\partial_x^s\partial_z \sigma (u_n) dW =: A_1 dt + A_2 dW,
\end{split}
\end{equation}
and applying $\p_z$ to \eqref{PE-vorticity-modified-system-approximation-1} gives
\begin{equation}
\begin{split}
      d \p_z v_n = &- \theta_\rho(\|u_n\|_{s-1})\theta_\kappa(\|\partial_z v_n - \p_z v_0\|_{L^\infty})(\p_z u_n \p_x v_n + u_n \p_{xz} v_n + \p_z w_n \p_z v_n + w_n\p_{zz} v_n - \frac1n \p_z \partial_{xx} v_n)dt
    \\
    &+ \theta_\rho(\|u_n\|_{s-1})\theta_\kappa(\|\partial_z v_n - \p_z v_0\|_{L^\infty})\partial_{zz} \sigma (u_n) dW =: B_1 dt + B_2 dW.
\end{split}
\end{equation}
Then It\^o's formula and integration by parts imply that 
\begin{equation}\label{ito-critical-term}
    \begin{split}
        &d\lnorm\frac{\p_x^s v_n}{\sqrt{\p_z v_n}}\rnorm^p = \frac p2\lnorm\frac{\p_x^s v_n}{\sqrt{\p_z v_n}}\rnorm^{p-2}\left(\left\langle 2A_1, \frac{\p_x^{s} v_n}{\p_z v_n} \right\rangle - \left\langle B_1,\frac{|\p_x^{s} v|^2}{|\p_z v_n|^2} \right\rangle\right)dt
        \\
        &\qquad +\frac{p(p-2)}{8} \lnorm\frac{\p_x^s v_n}{\sqrt{\p_z v_n}}\rnorm^{p-4} \left( \left\langle A_2,2\frac{\p_x^{s} v_n}{\p_z v_n} \right\rangle - \left\langle B_2, \frac{|\p_x^{s} v|^2}{|\p_z v_n|^2}\right\rangle \right)^2 dt
        \\
        &\qquad +\frac p4\lnorm\frac{\p_x^s v_n}{\sqrt{\p_z v_n}}\rnorm^{p-2} \left(\left\langle A_2^2,\frac2{\p_z v_n} \right\rangle - \left\langle 4\frac{\p_x^{s}v}{|\p_z v_n|^2} ,A_2B_2 \right\rangle +  \left\langle 2\frac{|\p_x^{s}v|^2}{|\p_z v_n|^3} ,B_2^2 \right\rangle \right) dt
        \\
        &\qquad + \frac p2\lnorm\frac{\p_x^s v_n}{\sqrt{\p_z v_n}}\rnorm^{p-2}\left(\left\langle 2\frac{\p_x^{s}v}{\p_z v_n}, A_2 \right\rangle - \left\langle \frac{|\p_x^{s}v|^2}{|\p_z v_n|^2}, B_2 \right\rangle  \right)dW
        \\
        =&- p\theta_\rho(\|u_n\|_{s-1})\theta_\kappa(\|\partial_z v_n - \p_z v_0\|_{L^\infty})\lnorm\frac{\p_x^s v_n}{\sqrt{\p_z v_n}}\rnorm^{p-2}\Bigg[\int \frac{1}{2\partial_z v_n} \big[u_n\partial_x + w_n\partial_z \big] |\partial_x^s v_n|^2   dxdz
        \\
        & \qquad  + \int \p_x^s w_n \p_x^s v_n dxdz + \int  \sum\limits_{k=0}^{s-1} \binom{s}{k} \frac{\partial_x^{s-k} u_n \partial_x^{k+1}  v_n \partial_x^s v_n}{\partial_z v_n}  dxdz  + \int  \sum\limits_{k=1}^{s-1} \binom{s}{k} \frac{\partial_x^{s-k} w_n \partial_x^{k} \partial_z v_n \partial_x^s v_n}{\partial_z v_n} dxdz 
        \\
        &\qquad + \frac 1n \lnorm\frac{ \p_x^{s+1} v_n}{\sqrt{\p_z v_n}}\rnorm^{2} - \frac1{2n} \int \left|\p_x^s v_n \right|^2 \partial_{xx} \left(\frac1{\p_z v_n}\right) dxdz
        \\
        &\qquad - \frac12 \int \frac{|\p_x^s v_n|^2}{|\p_z v_n|^2} \big( \p_z u_n \p_x v_n + u_n \p_{xz} v_n + \p_z w_n \p_z v_n + w_n\p_{zz} v_n - \frac1n \p_z \partial_{xx} v_n\big)  dxdz \Bigg] dt
        \\
        &+ \frac p8(p-2)\theta^2_\rho(\|u_n\|_{s-1})\theta^2_\kappa(\|\partial_z v_n - \p_z v_0\|_{L^\infty}) \lnorm\frac{\p_x^s v_n}{\sqrt{\p_z v_n}}\rnorm^{p-4}
        \\
        &\qquad \qquad \times\left( 2\left\langle \frac{\p_x^s v_n}{\p_z v_n}, \p_x^s\p_z \sigma(u_n) \right\rangle - \left\langle \left(\frac{\p_x^s v_n}{\p_z v_n}\right)^2, \p_{zz} \sigma(u_n) \right\rangle \right)^2 dt
        \\
        & + \frac p2 \theta^2_\rho(\|u_n\|_{s-1})\theta^2_\kappa(\|\partial_z v_n - \p_z v_0\|_{L^\infty}) \lnorm\frac{\p_x^s v_n}{\sqrt{\p_z v_n}}\rnorm^{p-2} 
        \\
        &\qquad \qquad \times\Big(\left\langle \frac 1{\p_z v_n} \p_x^s\p_z \sigma(u_n), \p_x^s\p_z \sigma(u_n) \right\rangle + \left\langle \frac {(\p_x^s v_n)^2}{(\p_z v_n)^3} \p_{zz} \sigma(u_n), \p_{zz} \sigma(u_n) \right\rangle 
        \\
        &\hspace{2in}
 -2\left\langle \frac {\p_x^s v_n}{(\p_z v_n)^2} \p_x^s\p_z \sigma(u_n), \p_{zz} \sigma(u_n) \right\rangle \Big)dt
        \\
        & + \frac p2\theta_\rho(\|u_n\|_{s-1})\theta_\kappa(\|\partial_z v_n - \p_z v_0\|_{L^\infty}) \lnorm\frac{\p_x^s v_n}{\sqrt{\p_z v_n}}\rnorm^{p-2}
        \\
        &\qquad\qquad\times\left( 2\left\langle \frac {\p_x^s v_n}{\p_z v_n} , \p_x^s\p_z \sigma(u_n) \right\rangle - \left\langle \left(\frac{\p_x^s v_n}{\p_z v_n}\right)^2 , \p_{zz} \sigma(u_n) \right\rangle \right) dW
        \\
        =: &I_1 dt + I_2 dt + I_3 dt + I_4 dW.
    \end{split}
\end{equation}

For the nonlinear terms in $I_1$, first by integration by parts and thanks to \eqref{PE-inviscid-3}, we have
\begin{equation*}
\begin{split}
      &\int \frac{1}{2\partial_z v_n} \big[u_n\partial_x + w_n\partial_z \big] |\partial_x^s v_n|^2   dxdz - \frac12 \int \frac{|\p_x^s v_n|^2}{|\p_z v_n|^2} \big( \p_z u_n \p_x v_n + u_n \p_{xz} v_n + \p_z w_n \p_z v_n + w\p_{zz} v_n 
      \\
      &\hspace{4in}- \frac1n \p_z \partial_{xx} v_n\big)  dxdz 
      \\
      = & - \frac12 \int \frac{|\p_x^s v_n|^2}{|\p_z v_n|^2} \big( \p_z u_n \p_x v_n + \p_z w_n \p_z v_n- \frac1n \p_z \partial_{xx} v_n \big)  dxdz 
      \\
      = &- \frac12 \int \frac{|\p_x^s v_n|^2}{|\p_z v_n|^2} \big( v_n \p_x v_n - \p_x u_n \p_z v_n - \frac1n \p_z \partial_{xx} v_n\big) dxdz.
\end{split}
\end{equation*}
Then with the Poincar\'e inequality \eqref{poincare}, the Sobolev inequality,  and the H\"older inequality, one has
\begin{equation}\label{inviscid-critical-est-nonlinear-1}
    \begin{split}
        &\left| p\theta_\rho(\|u_n\|_{s-1})\theta_\kappa(\|\partial_z v_n - \p_z v_0\|_{L^\infty})\lnorm\frac{\p_x^s v_n}{\sqrt{\p_z v_n}}\rnorm^{p-2} \frac12 \int \frac{|\p_x^s v_n|^2}{|\p_z v_n|^2} \big( v_n \p_x v_n - \p_x u_n \p_z v_n- \frac1n \p_z \partial_{xx} v_n \big) dxdz\right| 
        \\
        \leq & C_{p,\kappa} \theta_\rho(\|u_n\|_{s-1})\lnorm\frac{\p_x^s v_n}{\sqrt{\p_z v_n}}\rnorm^{p} (1+ \|v_n\|^2_{W^{3,\infty}})
        \\
        \leq & C_{p,\rho,\kappa} \lnorm\frac{\p_x^s v_n}{\sqrt{\p_z v_n}}\rnorm^{p}\leq C_{p,\rho,\kappa} \|v_n\|_{\widetilde H^s}^p,
    \end{split}
\end{equation}
where we have used the fact that $s\geq 6.$ Next, one can compute that
\begin{equation*}
    \partial_{xx}(\frac1{\partial_z v_n}) = - \frac{\partial_z \partial_{xx} v_n}{(\partial_z v_n)^2} + \frac{2|\partial_z \partial_{x} v_n|^2}{(\partial_z v_n)^3}.
\end{equation*}
Since $s\geq 6$, by the Sobolev inequality and the H\"older inequality we have
\begin{equation}\label{inviscid-critical-est-nonlinear-2}
\begin{split}
    &\left|p\theta_\rho(\|u_n\|_{s-1})\theta_\kappa(\|\partial_z v_n - \p_z v_0\|_{L^\infty})\lnorm\frac{\p_x^s v_n}{\sqrt{\p_z v_n}}\rnorm^{p-2} \frac1{2n} \int \left|\p_x^s v_n \right|^2 \partial_{xx} \left(\frac1{\p_z v_n}\right) dxdz \right| 
    \\
    \leq & C_{p,\kappa} \theta_\rho(\|u_n\|_{s-1})\lnorm\frac{\p_x^s v_n}{\sqrt{\p_z v_n}}\rnorm^{p} (\|v_n\|_{W^{3,\infty}} + \|v_n\|_{W^{2,\infty}}^2)
    \\
    \leq & C_{p,\rho,\kappa} \lnorm\frac{\p_x^s v_n}{\sqrt{\p_z v_n}}\rnorm^{p}\leq C_{p,\rho,\kappa} \|v_n\|_{\widetilde H^s}^p.
\end{split}
\end{equation}
Note that the constant is now independent of $n$.
For other nonlinear terms in $I_1$, one can repeat similar arguments as in  \textbf{Estimates of $\|D^\alpha v_n\|$ for $D^\alpha\neq \p_x^s$}, and use Lemma \ref{lemma:nonlinear-ine} to obtain 
\begin{equation}\label{inviscid-critical-est-nonlinear-3}
    \begin{split}
        &p\theta_\rho(\|u_n\|_{s-1})\theta_\kappa(\|\partial_z v_n - \p_z v_0\|_{L^\infty})\lnorm\frac{\p_x^s v_n}{\sqrt{\p_z v_n}}\rnorm^{p-2} \Bigg| \int  \sum\limits_{k=0}^{s-1} \binom{s}{k} \frac{\partial_x^{s-k} u_n \partial_x^{k+1}  v_n \partial_x^s v_n}{\partial_z v_n}  dxdz
        \\
        &\qquad \qquad \qquad \qquad \qquad \qquad \qquad \qquad\qquad \qquad \qquad+ \int  \sum\limits_{k=1}^{s-1} \binom{s}{k} \frac{\partial_x^{s-k} w_n \partial_x^{k} \partial_z v_n \partial_x^s v_n}{\partial_z v_n} dxdz\Bigg|
        \\
        \leq & C_{p,\rho,\kappa} \lnorm\frac{\p_x^s v_n}{\sqrt{\p_z v_n}}\rnorm^{p} \leq C_{p,\rho,\kappa} \|v_n\|_{\widetilde H^s}^p.
    \end{split}
\end{equation}
Lastly, by integration by parts and thanks to the boundary condition, we have the following cancellation:
\begin{equation}\label{eqn:cancellation}
\begin{split}
    &\int \p_x^s w_n \p_x^s v_n dxdz = \int \p_x^s w_n \p_x^s \p_z u_n dxdz = -\int \p_x^s \p_z w_n \p_x^s  u_n dxdz
    \\
    =& \int \p_x^s \p_x u_n\p_x^s  u_n dxdz = \frac12 \int \p_x(\p_x^s  u_n)^2 dxdz = 0.
\end{split}
\end{equation}
Notice that such cancellation is essential since it allows us to eliminate the highest derivative term $\p_x^s w$, and thus close the estimates. We've bounded all terms in $I_1$.

Next regarding $I_2$ and $I_3$, by the inequality \eqref{ine:controls}, Assumption~\eqref{noise-ine}, the Sobolev inequality, and the H\"older inequality, one has
\begin{equation}
    |I_2|+|I_3| \leq C_{p,\rho,\kappa} \lnorm\frac{\p_x^s v_n}{\sqrt{\p_z v_n}}\rnorm^{p}\leq C_{p,\rho,\kappa} \|v_n\|_{\widetilde H^s}^p.
\end{equation}

Finally, concerning $I_4$, using the Burkholder-Davis-Gundy inequality, the Cauchy-Schwartz inequality, Young's inequality, and the property of $\sigma$ in \eqref{noise-ine}, we deduce that
\begin{equation}
\begin{split}
    &\mathbb E \sup\limits_{t'\in[0,t]} \left|\int_0^{t'} I_4 dW \right| 
    \\
    \leq &C_{p,\rho,\kappa} \mathbb E\Big(\int_0^t \lnorm\frac{\p_x^s v_n}{\sqrt{\p_z v_n}}\rnorm^{2(p-1)}  (1+\|u_n\| + \|v_n\|_{\widetilde H^s})^2 dt'\Big)^{\frac12}
    \\
    \leq &\frac12 \mathbb E \sup\limits_{t'\in[0,t]} \lnorm\frac{\p_x^s v_n}{\sqrt{\p_z v_n}}\rnorm^p 
    + C_{p,\rho,\kappa} \mathbb E \int_0^t (1+ \|u_n\|^p + \|v_n\|_{\widetilde H^s}^p ) dt'.
\end{split}
\end{equation}

\medskip

\noindent\underline{\textbf{Proof of (i):}} Based on all previous estimates, we deduce
\begin{equation}
    \begin{split}
         \mathbb E \sup\limits_{t'\in[0,t]} \|u_n\|_{\tilde s}^p
        \leq &C_p \Big( \mathbb E \sup\limits_{t'\in[0,t]} \|u_n\|^p + \sum\limits_{\substack{0\leq |\alpha| \leq s\\ D^\alpha\neq \p_x^s}}\mathbb E \sup\limits_{t'\in[0,t]} \|D^\alpha v_n\|^p + \mathbb E \sup\limits_{t'\in[0,t]} \lnorm\frac{\p_x^s v_n}{\sqrt{\p_z v_n}} \rnorm^p\Big) 
        \\
        \leq & C_{p,\rho,\kappa}\mathbb E \Big[ 1 +  \|u_0\|^p+\|v_0\|_{\widetilde H^s}^p + \int_0^t \big(1 +  \|u_n\|^p+\|v_n\|_{\widetilde H^s}^p   \big) dt' \Big]
         \\
        \leq & C_{p,\rho,\kappa}\mathbb E \Big[ 1 +  \|u_0\|_{\tilde s}^p + \int_0^t \big(1 +  \|u_n\|_{\tilde s}^p   \big) dt' \Big].
    \end{split}
\end{equation}
With the Gr\"onwall inequality \cite[Lemma 5.3]{glatt2009strong}, for $p\geq 2$ and any $T>0$, one has
\begin{equation}\label{eq:estimate-1}
    \mathbb E \sup\limits_{t'\in[0,T]} \|u_n\|_{\tilde s}^p  \leq C_{p,\rho,\kappa}\Big(1+ \mathbb E \|u_0\|_{\tilde s}^p\Big) e^{C_{p,\rho,\kappa} T}.
\end{equation}

\medskip

\noindent\underline{\textbf{Proof of (ii):}} It follows the same as in \textbf{Convexity condition} in the proof of Proposition \ref{prop:viscous-PEs}.

\medskip

\noindent\underline{\textbf{Proof of (iii):}} Using the fractional Burkholder-Davis-Gundy inequality \eqref{eq:bdg.frac}, Assumption~\eqref{noise-ine}, and the estimate \eqref{eq:estimate-1}, we have
\begin{equation*}
    \begin{split}
        &\Eb \left| \int_0^\cdot \theta_\rho(\|u_n\|_{s-1})\theta_\kappa(\|\partial_z v_n - \p_z v_0\|_{L^\infty})\sigma(u_n) \, dW \right|^p_{W^{\alpha, p}\left(0, T; \mathcal D_{s-1}\right)} \leq C_p \Eb \int_0^T \| \sigma(u_n) \|^p_{L_2\left(\Uc, \mathcal D_{s-1}\right)} \, dt 
        \\
        \leq & C_p \Eb \left[  \int_0^T 1 + \|u_n\|_{s-1}^p \, dt \right] \leq C_p \Eb \left[  \int_0^T 1 + \|u_n\|_{\tilde s}^p \, dt \right]< \infty.
    \end{split}
\end{equation*}

\medskip

\noindent\underline{\textbf{Proof of (iv):}} Notice that, for any $t\in[0,T]$,
\begin{equation}
\begin{split}
    &u_n(t) - \int_0^t \theta_\rho(\|u_n\|_{s-1})\theta_\kappa(\|\partial_z v_n - \p_z v_0\|_{L^\infty})\sigma(u_n) dW
    \\
    = &u_n(0) - \int_0^t \theta_\rho(\|u_n\|_{s-1})\theta_\kappa(\|\partial_z v_n - \p_z v_0\|_{L^\infty})(u_n\p_x u_n + w_n\p_z u_n + \p_x p_n - \frac1n \partial_{xx} u_n) d t'.
\end{split}
\end{equation}
Since $s\geq 6$, the space $\mathcal D_{s-2}$ is a Banach algebra. By \eqref{equiv-beta}, \eqref{poincare}, and \eqref{eq:estimate-1}, we have
\begin{equation}\label{estimate:w12}
    \begin{split}
        &\mathbb E \left\| u_n - \int_0^\cdot \theta_\rho(\|u_n\|_{s-1})\theta_\kappa(\|\partial_z v_n - \p_z v_0\|_{L^\infty})\sigma(u_n) dW \right\|^2_{W^{1,2}(0,T;\mathcal D_{s-2})}
        \\
        &\leq C_{T,s,\kappa} \Big(1+ \|u_0\|^2_{\tilde s} + \sup\limits_{t\in[0,T]}\|u_n(t)\|^4_{\tilde s}\Big) < \infty.
    \end{split}
\end{equation}
Note that, here the pressure gradient disappears since $\p_x p_n$ is orthogonal to the space $\mathcal D_{s-2}$, and that $p\geq 4$ is required due to applying the estimate \eqref{eq:estimate-1} to the term $\sup\limits_{t\in[0,T]}\|u_n(t)\|^4_{\tilde s}$ in \eqref{estimate:w12}. 
\end{proof}
\begin{remark}\label{rmk:cancellation}
  The essential cancellation \eqref{eqn:cancellation} fails for the Galerkin system:
\begin{equation}
\int P_n(\partial_x^s w_n \partial_z v_n) \frac{\partial_x^s v_n}{\partial_z v_n} dxdz \neq \int \partial_x^s w_n \partial_x^s v_n dxdz,
\end{equation}
where $P_n$ is the projection onto a $n$-dimensional subspace of $H$, since $P_n(\partial_x^s w_n \partial_z v_n) \neq \partial_x^s w_n \partial_z v_n$.
\end{remark}

\begin{remark}\label{s6}
 The condition $s\geq 6$ is critical in obtaining the estimates in \eqref{inviscid-critical-est-nonlinear-1} and \eqref{inviscid-critical-est-nonlinear-2}.
\end{remark}

\subsection{Martingale solutions}\label{section:martingale}
In this subsection, we establish the global existence of martingale solutions to the modified system \eqref{PE-modified-system}, and the local existence of martingale solutions to the original system \eqref{PE-inviscid-system}. The proof of the existence of martingale solutions is a standard procedure with applying the Aubin-Lions lemma \ref{lemma:aubin-lions}, and we follow closely to \cite{brzezniak2021well,debussche2011local,hu2023local}. For completeness, we highlight the main steps below.

Given an initial distribution $\mu_0$ satisfying \eqref{eq:mu.zero} for some $p\geq 4$ (which is implied by \eqref{condition:mu-zero}), and a stochastic basis $\Sc = \left(\Omega, \Fc, \Fb, \Pb\right)$,
let $u_0$ be an $\Fc_0$-measurable $\mathcal D_{s,2\kappa}$-valued random variable with law $\mu_0$. For fixed $\rho \geq M$, for each $n\in\mathbb N$, let $u_n$ be the solutions to the viscous approximation system \eqref{PE-modified-system-approximation} with $u_n(0) = u_0 $ on the stochastic basis $\Sc$. Let $\Uc_0$ be an auxiliary Hilbert space such that
the embedding $\Uc \subseteq \Uc_0$ is Hilbert-Schmidt. For any positive time $T>0$, define
\begin{equation*}
	\Xc_{u} = L^2\left(0, T; \mathcal D_{s-1} \right) \cap C\left(\left[0, T\right], \mathcal D_{s-3}\right), \quad \Xc_W = C\left(\left[0, T\right], \Uc_0\right), \quad \Xc = \Xc_{u} \times \Xc_W,
\end{equation*}
Let $\mu_{u}^n$, $\mu^n_W$ and $\mu^n$ be laws of $u_n$, $W$ and $(u_n, W)$ on $\Xc_{u}$, $\Xc_W$ and $\Xc$, respectively, namely
\begin{equation}
	\label{eq:mu.measure}
	\mu^n_{u}(\cdot) = \Pb\left( \left\lbrace u_n \in \cdot \right\rbrace\right), \qquad \mu_W^n(\cdot) = \Pb\left( \left\lbrace W \in \cdot \right\rbrace \right), \qquad \mu^n = \mu_{u}^n \otimes \mu_W^n.
\end{equation}
Below, we provide the existence of global martingale solutions by establishing Propositions~\ref{prop:approximating.sequence}--\ref{prop:global.martingale.existence} that closely follow from \cite[Propositions 4.1 and 7.1]{debussche2011local} or \cite[Proposition~ 3.2--3.3]{brzezniak2021well}.

\begin{proposition}
\label{prop:approximating.sequence}
Let $\mu_0$ be a probability measure on $\mathcal D_{s,2\kappa}$ satisfying
\begin{equation}\label{condition:p}
    \int_{\mathcal D_{s,2\kappa}} \|u\|_{\tilde s}^p d\mu_0(u)<\infty  
\end{equation}
with $p\geq4$  
and let $(\mu^n)_{n \geq 1}$ be the sequence of measures defined in \eqref{eq:mu.measure}. Then there exists a probability space $(\tilde{\Omega}, \tilde{\Fc}, \tilde{\Pb})$, a subsequence $n_k \to \infty$ as $k \to \infty$ and a sequence of $\Xc$-valued random variables $(\tilde{u}_{n_k}, \tilde{W}_{n_k})$ such that
\begin{enumerate}
	\item $(\tilde{u}_{n_k}, \tilde{W}_{n_k})$ converges in $\Xc$ to $(\tilde{u}, \tilde{W}) \in \Xc$ almost surely,
	\item $\tilde{W}_{n_k}$ is a cylindrical Wiener process with reproducing kernel Hilbert space $\Uc$ adapted to the filtration $\left( \Fc_t^{n_k} \right)_{t \geq 0}$, where $\left( \Fc_t^{n_k} \right)_{t \geq 0}$ is the completion of $\sigma(\tilde{W}_{n_k}, \tilde{u}_{n_k}; 0 \leq t' \leq t)$,
	\item for $t\in[0,T]$, each pair $(\tilde{u}_{n_k}, \tilde{W}_{n_k})$ satisfies the equation 
	\begin{equation}
	\label{eq:appr.after.skorohod}
	\begin{split}
	    &d \tilde u_{n_k} + \big[\theta_\rho(\|\tilde u_{n_k}\|_{s-1})\theta_\kappa(\|\partial_z \tilde v_{n_k} - (\p_z \tilde v_{n_k})_0\|_{L^\infty})(\tilde u_{n_k}\partial_x \tilde u_{n_k} + \tilde w_{n_k}\partial_z \tilde u_{n_k}   - \frac1{n_k} \partial_{xx} \tilde u_{n_k})\big]dt  
	    \\
	    & \hspace{1in}
	    = \theta_\rho(\|\tilde u_{n_k}\|_{s-1})\theta_\kappa(\|\partial_z \tilde v_{n_k} - (\p_z \tilde v_{n_k})_0\|_{L^\infty})\sigma (\tilde u_{n_k}) d\tilde W_{n_k}  \quad \text{in } H,
	\end{split}
	\end{equation}
	where $\tilde v_{n_k} = \partial_z \tilde u_{n_k}$ and $\tilde w_{n_k}(x,z) = -\int_0^z \partial_x \tilde u_{n_k}(x, z') d z'$.
\end{enumerate}
\end{proposition}

\begin{remark}
In \cite[Proposition 3.2]{brzezniak2021well}, the authors required a condition similar to \eqref{condition:p} with $p\geq 8$, so that their estimate (3.10) in \cite{brzezniak2021well} is valid. The corresponding estimate in our derivation is \eqref{estimate:w12}, which only requires $p\geq 4$.
\end{remark}

\begin{proof}
For part (i), to establish the tightness of $\lbrace \mu^n \rbrace_{n\geq 1}$  in $\Xc$, one can follow the argument in \cite[Lemma 4.1]{debussche2011local} with the spaces $D(A)$, $V$, $H$, and $V'$ therein being replaced by $\mathcal D_{s}$, $\mathcal D_{s-1}$, $\mathcal D_{s-2}$, and $\mathcal D_{s-3}$, and using Lemmas~\ref{lemma:aubin-lions} and estimates in Proposition~\ref{proposition:estimate-inviscid}. Consequently $\lbrace \mu^n \rbrace_{n\geq 1}$ is weakly compact by Prokhorov's theorem, and the first assertion follows immediately by the Skorokhod Theorem (\cite[Theorem 2.4]{da2014stochastic}). 

Parts (ii) and (iii) follow easily from \cite[Proposition 3.2]{brzezniak2021well}; see also \cite[Section 4.3.4]{bensoussan1995stochastic}.
\end{proof}

The following proposition gives the global existence of martingale solutions to the modified system \eqref{PE-modified-system}.

\begin{proposition}
\label{prop:global.martingale.existence}
Let $T>0$ be an arbitrary positive time, and let $(\tilde{u}_{n_k}, \tilde{W}_{n_k})$ be a sequence of $\Xc$-valued random variables on a probability space $(\tilde{\Omega}, \tilde{\Fc}, \tilde{\Pb})$ such that
\begin{enumerate}
	\item $(\tilde{u}_{n_k}, \tilde{W}_{n_k}) \to (\tilde{u}, \tilde{W})$ in the topology of $\Xc$, $\tilde{\Pb}$-a.s., that is,
	\begin{equation*}
		\tilde{u}_{n_k} \to \tilde{u} \ \text{in} \ L^2\left(0, T; \mathcal D_{s-1}\right) \cap C\left(\left[0, T \right], \mathcal D_{s-3}\right), \ \tilde{W}_{n_k} \to \tilde{W} \ \text{in} \ C\left(\left[0, T\right], \Uc_0 \right),
	\end{equation*}
	\item $\tilde{W}_{n_k}$ is a cylindrical Wiener process with reproducing kernel Hilbert space $\Uc$ adapted to the filtration $\left( \Fc_t^{n_k} \right)_{t \geq 0}$ that contains  $\sigma(\tilde{W}_{n_k}, \tilde{u}_{n_k}; 0 \leq t' \leq t)$,
	\item each pair $(\tilde{u}_{n_k}, \tilde{W}_{n_k})$ satisfies \eqref{eq:appr.after.skorohod}.
\end{enumerate}
Now, let $\tilde{\Fc}_t$ be the completion of $\sigma(\tilde{W}(s), \tilde{u}(s); 0 \leq t' \leq t)$ and $\tilde{\Sc} = (\tilde{\Omega}, \tilde{\Fc}, ( \tilde{\Fc}_t )_{t \geq 0}, \tilde{\Pb})$. Then $(\tilde{\Sc}, \tilde{W}, \tilde{u})$ is a martingale solution to the modified system \eqref{PE-modified-system} on the time interval $[0,T]$. As $T>0$ is arbitrary, the solution exists globally in time.
Moreover, $\tilde{u}$ satisfies
\begin{equation}
	\label{eq:martingale.solution.approx.regularity}
	\tilde{u} \in L^2\left( \tilde \Omega; L^\infty\left( 0, T; \mathcal D_{s,\kappa} \right) \cap C([0,T]; \mathcal D_{s-1,\kappa}) \right) .
\end{equation}
\end{proposition}

\begin{proof}
For any $\phi\in H$ and $t\in[0,T]$, using \eqref{eq:appr.after.skorohod} brings
\begin{equation}\label{weak-sol-galerkin}
\begin{split}
     &\Big\langle \tilde u_{n_k}(t), \phi\Big\rangle + \int_0^t \Big\langle \theta_\rho(\|\tilde u_{n_k}\|_{s-1})\theta_\kappa(\|\partial_z \tilde v_{n_k} - (\p_z \tilde v_{n_k})_0\|_{L^\infty})(\tilde u_{n_k}\partial_x \tilde u_{n_k} + \tilde w_{n_k}\partial_z \tilde u_{n_k}  -\frac1{n_k} \partial_{xx} \tilde u_{n_k} ) ,\phi \Big\rangle dt'
	\\
	&= \Big\langle\tilde  u_{n_k}(0), \phi\Big\rangle + \int_0^t \Big\langle\theta_\rho(\|\tilde u_{n_k}\|_{s-1})\theta_\kappa(\|\partial_z \tilde v_{n_k} - (\p_z \tilde v_{n_k})_0\|_{L^\infty}) \sigma(\tilde u_{n_k}) , \phi \Big\rangle d\tilde{W}_{n_k}.
\end{split}
\end{equation}
From assumption~(i), we deduce that
\begin{equation}\label{convergence:1}
    \tilde u_{n_k} \to \tilde u \text{ in } L^2\left(0, T; \mathcal D_{s-1}\right) \cap C\left(\left[0, T \right], \mathcal D_{s-3}\right) \; \tilde{\Pb}-a.s..
\end{equation}
Thanks to Proposition \ref{proposition:estimate-inviscid} we know that $\{\tilde u_{n_k}\}$ is uniformly bounded in $\left( \tilde \Omega;
    L^\infty\big( 0,T; \mathcal D_{s}   \big)\right)$. Then by the
Banach–Alaoglu theorem one has
\begin{equation*}
\begin{split}
    &\tilde{u} \in L^2 \left( \tilde \Omega;
    L^\infty\big( 0,T; \mathcal D_{s}   \big)\right),
    \qquad \tilde{u}_{n_k} \overset{\ast}{\rightharpoonup} \tilde{u} \text{ in } L^2 \left( \tilde \Omega; L^\infty\big( 0,T; \mathcal D_{s}   \big) \right).
\end{split}
\end{equation*}
On the other hand because $\kappa\leq \p_{zz} \tilde u_{n_k} \leq \frac1\kappa$ for any $n_k$ and any $(t,x,z)\in[0,T]\times \mathbb D$ a.s., by virtue of \eqref{convergence:1} we infer that $\kappa\leq \p_{zz} \tilde u \leq \frac1\kappa$ for any $(t,x,z)\in[0,T]\times \mathbb D$ a.s.. This implies that 
\begin{equation}\label{bound:v-and-vtilde}
    \tilde{u} \in L^2 \left( \tilde \Omega;
    L^\infty\big( 0,T; \mathcal D_{s,\kappa}   \big)\right).
\end{equation}
Moreover, from Lemma~\ref{proposition:estimate-inviscid} with $p>2$, one has the following uniform integrability for $\tilde {u}_{n_k}$: 
\begin{equation}\label{uniform-integrable}
\begin{split}
     \sup\limits_{k\in \mathbb N} \tilde \Eb \left[\left(\int_0^T \|\tilde u_{n_k}\|^2_{s-1} dt \right)^{\frac{p}{2}}  \right] \leq \sup\limits_{k\in \mathbb N} \tilde \Eb \left[\left(\int_0^T \|\tilde u_{n_k}\|^2_{\tilde s} dt \right)^{\frac{p}{2}}  \right]
     \leq  C_T  \sup\limits_{k\in \mathbb N} \tilde \Eb \sup\limits_{t\in [0,T]} \|\tilde u_{n_k}\|_{\tilde s}^p   < \infty.
\end{split}
\end{equation}
The Vitali Convergence Theorem, together with \eqref{convergence:1} and \eqref{uniform-integrable}, implies that
\begin{equation}\label{convergence:vnk-probablity}
    \tilde{u}_{n_k} \rightarrow \tilde{u} \text{ in } L^2 \left( \tilde \Omega; L^2\big( 0,T; \mathcal D_{s-1}   \big) \right).
\end{equation}
Consequently, a further subsequence, with a slight abuse of notation still denoted by $\tilde u_{n_k}$, satisfies
\begin{equation}\label{convergence:Vnk-pw}
    \tilde{u}_{n_k} \rightarrow \tilde{u}  \text{ in } \mathcal D_{s-1} \text{ for a.a. } (t,\omega) \in (0,T)\times \tilde \Omega.
\end{equation}

\medskip

\noindent \textbf{\underline{Convergence of the linear terms}:}
As the initial data $u_{n_k}(0) = u_0$ and by the convergence \eqref{convergence:1}, we have
\begin{equation}\label{convergence:ini}
    \Big\langle\tilde  u_{n_k}(0), \phi\Big\rangle  = \Big\langle\tilde  u_0, \phi\Big\rangle, \quad \Big\langle\tilde  u_{n_k}(t), \phi\Big\rangle \to  \Big\langle\tilde  u(t), \phi\Big\rangle \text{ a.s..}
\end{equation}
By Proposition \ref{proposition:estimate-inviscid}, for any $t\in[0,T]$, as $n_k\to \infty$, we have
\begin{equation*}
\begin{split}
     &\mathbb E \int_0^T \left|\int_0^t \Big\langle \theta_\rho(\|\tilde u_{n_k}\|_{s-1})\theta_\kappa(\|\partial_z \tilde v_{n_k} - (\p_z \tilde v_{n_k})_0\|_{L^\infty}) \frac1{n_k} \partial_{xx} \tilde u_{n_k}  ,\phi \Big\rangle dt' \right| dt\rightarrow 0.
\end{split}
\end{equation*}
Consequently, a further subsequence, still denoted by $\tilde u_{n_k}$, satisfies
\begin{equation}\label{converge:viscosity}
    \int_0^t \Big\langle \theta_\rho(\|\tilde u_{n_k}\|_{s-1})\theta_\kappa(\|\partial_z \tilde v_{n_k} - (\p_z \tilde v_{n_k})_0\|_{L^\infty}) \frac1{n_k} \partial_{xx} \tilde u_{n_k}  ,\phi \Big\rangle dt' \to 0 \quad \text{ for a.a. } (t,\omega) \in (0,T)\times \tilde \Omega.
\end{equation}

\medskip

\noindent \textbf{\underline{Convergence of the nonlinear terms}:}
For $t\in[0,T]$, since $(\tilde v_{n_k})_0 = \tilde v_0$, one has
\begin{align*}
        &\Bigg|\int_0^t \Big\langle \theta_\rho(\|\tilde u_{n_k}\|_{s-1})\theta_\kappa(\|\partial_z \tilde v_{n_k} - (\p_z \tilde v_{n_k})_0\|_{L^\infty})(\tilde u_{n_k}\partial_x \tilde u_{n_k} + \tilde w_{n_k}\partial_z \tilde u_{n_k} ) 
        \\
        &\hspace{2in}- \theta_\rho(\|\tilde u\|_{s-1})\theta_\kappa(\|\partial_z \tilde v - \p_z \tilde v_0\|_{L^\infty})(\tilde u\partial_x \tilde u + \tilde w \partial_z \tilde u)  ,\phi \Big\rangle dt' \Bigg|
        \\
        \leq &\Bigg|\int_0^t \Big\langle \left(\theta_\rho(\|\tilde u_{n_k}\|_{s-1}) -\theta_\rho(\|\tilde u\|_{s-1}) \right)\theta_\kappa(\|\partial_z \tilde v_{n_k} - \p_z \tilde v_0\|_{L^\infty})(\tilde u_{n_k}\partial_x \tilde u_{n_k} + \tilde w_{n_k}\partial_z \tilde u_{n_k} )  ,\phi \Big\rangle dt' \Bigg|
        \\
        &+ \Bigg|\int_0^t \Big\langle \theta_\rho(\|\tilde u\|_{s-1})\left(\theta_\kappa(\|\partial_z \tilde v_{n_k} - \p_z \tilde v_0\|_{L^\infty}) - \theta_\kappa(\|\partial_z \tilde v - \p_z \tilde v_0\|_{L^\infty})\right)(\tilde u_{n_k}\partial_x \tilde u_{n_k} + \tilde w_{n_k}\partial_z \tilde u_{n_k} )  ,\phi \Big\rangle dt' \Bigg|
        \\
        &+ \Bigg|\int_0^t \Big\langle  \theta_\rho(\|\tilde u\|_{s-1}) \theta_\kappa(\|\partial_z \tilde v - \p_z \tilde v_0\|_{L^\infty}) \left(\tilde u_{n_k}\partial_x \tilde u_{n_k} + \tilde w_{n_k}\partial_z \tilde u_{n_k} - \tilde u\partial_x \tilde u  - \tilde w \partial_z \tilde u\right) , \phi \Big\rangle dt' \Bigg| 
        \\
        =:& I_1 + I_2 + I_3.
\end{align*}

For the term $I_1$, due to \eqref{convergence:Vnk-pw} and the smoothness of  $\theta_\rho$, we obtain
\begin{equation}\label{convergence-theta}
   \theta_\rho(\|\tilde u_{n_k}\|_{s-1}) - \theta_\rho(\|\tilde u\|_{s-1}) \rightarrow 0 \qquad \text{for a.a. } (t,\omega) \in (0,T)\times \tilde \Omega.
\end{equation}
By Proposition \ref{proposition:estimate-inviscid} and inequalities~\eqref{poincare}--\eqref{ine:controls}, we know that
\begin{align*}
    &\Eb \int_0^T \left| \Big\langle \left(\theta_\rho(\|\tilde u_{n_k}\|_{s-1}) -\theta_\rho(\|\tilde u\|_{s-1}) \right)\theta_\kappa(\|\partial_z \tilde v_{n_k} - \p_z \tilde v_0\|_{L^\infty})(\tilde u_{n_k}\partial_x \tilde u_{n_k} + \tilde w_{n_k}\partial_z \tilde u_{n_k} ) ,\phi \Big\rangle  \right|dt 
    \\
    &\leq C\Eb \int_0^T\|\tilde u_{n_k}\|_{\tilde s}^2 \|\phi\| dt <\infty.
\end{align*}
The dominated convergence theorem together with \eqref{convergence-theta} yields that 
\begin{equation*}
    \begin{split}
        &\Eb \int_0^T I_1 dt 
        \leq C_T \Eb \int_0^T \Big| \Big\langle \left(\theta_\rho(\|\tilde u_{n_k}\|_{s-1}) -\theta_\rho(\|\tilde u\|_{s-1}) \right)\theta_\kappa(\|\partial_z \tilde v_{n_k} - \p_z \tilde v_0\|_{L^\infty})
        \\
        &\hspace{2in}\times(\tilde u_{n_k}\partial_x \tilde u_{n_k} + \tilde w_{n_k}\partial_z \tilde u_{n_k} ) ,\phi \Big\rangle  \Big|dt \rightarrow 0.
    \end{split}
\end{equation*}
Thinning the sequence if necessary, we conclude that $I_1\rightarrow 0$ for a.a. $(t,\omega) \in (0,T)\times \tilde \Omega$. 

For the term $I_2$, using \eqref{convergence:Vnk-pw} and the Sobolev inequality, since $s\geq 6$, one has
\[
\|\partial_z \tilde v_{n_k} - \p_z \tilde v_0\|_{L^\infty} -  \|\partial_z \tilde v - \p_z \tilde v_0\|_{L^\infty} \leq  
\|\partial_z \tilde v_{n_k} - \partial_z \tilde v\|_{L^\infty} \leq 
C\|\tilde{u}_{n_k} - \tilde{u}\|_{s-1} \rightarrow 0, 
\]
for a.a. $(t,\omega) \in (0,T)\times \tilde \Omega.$ Since $\theta_\kappa$ is smooth, one has
\begin{equation}\label{convergence-theta-2}
   \theta_\kappa(\|\partial_z \tilde v_{n_k} - \p_z \tilde v_0\|_{L^\infty}) - \theta_\kappa(\|\partial_z \tilde v - \p_z \tilde v_0\|_{L^\infty}) \rightarrow 0, \qquad \text{for a.a. } (t,\omega) \in (0,T)\times \tilde \Omega.
\end{equation}
Then following similarly as for $I_1$, we can get that $I_2\rightarrow 0$ for a.a. $(t,\omega) \in (0,T)\times \tilde \Omega.$ 

For the term $I_3$, by the H\"older inequality and the Sobolev inequality, together with  \eqref{convergence:Vnk-pw}, one obtains that, for a.a. $(t,\omega) \in (0,T)\times \tilde \Omega$,
\begin{align*}
    I_3 &\leq \int_0^T \|\phi\|\Big(\|(\tilde u_{n_k} - \tilde u) \partial_x \tilde u_{n_k}\| + \|\tilde u (\partial_x \tilde u_{n_k} - \partial_x \tilde u)\| + \|(\tilde w_{n_k} - \tilde w) \partial_z \tilde u_{n_k}\| + \|\tilde w (\partial_z \tilde u_{n_k} - \partial_z \tilde u)\| \Big) dt
    \\
    & \leq \|\phi\| \int_0^T \|(\tilde u_{n_k} - \tilde u) \|_{s-1} (\|\tilde u_{n_k}\|_{s-1} + \|\tilde u\|_{s-1} ) dt \rightarrow 0.
\end{align*}

Combining the estimates of $I_1$ to $I_3$ gives, for a.a. $(t,\omega) \in (0,T)\times \tilde \Omega$, 
\begin{equation}\label{convergence:deter}
\begin{split}
    &\int_0^t \Big\langle \theta_\rho(\|\tilde u_{n_k}\|_{s-1})\theta_\kappa(\|\partial_z \tilde v_{n_k} - (\p_z \tilde v_{n_k})_0\|_{L^\infty})(\tilde u_{n_k}\partial_x \tilde u_{n_k} + \tilde w_{n_k}\partial_z \tilde u_{n_k} ),\phi \Big\rangle dt'
    \\
    &\rightarrow
    \int_0^t \Big\langle  \theta_\rho(\|\tilde u\|_{s-1})\theta_\kappa(\|\partial_z \tilde v - \p_z \tilde v_0\|_{L^\infty})(\tilde u\partial_x \tilde u + \tilde w \partial_z \tilde u),\phi \Big\rangle dt'.
\end{split}
\end{equation}

\medskip

\noindent \textbf{\underline{Convergence of the stochastic terms}:}
Similar to the convergence of nonlinear terms, with Assumption~\eqref{noise-ine} and the boundedness of $\tilde u_{n_k}$ and $\tilde u$, one can compute
\begin{equation*}
    \begin{split}
        &\left\| \theta_\rho(\|\tilde u_{n_k}\|_{s-1})\theta_\kappa(\|\partial_z \tilde v_{n_k} - (\p_z \tilde v_{n_k})_0\|_{L^\infty})\sigma_{}(\tilde u_{n_k}) - 
 \theta_\rho(\|\tilde u\|_{s-1})\theta_\kappa(\|\partial_z \tilde v- \p_z \tilde v_0\|_{L^\infty})\sigma(\tilde u) \right\|_{L^2\left(0, T; L_2\left(\Uc, L^2\right)\right)}^2 
 \\
 \leq &C \left\| (\theta_\rho(\|\tilde u_{n_k}\|_{s-1}) -\theta_\rho(\|\tilde u\|_{s-1}))
 \theta_\kappa(\|\partial_z \tilde v_{n_k} - (\p_z \tilde v_{n_k})_0\|_{L^\infty})\sigma_{}(\tilde u_{n_k})\right\|_{L^2\left(0, T; L_2\left(\Uc, L^2\right)\right)}^2  
 \\
 & + C \left\| \theta_\rho(\|\tilde u\|_{s-1})
 (\theta_\kappa(\|\partial_z \tilde v_{n_k} - (\p_z \tilde v_{n_k})_0\|_{L^\infty} )- \theta_\kappa(\|\partial_z \tilde v- \p_z \tilde v_0\|_{L^\infty}))\sigma_{}(\tilde u_{n_k})\right\|_{L^2\left(0, T; L_2\left(\Uc, L^2\right)\right)}^2
 \\
 & + C \left\| \theta_\rho(\|\tilde u\|_{s-1})
  \theta_\kappa(\|\partial_z \tilde v- \p_z \tilde v_0\|_{L^\infty})(\sigma_{}(\tilde u_{n_k}) - \sigma(\tilde u))\right\|_{L^2\left(0, T; L_2\left(\Uc, L^2\right)\right)}^2 \to 0.
    \end{split}
\end{equation*}
Therefore, with the convergence result \eqref{convergence:1} and \eqref{convergence:Vnk-pw}, we have
\begin{align*}
    &\theta_\rho(\|\tilde u_{n_k}\|_{s-1})\theta_\kappa(\|\partial_z \tilde v_{n_k} - (\p_z \tilde v_{n_k})_0\|_{L^\infty})\sigma(\tilde u_{n_k}) 
    \\
    &\to \theta_\rho(\|\tilde u\|_{s-1})\theta_\kappa(\|\partial_z \tilde v- \p_z \tilde v_0\|_{L^\infty})\sigma(\tilde u) \text{ in } L^2\left(0,T; L_2\left(\Uc, L^2\right)\right), \; \tilde{\Pb}-a.s..
\end{align*} 
This implies the convergence in probability in $L^2\left(0, T; L_2\left(\Uc, L^2\right)\right).$ Thanks to \cite[Lemma 2.1]{debussche2011local}, from assumption (i) in Proposition \ref{prop:global.martingale.existence}, we obtain that
\begin{equation}
	\label{eq:sigma.dw.convergence}
	\int_0^\cdot \theta_\rho(\|\tilde u_{n_k}\|_{s-1})\theta_\kappa(\|\partial_z \tilde v_{n_k} - (\p_z \tilde v_{n_k})_0\|_{L^\infty})\sigma(\tilde u_{n_k}) \, d\tilde{W}_{n_k} \to \int_0^\cdot \theta_\rho(\|\tilde u\|_{s-1})\theta_\kappa(\|\partial_z \tilde v- \p_z \tilde v_0\|_{L^\infty})\sigma(\tilde u) \, d\tilde{W}
\end{equation}
in probability in $L^2\left(0,T; L^2\right)$. Then with Proposition~\ref{proposition:estimate-inviscid} and Assumption~\eqref{noise-ine}, for $p>2$, we get
\begin{equation}
    	\label{eq:sigma.nk.lp}
     \begin{split}
         &\sup_{k \in \Nb} \Eb \left[ \int_0^T \|\theta_\rho(\|\tilde u_{n_k}\|_{s-1})\theta_\kappa(\|\partial_z \tilde v_{n_k} - (\p_z \tilde v_{n_k})_0\|_{L^\infty}) \sigma(\tilde u_{n_k}) \|_{L_2\left(\Uc, L^2\right)}^{2} \, dt \right]^{\frac p2} 
         \\
         &\leq C_T \sup_{k \in \Nb} \Eb \left[ 1 + \sup_{t \in [0,T]} \| \tilde u_{n_k} \|^{p} \right] <\infty.
     \end{split} 		
\end{equation}
The estimate \eqref{eq:sigma.nk.lp} together with \eqref{eq:bdg} and the Vitali convergence theorem yield that the convergence \eqref{eq:sigma.dw.convergence} also occurs in the space $L^2(\tilde{\Omega}; L^2\left(0,T; L^2\right))$. Then for any $\mathcal R \subset [0,T] \times \tilde \Omega$ measurable, one has
\begin{equation*}
\begin{split}
    &\lim_{k \rightarrow \infty}\Eb 
     \int_0^T \chi_{\mathcal{R}} 
     \left( 
       \int_0^t \Big\langle 
      \theta_\rho(\|\tilde u_{n_k}\|_{s-1})\theta_\kappa(\|\partial_z \tilde v_{n_k} - (\p_z \tilde v_{n_k})_0\|_{L^\infty})\sigma(\tilde{u}_{n_k})
       ,  \phi\Big\rangle  d\tilde W_{n_k}
    \right) dt 
    \\
    &=
  \Eb 
    \int_0^T \chi_{\mathcal{R}} 
    \left( 
      \int_0^t \Big\langle 
    \theta_\rho(\|\tilde u\|_{s-1})\theta_\kappa(\|\partial_z \tilde v- \p_z \tilde v_0\|_{L^\infty})\sigma(\tilde{u})
     ,  \phi\Big\rangle  d \tilde W
   \right) dt.
\end{split} 
\end{equation*}
This implies that for a.a. $(t, \omega) \in [0,T] \times \tilde{\Omega}$,
\begin{equation}\label{convergence:sto}
\begin{split}
    &\int_0^t \Big\langle \theta_\rho(\|\tilde u_{n_k}\|_{s-1})\theta_\kappa(\|\partial_z \tilde v_{n_k} - (\p_z \tilde v_{n_k})_0\|_{L^\infty})\sigma(\tilde u_{n_k}), \phi \Big\rangle d\tilde{W}_{n_k}
    \\
    &\rightarrow \int_0^t \Big\langle \theta_\rho(\|\tilde u\|_{s-1})\theta_\kappa(\|\partial_z \tilde v- \p_z \tilde v_0\|_{L^\infty}) \sigma(\tilde u), \phi \Big\rangle d\tilde{W}.
\end{split} 
\end{equation}

\medskip

\noindent \textbf{\underline{Combining the estimates}:} Applying the convergences   \eqref{convergence:Vnk-pw},  \eqref{convergence:ini}, \eqref{converge:viscosity}, \eqref{convergence:deter} and \eqref{convergence:sto} to \eqref{weak-sol-galerkin}, we infer that for all $\phi \in H$ and for a.a. $(t, \omega) \in [0,T] \times \tilde{\Omega}$,
\begin{equation}\label{equality-before-st}
\begin{split}
     &\Big\langle \tilde u(t), \phi\Big\rangle + \int_0^t \Big\langle \theta_\rho(\|\tilde u\|_{s-1})\theta_\kappa(\|\partial_z \tilde v - \p_z \tilde v_0\|_{L^\infty})(\tilde u\partial_x \tilde u + \tilde w\partial_z \tilde u  ) ,\phi \Big\rangle dt'
	\\
	&= \Big\langle\tilde  u(0), \phi\Big\rangle + \int_0^t \Big\langle\theta_\rho(\|\tilde u\|_{s-1})\theta_\kappa(\|\partial_z \tilde v - \p_z \tilde v_0\|_{L^\infty}) \sigma(\tilde u) , \phi \Big\rangle d\tilde{W}.
\end{split}
\end{equation}
As $T$ is arbitrary, we have obtained the desired result that  $(\tilde{\Sc}, \tilde{W}, \tilde{u})$ is a global martingale solution to the modified system \eqref{PE-modified-system}. Note that the pressure gradient disappears due to the orthogonality in $H$.

\medskip

\noindent \textbf{\underline{Continuity in time}:} It remains to prove \eqref{eq:martingale.solution.approx.regularity}. The $L^\infty$ bound in $\mathcal D_{s,\kappa}$ has been shown in \eqref{bound:v-and-vtilde}. The ``continuity in time'' property is essential to justify the stopping time defined later in \eqref{stopingtime:eta} is positive almost surely, which is used to define the local martingale solution to the original system~\eqref{PE-inviscid-system}.
The proof follows similarly as in \cite[Section 7.3]{debussche2011local} and \cite[Proposition 3.4]{hu2023local}. By the property of $\sigma$ in \eqref{noise-ine} and the regularity of $\tilde u$ \eqref{bound:v-and-vtilde}, as $\mathcal D_{s,\kappa}\subseteq \mathcal D_s$, we have
\[
	\theta_\rho(\|\tilde u\|_{s-1})\theta_\kappa(\|\partial_z \tilde v - \p_z \tilde v_0\|_{L^\infty}) \sigma(\tilde u) \in L^2\left(\tilde{\Omega}; L^2\left(0,T; L_2\left(\Uc, \mathcal D_{s}\right)\right)\right).
\]
Then there exists a continuous in time version of the solution to
\[
	dZ  = \theta_\rho(\|\tilde u\|_{s-1})\theta_\kappa(\|\partial_z \tilde v - \p_z \tilde v_0\|_{L^\infty}) \sigma(\tilde u) \, d\tilde{W}, \qquad Z(0) = \tilde u_0,
\]
satisfying
\begin{equation}
	\label{eq:z.regularity}
	Z \in L^2\left(\tilde{\Omega}; C\left([0,T], \mathcal D_{s}\right)\right).
\end{equation}
Defining $\bar{u} = \tilde u - Z$, by \eqref{PE-modified-1} we have $\Pb$-a.s. 
\begin{equation*}
	\tfrac{d}{dt} \bar{u}  +  \theta_\rho(\|\tilde u\|_{s-1})\theta_\kappa(\|\partial_z \tilde v - \p_z \tilde v_0\|_{L^\infty})(\tilde u\partial_x \tilde u +\tilde w\partial_z \tilde u + \p_x \tilde p  )  = 0, \qquad \bar{u}(0) = 0.
\end{equation*}
Using \eqref{bound:v-and-vtilde}, \eqref{poincare} and \eqref{ine:controls}, and the fact that $D_{s-1}$ is a Banach algebra gives
\[
	\theta_\rho(\|\tilde u\|_{s-1})\theta_\kappa(\|\partial_z \tilde v - \p_z \tilde v_0\|_{L^\infty})(\tilde u\partial_x \tilde u +\tilde w\partial_z \tilde u) \in L^2\left(\tilde{\Omega}; L^2\left(0,T; \mathcal D_{s-1}\right)\right).
\]
Since $\partial_x \tilde p$ is orthogonal to $D_{s-1}$, we have
$
\frac{d}{dt} \bar{u} \in L^2\left(\tilde{\Omega}; L^2\left(0,T; \mathcal D_{s-1}\right)\right).
$
Combining \eqref{bound:v-and-vtilde} and \eqref{eq:z.regularity}, and using $\bar u = \tilde u -Z$,  we easily obtain
$
 \bar u \in L^2\left(\tilde{\Omega}; L^\infty\left(0,T; \mathcal D_{s}\right)\right).
$
Next, using the standard Aubin-Lions compactness theorem \cite[Corollary 4]{simon1986compact}, for any $s'\in[s-1, s)$ we have 
\[
 \bar u \in L^2\left(\tilde{\Omega}; C\left([0,T], \mathcal D_{s'}\right)\right).
\]
In particular, $\bar u \in L^2\left(\tilde{\Omega}; C\left([0,T], \mathcal D_{s-1}\right)\right).$
This together with \eqref{eq:z.regularity} and the fact that $\kappa\leq \tilde u \leq \frac1\kappa$ imply that $\tilde{u} \in L^2\left( \Omega; C\left(\left[0,T \right], \mathcal D_{s-1,\kappa}\right)\right)$. 
\end{proof}

\begin{remark}
Unlike the viscous system~\eqref{PE-modified-system-approximation} where the continuity in $\mathcal D_{s,\kappa}$ is available, for the inviscid system~\eqref{PE-modified-system} we have a weaker result, i.e., the continuity in $\mathcal D_{s',\kappa}$ for $s'<s.$ Nevertheless, this is enough for our purpose to define a suitable stopping time to construct the local martingale solutions to the original inviscid system \eqref{PE-inviscid-system}, see the next corollary.
\end{remark}

\begin{corollary}
\label{cor:loc.mart.sol}
Suppose that $\mu_0$ satisfies \eqref{condition:mu-zero} with a constant $M>0$. Let $\rho \geq M$, and let $(\tilde{\Sc}, \tilde{W}, \tilde{u})$ be the global martingale solution to the modified system \eqref{PE-modified-system} obtained in Proposition~\ref{prop:global.martingale.existence}. Define the stopping time $\eta$ by
\begin{equation}\label{stopingtime:eta}
	\eta = \inf \left\lbrace t \geq 0: \,\| \tilde{u}(t)\|_{s-1} \geq \frac\rho2 \right\rbrace \wedge  \inf \left\lbrace t \geq 0: \, \|\partial_{zz} \tilde u(t) - \partial_{zz} \tilde u(0) \|_{L^\infty} \geq \frac\kappa2 \right\rbrace,
\end{equation}
Then $\eta>0$ $\tilde{\mathbb{P}}$-a.s., $(\tilde{\Sc}, \tilde{W}, \tilde{u}, \eta)$ is a local martingale solution to system \eqref{PE-inviscid-system}, 
and for any $T>0$,
\begin{equation}\label{regularity-after-st}
	\tilde{u}\left( \cdot \wedge \eta \right) \in L^2 \left(\tilde\Omega; C\left( [0, T], \mathcal D_{s-1,\kappa} \right)  \right), \quad \mathds{1}_{[0, \eta]}(\cdot) \tilde u(\cdot)  \in L^2\left(\tilde \Omega; L^\infty\left( 0, T; \mathcal D_{s,\kappa} \right) \right).
\end{equation}
\end{corollary}
\begin{proof}
First of all, the stopping time is well-defined due to the continuity in time of $\tilde u$ \eqref{eq:martingale.solution.approx.regularity}. Notice that for $t \in[0, \eta]$, the cut-off functions $\theta_\rho$ and $\theta_\kappa$ are not activated in system \eqref{PE-modified-system}, and thus it coincides with the original system \eqref{PE-inviscid-system}. Then with \eqref{equality-before-st}, we have 
\begin{equation}
	\tilde u\left(t \wedge \eta\right) + \int_0^{t \wedge \eta} \big[\tilde u\p_x \tilde u + \tilde w\p_z \tilde u\big] dt' = \tilde u(0) + \int_0^{t \wedge \eta} \sigma(\tilde u) \, d\tilde W \quad \text{in } H,
\end{equation}
which completes the proof of \eqref{eq:solution.def}.
The fact that $\eta>0$ $\tilde{\mathbb{P}}$-a.s. follows from the condition~\eqref{condition:mu-zero} and $\rho \geq M$. 
By continuity in time of $\tilde u$ and the definition of $\eta$, we know on $t\in[0,\eta]$ it holds that $\|\partial_{zz} \tilde u(t) - \partial_{zz} \tilde u(0) \|_{L^\infty} \leq \frac\kappa2$. By the regularity of $\tilde u$ and as $\tilde u(0)\in \mathcal D_{s,2\kappa}$ a.s., we infer that $\kappa\leq\|\p_{zz} \tilde u (t)\|\leq \frac1\kappa$ a.s. for all $(x,z)\in \mathbb D$ on $t\in[0,\eta]$. For $t\in(\eta,T]$, $\tilde u(t\wedge \eta) = \tilde u(\eta)$ and therefore $\kappa\leq\|\p_{zz} \tilde u (t\wedge \eta)\|\leq \frac1\kappa$ for any $t\in[0,T]$, and meaning that $\tilde u(t \wedge\eta)$ satisfies the local Rayleigh condition. Then the regularity \eqref{regularity-after-st} is implied by \eqref{eq:martingale.solution.approx.regularity}.
\end{proof}

\section{Pathwise Uniqueness}\label{sec:uniqueness-inviscid}
In this section, we first establish the following proposition concerning the pathwise uniqueness of the modified system \eqref{PE-modified-system}.
\begin{proposition}
\label{prop:pathwise.uniqueness}
Suppose that $\sigma$ satisfies \eqref{noise-ine}, and
let $\Sc = \left(\Omega, \Fc, \Fb = \Fct, \Pb \right)$ and $W$ be fixed.  
Suppose 
there exist two global martingale solutions $\left(\Sc, W, u^1\right)$ and $\left(\Sc, W, u^2\right)$ to the modified system \eqref{PE-modified-system}. Denote $\Omega_0 = \left\lbrace u^1_0 = u^2_0 \right\rbrace 
\subseteq \Omega$, then
\begin{equation*}
	\Pb \left( \left\lbrace \mathds{1}_{\Omega_0}\left(u^1(t) - u^2(t)\right) = 0 \ \text{for all} \ t\geq 0 \right\rbrace \right) = 1.
\end{equation*}
\end{proposition}

\begin{proof}
Let $u = u^1 - u^2$, $v=v^1-v^2 = \p_z u^1 - \p_z u^2$, and $\bar{u} = \mathds{1}_{\Omega_0} u$, $\bar{v} = \mathds{1}_{\Omega_0} v$. Denote by 
$$
\|v\|^2_{\hat H^{s-1}} := \sum\limits_{\substack{|\alpha|\leq s-1,\\ D^\alpha \neq \p_x^{s-1}}} \|D^\alpha v\|^2 + \lnorm \frac{\p_x^{s-1} v}{\sqrt{\p_z v^1}} \rnorm^2, \quad \|u\|^2_{\widehat{s-1}} := \|u\|^2 + \|v\|^2_{\hat H^{s-1}}.
$$
From Proposition \ref{proposition:estimate-inviscid}~(ii) and Proposition \ref{prop:global.martingale.existence}, we know that $\kappa\leq \partial_z v^1\leq\frac1{\kappa}$ and $\kappa\leq\partial_z v^2\leq\frac1{\kappa}$ for all $t\geq 0$ a.s.. Then we have the following bound:
\begin{equation}
    \begin{split}
        \|u\|_{s-1}^2 \leq (1+\lnorm\sqrt{\p_z v^1}\rnorm_{L^\infty}^2) \|u\|_{\widehat{s-1}}^2 \leq C_{\kappa} \|u\|_{\widehat{s-1}}^2.
    \end{split}
\end{equation}
Let $\eta^n$ be the stopping time defined by
\begin{equation}
    \begin{split}
        \eta^n = \inf \Big\lbrace t \geq 0 : \int_0^t 
    \left(1+\|u^1\|_{\tilde s}^2+ \|u^2\|_{\tilde s}^2\right)\, dt' \geq n \Big\rbrace.
    \end{split}
\end{equation}
Since $u_0 \in L^p(\Omega; \mathcal D_{s,2\kappa})$ with $p\geq 4$, from the estimates in Proposition~\ref{proposition:estimate-inviscid} we deduce that $\lim\limits_{n\to \infty}\eta^n \to \infty$ $\Pb$-a.s.\ and therefore it suffices to show that 
\[
	\Eb \sup_{t' \in \left[0, \eta^n \wedge t \right]} \| u(t') \|_{\widehat{s-1}}^2  = 0,
\]
for all $t\geq 0$ and $n \in \Nb$. 

Subtracting the equation of $u^2$ from the one of $u^1$, and the one of $v^2$ from the one of $v^1$, produces
\begin{equation*}
    \begin{split}
       &d u + \Big[\theta_\rho(\|u^1\|_{s-1})\theta_\kappa(\|\partial_z v^1 - \p_z v^1_0\|_{L^\infty})(u^1\partial_x u^1 + w^1\partial_z u^1 + \p_x p^1  )
       \\
       &\qquad \qquad - \theta_\rho(\|u^2\|_{s-1})\theta_\kappa(\|\partial_z v^2 - \p_z v^2_0\|_{L^\infty})(u^2\partial_x u^2 + w^2\partial_z u^2 + \p_x p^2  )\Big] dt \nonumber
    \\
    &  =\Big[\theta_\rho(\|u^1\|_{s-1})\theta_\kappa(\|\partial_z v^1 - \p_z v^1_0\|_{L^\infty})\sigma (u^1) 
- \theta_\rho(\|u^2\|_{s-1})\theta_\kappa(\|\partial_z v^2 - \p_z v^2_0\|_{L^\infty})\sigma (u^2) \Big] dW , 
	   \\
	   &u(0) = u^1(0) - u^2(0),
    \end{split}
\end{equation*}
and
\begin{equation*}
    \begin{split}
       &d v + \Big[\theta_\rho(\|u^1\|_{s-1})\theta_\kappa(\|\partial_z v^1 - \p_z v^1_0\|_{L^\infty})(u^1\partial_x v_n^1 + w^1\partial_z v^1   )
       \\
       &\qquad \qquad - \theta_\rho(\|u^2\|_{s-1})\theta_\kappa(\|\partial_z v^2 - \p_z v^2_0\|_{L^\infty})(u^2\partial_x v^2 + w^2\partial_z v^2   )\Big] dt \nonumber
    \\
    &  =\Big[\theta_\rho(\|u^1\|_{s-1})\theta_\kappa(\|\partial_z v^1 - \p_z v^1_0\|_{L^\infty})\p_z\sigma (u^1) 
    - \theta_\rho(\|u^2\|_{s-1})\theta_\kappa(\|\partial_z v^2 - \p_z v^2_0\|_{L^\infty})\p_z\sigma (u^2) \Big] dW , 
	   \\
	   &v(0) = v^1(0) - v^2(0).
    \end{split}
\end{equation*}

\medskip

\noindent\underline{\textbf{Estimate of $\|\bar u\|$:}} By It\^o's formula, one has 
\begin{equation*}
    \begin{split}
        d \|u\|^2 
    = &-2 \Big\langle\theta_\rho(\|u^1\|_{s-1})\theta_\kappa(\|\partial_z v^1 - \p_z v^1_0\|_{L^\infty})(u^1\partial_x u^1 + w^1\partial_z u^1 + \p_x p^1  )
       \\
       &\qquad - \theta_\rho(\|u^2\|_{s-1})\theta_\kappa(\|\partial_z v^2 - \p_z v^2_0\|_{L^\infty})(u^2\partial_x u^2 + w^2\partial_z u^2 + \p_x p^2  ), u\Big\rangle dt
   \\
   & + \Big\|\theta_\rho(\|u^1\|_{s-1})\theta_\kappa(\|\partial_z v^1 - \p_z v^1_0\|_{L^\infty})\sigma (u^1) 
    \\
    &\qquad - \theta_\rho(\|u^2\|_{s-1})\theta_\kappa(\|\partial_z v^2 - \p_z v^2_0\|_{L^\infty})\sigma (u^2) \Big\|^2_{L_2(\mathscr U, L^2)} dt
    \\
   & + 2\Big\langle \theta_\rho(\|u^1\|_{s-1})\theta_\kappa(\|\partial_z v^1 - \p_z v^1_0\|_{L^\infty})\sigma (u^1) 
    \\
    &\qquad - \theta_\rho(\|u^2\|_{s-1})\theta_\kappa(\|\partial_z v^2 - \p_z v^2_0\|_{L^\infty})\sigma (u^2), u \Big\rangle dW
    \\
    =: & I_1 dt + I_2 dt + I_3 dW.
    \end{split}
\end{equation*}
Fix $n \in \Nb$ and let $\eta_a$, $\eta_b$ be stopping times such that $0 \leq \eta_a \leq \eta_b \leq \eta^n$. Integrating in time and taking supremums, multiplying by $\mathds{1}_{\Omega_0}$ and taking the expected value lead to
\begin{equation*}
    \begin{split}
        \Eb  &\sup_{t \in \left[\eta_a, \eta_b\right]} \| \bar u(t) \|^2  \leq  \Eb \| \bar{u}(\eta_a) \|^2 + \Eb\int_{\eta_a}^{\eta_b} \mathds{1}_{\Omega_0} (|I_1| + |I_2|) dt + \Eb \sup_{t \in \left[\eta_a, \eta_b\right]} \left|\int_{\eta_a}^t \mathds{1}_{\Omega_0}I_3 dW\right|.
    \end{split}
\end{equation*}
Using the triangle inequality, the Cauchy-Schwarz inequality, the Sobolev inequality, the H\"older inequality, and Young's inequality, thanks to the Lipschitz property of $\theta_\rho$ and $\theta_\kappa$, by integration by parts, we have
    \begin{equation*}
    \begin{split}
      \Eb\int_{\eta_a}^{\eta_b} \mathds{1}_{\Omega_0}|I_1| dt \leq &C_{\rho,\kappa} \Eb\int_{\eta_a}^{\eta_b}  \|
      \bar u\|_{s-1}\|u^1\|^2_{\tilde s} \|\bar u\| dt 
      + C \Eb\int_{\eta_a}^{\eta_b}\|
      \bar u\|_{s-1}^2 (\|u^1\|_{\tilde{s}}+ \|u^2\|_{\tilde{s}}) dt
      \\
      \leq & C_{\rho,\kappa} \Eb\int_{\eta_a}^{\eta_b} \|
      \bar u\|_{\widehat{s-1}}^2 (1+\|u^1\|^2_{\tilde{s}}+ \|u^2\|^2_{\tilde{s}}) dt.
    \end{split}
    \end{equation*}
For the term $I_2$, thanks to Assumption~\eqref{noise-ine}, by a similar calculation we have
    \begin{equation*}
        \begin{split}
            \Eb\int_{\eta_a}^{\eta_b} \mathds{1}_{\Omega_0}|I_2| dt \leq C_{\rho,\kappa}\Eb\int_{\eta_a}^{\eta_b}\|
      \bar u\|_{\widehat{s-1}}^2
      (1+\|u^1\|^2_{\tilde{s}}+ \|u^2\|^2_{\tilde{s}})  dt.
        \end{split}
    \end{equation*}
Regarding the term $I_3$, using the Burkholder-Davis-Gundy inequality and Young's inequality,
    \begin{equation*}
        \begin{split}
            &\Eb \sup_{t \in \left[\eta_a, \eta_b\right]} \left|\int_{\eta_a}^t \mathds{1}_{\Omega_0} I_3 dW\right| 
            \\
            \leq &C_{\rho,\kappa} \Eb \left( \int_{\eta_a}^{\eta_b} \|
      \bar u\|_{\widehat{s-1}}^2(1+\|u^1\|^2_{\tilde{s}}+ \|u^2\|^2_{\tilde{s}}) \| \bar{u} \|^2  dt \right)^{1/2}
      \\
      \leq & \frac1{2} \Eb \sup_{t \in \left[\eta_a, \eta_b\right]} \| \bar{u}(t) \|^2 + C_{\rho,\kappa} \Eb \int_{\eta_a}^{\eta_b} \|
      \bar u\|_{\widehat{s-1}}^2(1+\|u^1\|^2_{\tilde{s}}+ \|u^2\|^2_{\tilde{s}}) dt.
        \end{split}
    \end{equation*}
Therefore, one gets
\begin{equation}\label{uni-l2}
    \begin{split}
        \Eb  \sup_{t \in \left[\eta_a, \eta_b\right]} \| \bar u(t) \|^2  \leq  2\Eb \| \bar{u}(\eta_a) \|^2 +
       C_{\rho,\kappa} \Eb \int_{\eta_a}^{\eta_b} \|
      \bar u\|_{\widehat{s-1}}^2(1+\|u^1\|^2_{\tilde{s}}+ \|u^2\|^2_{\tilde{s}}) dt.
    \end{split}
\end{equation}

\medskip

\noindent\underline{\textbf{Estimates of $\|D^\alpha \bar v\|$ for $D^\alpha \neq \p_x^{s-1}$:}} For $0\leq |\alpha|\leq s-1$ and $D^\alpha \neq \p_x^{s-1}$, by It\^o's formula, 
\begin{equation*}
    \begin{split}
        d \|D^\alpha v\|^2 
    = &-2 \Big\langle\theta_\rho(\|u^1\|_{s-1})\theta_\kappa(\|\partial_z v^1 - \p_z v^1_0\|_{L^\infty})D^\alpha(u^1\partial_x v^1 + w^1\partial_z v^1 )  
       \\
       &\qquad - \theta_\rho(\|u^2\|_{s-1})\theta_\kappa(\|\partial_z v^2 - \p_z v^2_0\|_{L^\infty})D^\alpha(u^2\partial_x v^2 + w^2\partial_z v^2 ), D^\alpha v\Big\rangle dt
   \\
   & + \Big\|\theta_\rho(\|u^1\|_{s-1})\theta_\kappa(\|\partial_z v^1 - \p_z v^1_0\|_{L^\infty})D^\alpha \p_z \sigma (u^1) 
    \\
    &\qquad - \theta_\rho(\|u^2\|_{s-1})\theta_\kappa(\|\partial_z v^2 - \p_z v^2_0\|_{L^\infty})D^\alpha \p_z \sigma (u^2) \Big\|^2_{L_2(\mathscr U, L^2)} dt
    \\
   & + 2\Big\langle \theta_\rho(\|u^1\|_{s-1})\theta_\kappa(\|\partial_z v^1 - \p_z v^1_0\|_{L^\infty})D^\alpha \p_z\sigma (u^1) 
    \\
    &\qquad - \theta_\rho(\|u^2\|_{s-1})\theta_\kappa(\|\partial_z v^2 - \p_z v^2_0\|_{L^\infty})D^\alpha \p_z\sigma (u^2), D^\alpha v \Big\rangle dW
    \\
    =: & I_1 dt + I_2 dt + I_3 dW.
    \end{split}
\end{equation*}
Fix $n \in \Nb$ and let $\eta_a$, $\eta_b$ be stopping times such that $0 \leq \eta_a \leq \eta_b \leq \eta^n$. Integrating in time over $[\eta_a,\eta_b]$ and taking the supremum, multiplying by $\mathds{1}_{\Omega_0}$ and taking the expected value lead to
\begin{equation*}
    \begin{split}
        \Eb  &\sup_{t \in \left[\eta_a, \eta_b\right]} \| D^\alpha\bar  v(t) \|^2  \leq  \Eb \| D^\alpha \bar{v}(\eta_a) \|^2 + \Eb\int_{\eta_a}^{\eta_b} \mathds{1}_{\Omega_0} (|I_1| + |I_2|) dt + \Eb \sup_{t \in \left[\eta_a, \eta_b\right]} \left|\int_{\eta_a}^t \mathds{1}_{\Omega_0}I_3 dW\right|.
    \end{split}
\end{equation*}
Using the triangle inequality, the Cauchy-Schwarz inequality, the Sobolev inequality, the H\"older inequality, and Young's inequality, thanks to the Lipschitz property of $\theta_\rho$ and $\theta_\kappa$, by integration by parts, we have
    \begin{equation}\label{unique-noncritical-1}
    \begin{split}
      \Eb\int_{\eta_a}^{\eta_b} \mathds{1}_{\Omega_0}|I_1| dt \leq &C_{\rho,\kappa} \Eb\int_{\eta_a}^{\eta_b}  \|
      \bar u\|_{s-1}\|u^1\|^2_{\tilde s} \|v\|_{\hat H^{s-1}} dt +
      C \Eb\int_{\eta_a}^{\eta_b}\|
      \bar u\|_{s-1}^2 (\|u^1\|_{\tilde{s}}+ \|u^2\|_{\tilde{s}}) dt
      \\
      \leq & C_{\rho,\kappa} \Eb\int_{\eta_a}^{\eta_b} \|
      \bar u\|_{\widehat{s-1}}^2 (1+\|u^1\|^2_{\tilde{s}}+ \|u^2\|^2_{\tilde{s}}) dt.
    \end{split}
    \end{equation}
For the term $I_2$, thanks to Assumption~\eqref{noise-ine}, by a similar calculation we have
    \begin{equation}\label{unique-noncritical-2}
        \begin{split}
            \Eb\int_{\eta_a}^{\eta_b} \mathds{1}_{\Omega_0}|I_2| dt \leq  C_{\rho,\kappa} \Eb\int_{\eta_a}^{\eta_b} \|
      \bar u\|_{\widehat{s-1}}^2 (1+\|u^1\|^2_{\tilde{s}}+ \|u^2\|^2_{\tilde{s}}) dt.
        \end{split}
    \end{equation}
Regarding the term $I_3$, using the Burkholder-Davis-Gundy inequality yields
    \begin{equation}\label{unique-noncritical-3}
        \begin{split}
            &\Eb \sup_{t \in \left[\eta_a, \eta_b\right]} \left|\int_{\eta_a}^t \mathds{1}_{\Omega_0} I_3 dW\right| 
            \\
            \leq &C_{\rho,\kappa} \Eb \left( \int_{\eta_a}^{\eta_b} \|
      \bar u\|_{\widehat{s-1}}^2  (1+\|u^1\|^2_{\tilde{s}}+ \|u^2\|^2_{\tilde{s}}) \| D^\alpha \bar v \|^2  dt \right)^{1/2}
      \\
      \leq & \frac1{2} \Eb \sup_{t \in \left[\eta_a, \eta_b\right]} \| D^\alpha \bar{v}(t) \|^2 + C_{\rho,\kappa} \Eb \int_{\eta_a}^{\eta_b} \|
      \bar u\|_{\widehat{s-1}}^2 (1+\|u^1\|^2_{\tilde{s}}+ \|u^2\|^2_{\tilde{s}}) dt.
        \end{split}
    \end{equation}
Therefore, one gets
\begin{equation}\label{uni-higher}
    \begin{split}
        \Eb  \sup_{t \in \left[\eta_a, \eta_b\right]} \| D^\alpha \bar v(t) \|^2  
        \leq  2\Eb \| D^\alpha\bar{v}(\eta_a) \|^2+
       C_{\rho,\kappa} \Eb \int_{\eta_a}^{\eta_b} \|
      \bar u\|_{\widehat{s-1}}^2(1+\|u^1\|^2_{\tilde{s}}+ \|u^2\|^2_{\tilde{s}}) dt.
    \end{split}
\end{equation}

\medskip

\noindent\underline{\textbf{Estimate of $\displaystyle \lnorm\frac{\p_x^{s-1} \bar v}{\sqrt{\p_z v^1}}\rnorm$:}} When $D^\alpha = \partial_x^{s-1}$, in order to get the equation of $\displaystyle d\lnorm\frac{\p_x^{s-1} v}{\sqrt{\p_z v^1}}\rnorm^2$, we first find
\begin{equation*}
\begin{split}
    d \partial_x^{s-1} v =
    &-\Big[\theta_\rho(\|u^1\|_{s-1})\theta_\kappa(\|\partial_z v^1 - \p_z v^1_0\|_{L^\infty})\p_x^{s-1}(u^1\partial_x v^1 + w^1\partial_z v^1   )
       \\
       &\qquad \qquad - \theta_\rho(\|u^2\|_{s-1})\theta_\kappa(\|\partial_z v^2 - \p_z v^2_0\|_{L^\infty})\p_x^{s-1}(u^2\partial_x v^2 + w^2\partial_z v^2   )\Big] dt
    \\
    & +\Big[\theta_\rho(\|u^1\|_{s-1})\theta_\kappa(\|\partial_z v^1 - \p_z v^1_0\|_{L^\infty})\p_x^{s-1}\p_z\sigma (u^1) 
    \\
    &\qquad - \theta_\rho(\|u^2\|_{s-1})\theta_\kappa(\|\partial_z v^2 - \p_z v^2_0\|_{L^\infty})\p_x^{s-1}\p_z\sigma (u^2) \Big] dW
    \\
    =& -\left(\theta_\rho(\|u^1\|_{s-1}) -\theta_\rho(\|u^2\|_{s-1}) \right)\theta_\kappa(\|\partial_z v^1 - \p_z v^1_0\|_{L^\infty})\p_x^{s-1} (u^1\partial_x v^1 + w^1\partial_z v^1   ) dt
    \\
    &-\theta_\rho(\|u^2\|_{s-1})\left(\theta_\kappa(\|\partial_z v^1 - \p_z v^1_0\|_{L^\infty}) - \theta_\kappa (\|\partial_z v^2 - \p_z v^2_0\|_{L^\infty}) \right)\p_x^{s-1}(u^1\partial_x v^1 + w^1\partial_z v^1   ) dt
    \\
    &-\theta_\rho(\|u^2\|_{s-1})\theta_\kappa(\|\partial_z v^2 - \p_z v^2_0\|_{L^\infty}) \p_x^{s-1}\left(u\p_x v^1+ u^2 \p_x v + w \p_z v^1 + w^2 \p_z v\right) dt
    \\
    &+\left(\theta_\rho(\|u^1\|_{s-1}) -\theta_\rho(\|u^2\|_{s-1}) \right)\theta_\kappa(\|\partial_z v^1 - \p_z v^1_0\|_{L^\infty}) \partial_x^{s-1}\partial_z \sigma (u^1) dW
    \\
    &+ \theta_\rho(\|u^2\|_{s-1})\left(\theta_\kappa(\|\partial_z v^1 - \p_z v^1_0\|_{L^\infty}) - \theta_\kappa(\|\partial_z v^2 - \p_z v^2_0\|_{L^\infty}) \right) \partial_x^{s-1}\partial_z \sigma (u^1) dW
    \\
    &+ \theta_\rho(\|u^2\|_{s-1})\theta_\kappa(\|\partial_z v^2 - \p_z v^2_0\|_{L^\infty}) \left(\partial_x^{s-1}\partial_z \sigma (u^1) - \partial_x^{s-1}\partial_z \sigma (u^2) \right)dW
    \\
    =: &( A_1 + A_2  + A_3) dt + (A_4+A_5+A_6)dW,
\end{split}
\end{equation*}
and
\begin{equation*}
\begin{split}
      &d \p_z v^1 = -\theta_\rho(\|u^1\|_{s-1})\theta_\kappa(\|\partial_z v^1 - \p_z v^1_0\|_{L^\infty})(\p_z u^1 \p_x v^1 + u^1 \p_{xz} v^1 + \p_z w^1 \p_z v^1 + w^1\p_{zz} v^1)dt
    \\
     &+\theta_\rho(\|u^1\|_{s-1})\theta_\kappa(\|\partial_z v^1 - \p_z v^1_0\|_{L^\infty})\partial_{zz} \sigma (u^1) dW=: B_1 dt + B_2 dW.
\end{split}
\end{equation*}
By It\^o's formula, we have 
\begin{equation*}
    \begin{split}
        &d\lnorm\frac{\p_x^{s-1} v}{\sqrt{\p_z v^1}}\rnorm^2 = \left(\left\langle 2(A_1+A_2+A_3), \frac{\p_x^{s-1} v}{\p_z v^1} \right\rangle - \left\langle B_1,\frac{|\p_x^{s-1} v|^2}{|\p_z v^1|^2} \right\rangle\right)dt
        \\
        &+\frac12\left(\left\langle 2(A_4+A_5+A_6)^2,\frac1{\p_z v^1} \right\rangle - \left\langle 4\frac{\p_x^{s-1}v}{|\p_z v^1|^2} ,(A_4+A_5+A_6)B_2 \right\rangle +  \left\langle 2\frac{|\p_x^{s-1}v|^2}{|\p_z v^1|^3} ,B_2^2 \right\rangle \right) dt
        \\
        &+ \left(\left\langle 2\frac{\p_x^{s-1}v}{\p_z v^1}, (A_4+A_5+A_6) \right\rangle - \left\langle \frac{|\p_x^{s-1}v|^2}{|\p_z v^1|^2}, B_2 \right\rangle  \right)dW.
    \end{split}
\end{equation*}
Fix $n \in \Nb$ and let $\eta_a$, $\eta_b$ be stopping times such that $0 \leq \eta_a \leq \eta_b \leq \eta^n$. Integrating in time over $[\tau_a, \tau_b]$ and taking the supremum, multiplying by $\mathds{1}_{\Omega_0}$ and taking the expected value lead to
\begin{equation*}
    \begin{split}
        \Eb  &\sup_{t \in \left[\eta_a, \eta_b\right]} \lnorm\frac{\p_x^{s-1} \bar v}{\sqrt{\p_z v^1}}(t)\rnorm^2  \leq  \Eb \lnorm \frac{\p_x^{s-1} \bar v}{\sqrt{\p_z v^1}}(\eta_a) \rnorm^2 + \Eb\int_{\eta_a}^{\eta_b} \mathds{1}_{\Omega_0} \left|\left\langle 2(A_1+A_2+A_3), \frac{\p_x^{s-1} v}{\p_z v^1} \right\rangle - \left\langle B_1,\frac{|\p_x^{s-1} v|^2}{|\p_z v^1|^2} \right\rangle\right| dt 
        \\
        &+ \Eb\int_{\eta_a}^{\eta_b} \mathds{1}_{\Omega_0} \frac12\left|\left\langle 2(A_4+A_5+A_6)^2,\frac1{\p_z v^1} \right\rangle - \left\langle 4\frac{\p_x^{s-1}v}{|\p_z v^1|^2} ,(A_4+A_5+A_6)B_2 \right\rangle +  \left\langle 2\frac{|\p_x^{s-1}v|^2}{|\p_z v^1|^3} ,B_2^2 \right\rangle \right| dt
        \\
        &+ \Eb \sup_{t \in \left[\eta_a, \eta_b\right]} \left|\int_{\eta_a}^t \mathds{1}_{\Omega_0} \left(\left\langle 2\frac{\p_x^{s-1}v}{\p_z v^1}, (A_4+A_5+A_6) \right\rangle - \left\langle \frac{|\p_x^{s-1}v|^2}{|\p_z v^1|^2}, B_2 \right\rangle  \right) dW\right|.
    \end{split}
\end{equation*}
Using the triangle inequality, the Cauchy-Schwarz inequality, the Sobolev inequality, the H\"older inequality, Young's inequality, and Assumption~\eqref{noise-ine}, thanks to the Lipschitz property of $\theta_\rho$ and $\theta_\kappa$,  by integration by parts and \eqref{PE-inviscid-3}, we have
\begin{itemize}
    \item For $\left\langle A_1+A_2, \frac{\p_x^{s-1} v}{\p_z v^1} \right\rangle$,
    \begin{equation*}
    \begin{split}
        \left|\left\langle A_1+A_2, \frac{\p_x^{s-1} v}{\p_z v^1} \right\rangle\right| &\leq C_{\rho,\kappa} \|u\|_{s-1}\|u^1\|^2_{s}\lnorm\frac{\p_x^{s-1} v}{\p_z v^1}\rnorm
        \\
        &\leq C_{\rho,\kappa} \left(1+\lnorm \sqrt{\p_z v^1}\rnorm_{L^\infty}^2 + \frac{1}{\lnorm \sqrt{\p_z v^1}\rnorm_{L^\infty}^2}\right)\|u^1\|_{\tilde s}^2 \|u\|_{\widehat{s-1}}^2
        \\
        &\leq C_{\rho,\kappa}\|u^1\|_{\tilde s}^2 \|u\|_{\widehat{s-1}}^2.
    \end{split}
\end{equation*}

\item For $\left\langle A_3, \frac{\p_x^{s-1} v}{\p_z v^1} \right\rangle$, 
\begin{equation*}
    \begin{split}
      &\left\langle \p_x^{s-1}\left(u\p_x v^1  + w \p_z v^1 \right), \frac{\p_x^{s-1} v}{\p_z v^1} \right\rangle 
      \\
      =&\left\langle \p_x^{s-1}\left(u\p_x v^1\right) + \p_x^{s-1} w \p_z v^1 + \sum\limits_{k=0}^{s-2} \binom{s}{k} \partial_x^{s-k-1} \p_z v^1 \partial_x^{k} w, \frac{\p_x^{s-1} v}{\p_z v^1} \right\rangle
      \\
      =&\left\langle \p_x^{s-1}\left(u\p_x v^1\right)+ \sum\limits_{k=0}^{s-2} \binom{s}{k} \partial_x^{s-k-1} \p_z v^1 \partial_x^{k} w, \frac{\p_x^{s-1} v}{\p_z v^1} \right\rangle
      \\
      \leq &C_{\rho,\kappa}\|u^1\|_{\tilde s} \|u\|_{\widehat{s-1}}^2,
    \end{split}
\end{equation*}
where we have used the cancellation $\langle \p_x^{s-1} w \p_z v^1 , \frac{\p_x^{s-1} v}{\p_z v^1}\rangle=0$. Next, by integration by parts,
\begin{equation*}
    \begin{split}
      &\left\langle \p_x^{s-1}\left(u^2\p_x v + w^2 \p_z v \right), \frac{\p_x^{s-1} v}{\p_z v^1} \right\rangle 
      \\
      =& \int \frac{1}{2}\frac{|\p_x^{s-1} v|^2}{|\p_z v^1|^2}(u^2\p_{xz} v^1 + w^2\p_{zz} v^1) + \frac{\p_x^{s-1} v}{\p_z v^1} \sum\limits_{k=0}^{s-2} \binom{s}{k} (\partial_x^{s-k-1} u^2 \partial_x^{k+1} v + \partial_x^{s-k-1} w^2 \partial_x^{k} \p_z v )dxdz 
      \\
      \leq & C_{\rho,\kappa}
    \|u\|_{\widehat{s-1}}^2(\|u^1\|_{\tilde s}^2+\|u^2\|_{\tilde s}^2).
    \end{split}
\end{equation*}
Therefore, we have
\begin{equation*}
    \begin{split}
        \left\langle A_3, \frac{\p_x^{s-1} v}{\p_z v^1} \right\rangle \leq &C_{\rho,\kappa}\|u\|_{\widehat{s-1}}^2(1+\|u^1\|_{\tilde s}^2+\|u^2\|_{\tilde s}^2).
    \end{split}
\end{equation*}

\item For $\left\langle B_1,\frac{|\p_x^{s-1} v|^2}{|\p_z v^1|^2} \right\rangle$,
\begin{equation*}
\begin{split}
   \left|\left\langle B_1,\frac{|\p_x^{s-1} v|^2}{|\p_z v^1|^2} \right\rangle\right| \leq  C_{\rho,\kappa}\|u^1\|_{\tilde s}^2 \|u\|_{\widehat{s-1}}^2.
\end{split}
\end{equation*}

\item For the rest of $dt$ terms, we have
\begin{equation*}
    \begin{split}
        &\left|\mathds{1}_{\Omega_0}\frac12\left(\left\langle 2(A_4+A_5+A_6)^2,\frac1{\p_z v^1} \right\rangle - \left\langle 4\frac{\p_x^{s-1}v}{|\p_z v^1|^2} ,(A_4+A_5+A_6)B_2 \right\rangle +  \left\langle 2\frac{|\p_x^{s-1}v|^2}{|\p_z v^1|^3} ,B_2^2 \right\rangle \right)\right|
        \\
        \leq & C_{\rho,\kappa}\|u\|^2_{\widehat{s-1}}(1+\|u^1\|^2_{\tilde s}).
    \end{split}
\end{equation*}

\item For the $dW$ terms, by the Burkholder-Davis-Gundy inequality,
    \begin{equation*}
        \begin{split}
            &\Eb \sup_{t \in \left[\eta_a, \eta_b\right]} \left|\int_{\eta_a}^t \mathds{1}_{\Omega_0} \left(\left\langle 2\frac{\p_x^{s-1}v}{\p_z v^1}, (A_4+A_5+A_6) \right\rangle - \left\langle \frac{|\p_x^{s-1}v|^2}{|\p_z v^1|^2}, B_2 \right\rangle  \right) dW\right|
            \\
            \leq &C_{\rho,\kappa} \Eb \left( \int_{\eta_a}^{\eta_b}\|
      \bar u\|^2_{\widehat{s-1}} \left(1+\lnorm \sqrt{\p_z v^1}\rnorm_{L^\infty}^2 + \frac{1}{\lnorm \sqrt{\p_z v^1}\rnorm_{L^\infty}^2}\right) (1+\|u^1\|_{\tilde s}^2) \lnorm\frac{\p_x^{s-1}
      \bar v}{\sqrt{\p_z v^1}}\rnorm^2  dt \right)^{1/2}
      \\
      \leq & \frac1{2} \Eb \sup_{t \in \left[\eta_a, \eta_b\right]}  \lnorm\frac{\p_x^{s-1}\bar v}{\sqrt{\p_z v^1}}(t)\rnorm^2 + C_{\rho,\kappa} \Eb \int_{\eta_a}^{\eta_b} \|
      \bar u\|^2_{\widehat{s-1}} \left(1+\|u^1\|_{\tilde s}^2\right) dt.
        \end{split}
    \end{equation*}
\end{itemize}
Combining all the estimates above brings that 
\begin{equation}\label{uni-critical}
    \begin{split}
       \Eb  \sup_{t \in \left[\eta_a, \eta_b\right]} \lnorm\frac{\p_x^{s-1} \bar v}{\sqrt{\p_z v^1}}(t)\rnorm^2  \leq  2\Eb \lnorm \frac{\p_x^{s-1} \bar v}{\sqrt{\p_z v^1}}(\eta_a) \rnorm^2 
       + C_{\rho,\kappa}\Eb \int_{\eta_a}^{\eta_b}(1+\|u^1\|_{\tilde s}^2+ \|u^2\|_{\tilde s}^2)\|\bar u\|^2_{\widehat{s-1}} dt .
    \end{split}
\end{equation}

\medskip

\noindent\underline{\textbf{Combining the estimates:}} Based on \eqref{uni-l2}, \eqref{uni-higher} and \eqref{uni-critical}, we have
\begin{equation}
    \begin{split}
        \Eb  \sup_{t \in \left[\eta_a, \eta_b\right]} \lnorm\bar u(t)\rnorm^2_{\widehat{s-1}} \leq  C\Eb  \lnorm\bar u(\eta_a)\rnorm^2_{\widehat{s-1}}
        + C_{\rho,\kappa}\Eb \int_{\eta_a}^{\eta_b}(1+\|u^1\|_{\tilde s}^2+ \|u^2\|_{\tilde s}^2)\|\bar u\|^2_{\widehat{s-1}} dt.
    \end{split}
\end{equation}
As $0\leq\eta_a\leq \eta_b\leq \eta^n$ are arbitrary, then the result follows from the stochastic Gr\"onwall Lemma (see \cite[Lemma 5.3]{glatt2009strong}). 
\end{proof}

\section{Proof of Theorem \ref{thm:main-1}}\label{sec:proof-thm}

Now we are ready to prove the first main result, Theorem \ref{thm:main-1}, which concerns the existence of local maximal pathwise solutions to system \eqref{PE-inviscid-system}. The ultimate goal is to provide pathwise solutions with $L^2(\Omega; \mathcal D_{s, 2\kappa})$ initial data, and the proof will be split into three subsections.

\subsection{Step 1: Local solution with weaker regularity and bounded initial data}\label{sec:step1}

We start by assuming that 
\begin{equation}\label{ic-requirement}
    \Pb\left(u_0\in \mathcal D_{s,2\kappa}: \|u_0\|_{\tilde s} < \frac M2 \right) = 1,
\end{equation}
 i.e., $\|u_0\|_{\tilde s} < \frac M2$ a.s..
Following from \cite[Section 5.2]{debussche2011local} and the Gy\"ongy-Krylov theorem \cite{gyongy1996existence}, together with Proposition \ref{prop:pathwise.uniqueness} and the proof of convergence in Proposition \ref{prop:global.martingale.existence}, one can obtain a unique local pathwise solution $(u,\eta)$ to the original system \eqref{PE-inviscid-system} satisfying the weaker regularity:
\begin{equation}\label{path-regularity-old}
	u\left( \cdot \wedge \eta \right) \in L^2 \left(\Omega; C\left( [0, T], \mathcal D_{s-1,\kappa} \right)  \right), \quad \mathds{1}_{[0, \eta]}(\cdot) u(\cdot)  \in L^2\left( \Omega; L^\infty\left( 0, T; \mathcal D_{s,\kappa} \right) \right).
\end{equation}
Here $\eta$ is defined in \eqref{stopingtime:eta} and $\rho \geq M+4$. Note that $\rho\geq M+4$ can guarantee that $\eta >0 $ $\mathbb P$-a.s.. 

The solution $(u,\eta)$ obtained above only satisfies \eqref{path-regularity-old}. Therefore, we do not yet refer to \((u,\eta)\) as a local pathwise solution in the sense of Definition~\ref{def:pathwise.sol}, since the continuity of \(u\) is only known in \(\mathcal D_{s-1,\kappa}\), rather than in \(\mathcal D_{s,\kappa}\). This regularity is not sufficient for the localization argument later in Section~\ref{sec:proof-localization},  to extend the existence result to $L^2$ initial data. To overcome this issue, we show in the next subsection that the solution is indeed continuous in time in the space $\mathcal D_{s,\kappa}$, \[u\left( \cdot \wedge \tau \right) \in L^2 \left(\Omega; C\left( [0, T], \mathcal \mathcal D_{s,\kappa} \right)  \right),  \text{ for some stopping time } \tau. \]

\subsection{Step 2: Improved regularity of the solution}\label{sec:step2}

The key result is an abstract Cauchy lemma (Lemma~\ref{lemma:cauchy}). Note that although \cite[Lemma~5.1]{glatt2009strong} and \cite[Lemma~7.1]{glatt2014local} have had similar results, the norms used in our lemma are quite different and thus their proofs do not apply directly. As a result, we shall provide complete proof with full details below.

\subsubsection{The smoothing operator $P_n$}
We will consider Fourier truncation in the periodic $x$ variable and extension–mollification–restriction in the vertical $z$ variable. The treatment in the $z$ variable is inspired and similar to that in \cite{glatt2014local}. To be more specific, let $S_n$ be the Fourier truncation in the periodic $x$ variable,
\[
(S_n f)(x,z):=\sum_{|k|\le n}\widehat f_k(z)e^{2\pi i kx},
\qquad
\widehat f_k(z):=\int_{\mathbb T} f(x,z)e^{-2\pi i kx}\,dx.
\]
For $m\in \mathbb N$, let
\[
E:H^m(0,1)\to H^m(\mathbb R), \qquad
R:H^m(\mathbb R)\to H^m(0,1)
\]
be a bounded extension operator and the restriction operator, and let $\widetilde F_{\varepsilon}$ denote a standard mollifier
on $\mathbb R$, acting in the $z$ variable only. Define, for $f=f(x,z)$,
\[
(E_z f)(x,\cdot):=E(f(x,\cdot)),
\qquad
(R_z g)(x,\cdot):=R(g(x,\cdot)),
\]
and
\[
(I f)(x,z):=f(x,z)-\int_0^1 f(x,\zeta)\,d\zeta.
\]
Finally define
\begin{align}\label{Pn}
   J_n:=I\,R_z\,\widetilde F_{1/n}\,E_z,
\qquad
P_n:=S_nJ_n=J_nS_n. 
\end{align}
Here $S_n$ and $J_n$ commute since $S_n$ acts only on $x$ while $J_n$ acts only on $z$. We have the following properties of the smoothing operator $P_n$, which will be used in the rest of this section.

\begin{lemma}
For every integer $m\ge 0$, there exists a constant
$C_m>0$, independent of $n$, such that for every $f\in H^m\cap H$,
\begin{align}
\|P_n f\|_{H^m} &\leq C_m \|f\|_{H^m}, \label{eq:PnHs}\\
\|P_n f\|_{H^{m+1}} &\leq C_m n \|f\|_{H^m}, \label{eq:PnHs1}\\
\|P_n f-f\|_{H^{m-1}} &\leq C_m \frac1n\|f\|_{H^m}, \label{eq:PndiffHsminus1} \quad \text{if }  m\geq 1,\\
\lim_{n\to\infty}\|P_n f-f\|_{H^m} &= 0, \label{eq:PnconvHs}\\
\lim_{n\to\infty}n\|P_n f-f\|_{H^{m-1}} &= 0 \quad \text{if }  m\geq 1. \label{eq:PnconvHsminus1}
\end{align}
\end{lemma}

\begin{proof}

We first note that $I$ is a bounded operator on $H^m$ for every integer $m\ge 0$.
Therefore $J_n$ and thus $P_n$ map $H$ into $H$.

We now review the properties of $J_n$. Since $(0,1)$ is a smooth bounded interval, and noting that $I$ is a bounded operator,
the standard extension--mollification--restriction construction on $(0,1)$ has the
same properties as the smoothing operator in \cite[Appendix A]{glatt2014local}: it is uniformly bounded
on Sobolev spaces (A.2), gains one derivative with factor $n$ (A.3), tail estimate (A.4), converges strongly (A.5), and
satisfies the stronger $n$-weighted error estimate (A.6). Applying these one-dimensional
bounds fiberwise in $x$ to $\partial_x^a f(x,\cdot)$, we obtain for all integers
$a,b\ge 0$:
\begin{align}
\|\partial_x^a J_n f\|_{L_x^2 H_z^b}
&\le C_b \|\partial_x^a f\|_{L_x^2 H_z^b}, \label{eq:Jn1}\\
\|\partial_x^a J_n f\|_{L_x^2 H_z^{b+1}}
&\le C_b n \|\partial_x^a f\|_{L_x^2 H_z^b}, \label{eq:Jn2}\\
\|\partial_x^a (J_n f-f)\|_{L_x^2 H_z^{b}}
&\le C_b \frac1n \|\partial_x^a f\|_{L_x^2 H_z^{b+1}}, \label{eq:Jn3}\\
\lim_{n\to \infty}\|\partial_x^a(J_n f-f)\|_{L_x^2 H_z^b}
&= 0 \quad\text{whenever }\partial_x^a f\in L_x^2 H_z^{b}, \label{eq:Jn4}\\
\lim_{n\to \infty}n\|\partial_x^a(J_n f-f)\|_{L_x^2 H_z^b}
&= 0
\quad\text{whenever }\partial_x^a f\in L_x^2 H_z^{b+1}. \label{eq:Jn5}
\end{align}
Here
\[
\|u\|_{L_x^2 H_z^b}^2:=\int_{\mathbb T}\|u(x,\cdot)\|_{H^b(0,1)}^2\,dx.
\]

Next, we recall the standard properties of the Fourier cutoff $S_n$ in the periodic
variable $x$:
\begin{align}
\|S_n g\|_{H^m} &\le \|g\|_{H^m}, \label{eq:Sn1}\\
\lim_{n\to \infty}\|S_n g-g\|_{H^m} &= 0. \label{eq:Sn3}
\end{align}
Moreover, for every integer $m\ge 1$,
\begin{align}
\|S_ng - g\|_{H^{m-1}} &\leq C_m \frac1n \|g\|_{H^m}, \label{eq:Sn2}\\
\lim_{n\to \infty}n\|S_n g-g\|_{H^{m-1}} &= 0, \label{eq:Sn4}
\end{align}
for all $g\in H^m$. To prove \eqref{eq:Sn2} and \eqref{eq:Sn4}, it is enough to use the
equivalent mixed-derivative norm on $H^{m-1}$: for each pair of integers
$a,b\ge 0$ with $a+b\le m-1$,
\[
n^2\|\partial_x^a\partial_z^b(S_n g-g)\|_{L^2}^2
=
\sum_{|k|>n} n^2 (2\pi k)^{2a}\|\partial_z^b \widehat g_k\|_{L^2(0,1)}^2.
\]
Since $|k|>n$ implies $n^2\le k^2$, we obtain
\[
n^2 (2\pi k)^{2a}\le (2\pi k)^{2(a+1)},
\]
hence
\[
n^2\|\partial_x^a\partial_z^b(S_n g-g)\|_{L^2}^2
\le
\sum_{|k|>n}(2\pi k)^{2(a+1)}\|\partial_z^b \widehat g_k\|_{L^2(0,1)}^2.
\]
Because $(a+1)+b\le m$, the right-hand side is bounded by the tail of the
$H^m$ norm of $g$, and therefore tends to $0$ as $n\to\infty$, which gives \eqref{eq:Sn4}. One can also directly bound the right-side by $C\|g\|_{H^m}^2$, which gives \eqref{eq:Sn2}.

We now prove \eqref{eq:PnHs}--\eqref{eq:PnconvHsminus1}.

For \eqref{eq:PnHs}, using the equivalent mixed-derivative norm on $H^m$,
for every $a,b\ge 0$ with $a+b\le m$ we have
\[
\|\partial_x^a\partial_z^b P_n f\|_{L^2}
=
\|S_n(\partial_x^a\partial_z^b J_n f)\|_{L^2}
\le
\|\partial_x^a\partial_z^b J_n f\|_{L^2}.
\]
Here we have used the fact that $S_n$ commutes with $\partial_x^a$ and $\partial_z^b$.
Summing over $a+b\le m$ and using \eqref{eq:Jn1}, we obtain
\[
\|P_n f\|_{H^m}\le C_m \|f\|_{H^m}.
\]

For \eqref{eq:PnHs1}, let $a,b\ge 0$ with $a+b\le m+1$.
If $a\ge 1$, then
$
\partial_x^a\partial_z^b P_n f
=
\partial_x S_n(\partial_x^{a-1}\partial_z^b J_n f).
$
Since $S_n$ contains only frequencies $|k|\le n$,
\[
\|\partial_x^a\partial_z^b P_n f\|_{L^2}
\le
n \|\partial_x^{a-1}\partial_z^b J_n f\|_{L^2}
\le
C_m n \|f\|_{H^m},
\]
where we used \eqref{eq:Jn1} and the fact that $(a-1)+b\le m$.
If $a=0$, then $b\le m+1$. When $b\geq 1$, \eqref{eq:Jn2} gives
\[
\|\partial_z^b P_n f\|_{L^2}
\le
\|\partial_z^b J_n f\|_{L^2}
\le
C_m n \|f\|_{H^m}, 
\]
while the case $b=0$ is trivial since \eqref{eq:PnHs} implies
\[
\|P_n f\|_{L^2}\le C\|f\|_{L^2}\le C n \|f\|_{L^2},
\qquad n\ge 1.
\]
Summing over all $a+b\le m+1$ proves \eqref{eq:PnHs1}.

For \eqref{eq:PnconvHs}, we write
\[
P_n f-f
=
S_n(J_n f-f)+(S_n f-f).
\]
Hence, by \eqref{eq:Sn1},
\[
\|P_n f-f\|_{H^m}
\le
\|J_n f-f\|_{H^m}+\|S_n f-f\|_{H^m}.
\]
The first term tends to $0$ by \eqref{eq:Jn4}, and the second by \eqref{eq:Sn3},
which proves \eqref{eq:PnconvHs}.

For \eqref{eq:PnconvHsminus1}, the same decomposition yields
\begin{equation}\label{eq:tail-est}
    n\|P_n f-f\|_{H^{m-1}}
\le
n\|J_n f-f\|_{H^{m-1}}
+
n\|S_n f-f\|_{H^{m-1}}.
\end{equation}
The second term tends to $0$ by \eqref{eq:Sn4}. For the first term, if
$a+b\le m-1$, then $a+(b+1)\le m$, so $\partial_x^a f\in L_x^2H_z^{b+1}$.
Therefore \eqref{eq:Jn5} applies to each mixed derivative of order at most $m-1$,
and summing over all such derivatives gives
\[
n\|J_n f-f\|_{H^{m-1}}\to 0.
\]

Finally, for \eqref{eq:PndiffHsminus1}, both terms in the right-side of \eqref{eq:tail-est} are bounded by $C_m \|f\|_{H^m}$ due to \eqref{eq:Jn3} and \eqref{eq:Sn2}, respectively. This proves \eqref{eq:PndiffHsminus1} and concludes the proof.
\end{proof}

\subsubsection{An abstract Cauchy lemma}
We first define several notations that will be used in the lemma. 

Let $u_0$ satisfies \eqref{ic-requirement}, and let $u_0^j = P_j u_0$ where $P_j$ are smoothing operators defined in \eqref{Pn}. By the definition of $P_j$, one has
\[
\|u_0^j\|_{m}\leq C_m\|u_0\|_{m} \quad \text{for any } 0\leq m\leq s. 
\]
Then by \eqref{eq:PndiffHsminus1} and the Sobolev inequality, since $s\geq 6$ we have
\begin{align*}
    \|\partial_{zz} u_0^j - \p_{zz} u_0\|_{L^\infty} \leq C\|u_0^j- u_0\|_{H^4} = C\|(I-P_j) u_0\|_{H^4}  \leq \frac{C}{j} \|u_0\|_{5} \leq \frac{C_{\kappa}}{j} \|u_0\|_{\tilde s} \leq \frac{C_{\kappa}M}{2j}
\end{align*}
for some $C_\kappa$ depends on $\kappa$.
Therefore, one can choose $j$ large enough such that $\frac{C_\kappa M}{2j} \leq \frac12\kappa$ (for instance, $j\geq \frac{C_\kappa M}\kappa$), and consequently $\|\partial_{zz} u_0^j - \p_{zz} u_0\|_{L^\infty} \leq \frac12\kappa$ which gives $\frac32\kappa\leq\partial_{zz} u_0^j\leq \frac2{3\kappa}$. 

Denote by
$$ s' = s+1.$$
Then, for $j\geq \lceil \frac{C_\kappa M}\kappa \rceil=:J$, using \eqref{eq:PnHs1} yields
\[
u_0^j \in \mathcal D_{s',\frac32\kappa}, 
\quad  \quad\|u_0^j\|_{\tilde s'} \leq C_{\kappa,j,s} \|u_0\|_{\tilde s} \leq C_{\kappa,j,s} M \text{ a.s.},
\]
for some constant $C_{\kappa,j,s}$ depending only on $\kappa$, $j$, and $s$.

Now, with higher regularity $u_0^j \in \mathcal D_{s',\frac32\kappa}$, we consider the modified system \eqref{PE-modified-system} with cutoff function $\theta_\rho(\|u\|_{s'-1})$ and initial data $u_0^j$.
Let $(u^j, \eta^j)$ be the corresponding local pathwise solutions of the original system~\eqref{PE-inviscid-system} obtained from Section~\ref{sec:step1} satisfying \eqref{path-regularity-old}, where $\eta^j$ is defined as 
\begin{equation}\label{stopingtime:eta-j}
	\eta^j = \inf \left\lbrace t \geq 0: \,\| u^j(t)\|_{s'-1}\equiv\| u^j(t)\|_{s} \geq \frac\rho2 \right\rbrace \wedge  \inf \left\lbrace t \geq 0: \, \|\partial_{zz}  u^j(t) - \partial_{zz}  u^j_0 \|_{L^\infty} \geq \frac\kappa4 \right\rbrace.
\end{equation}
Note that \eqref{stopingtime:eta-j} differ from \eqref{stopingtime:eta} in the sense that it corresponds to the cut-off functions $\theta_\rho(\|u^j\|_{s})$ and $\theta_{\frac\kappa2}(\|\partial_{zz}  u^j - \partial_{zz}  u^j_0\|_{L^\infty})$ in the modified system \eqref{PE-modified-system} for $u^j$. Here we require
\begin{equation}\label{cauchy-rho}
    \rho \geq (1+\tilde{C}_\kappa)(\frac{C_s M}{\tilde{c}_\kappa} + 4 )(1+\frac1{\tilde{c}_\kappa}),
\end{equation}
where $\tilde{C}_\kappa$ and $\tilde{c}_\kappa$ appear in \eqref{ctildekappa}. Since 
\[
\|u_0^j\|_{s} \leq C_s\|u_0\|_s \leq \frac{C_s}{\tilde{c}_\kappa} \|u_0\|_{\tilde s}\leq \frac{MC_s}{2\tilde{c}_\kappa},
\]
such requirement can guarantee that $\frac\rho2 \geq \|u_0^j\|_{s}+2$ and therefore $\eta^j>0$ a.s.. Recall that $s'=s+1$, thus 
\begin{equation}\label{uj-regularity}
	u^j\left( \cdot \wedge \eta^j \right) \in L^2 \left(\Omega; C\left( [0, T], \mathcal D_{s,\kappa} \right)  \right), \quad \mathds{1}_{[0, \eta^j]}(\cdot) u^j(\cdot)  \in L^2\left( \Omega; L^\infty\left( 0, T; \mathcal D_{s+1,\kappa} \right) \right).
\end{equation}

Then, let us fix any $T>0$ and for $j\geq J$ define the sequence of stopping times
\begin{equation}\label{stopping-time-taujT}
    \tau_j^T := \inf \left\lbrace t\geq 0: \|u^j(t)\|_{\tilde s} \geq 2 + \|u_0^j\|_{\tilde s} \right\rbrace \wedge \inf \left\lbrace t \geq 0: \, \|\partial_{zz}  u^j(t) - \partial_{zz}  u^j_0 \|_{L^\infty} \geq \frac\kappa4 \right\rbrace \wedge T,
\end{equation}
and for $j,k\geq J$ let
\begin{equation}\label{taujk}
    \tau_{j,k}^T = \tau_j^T \wedge \tau_k^T.
\end{equation}
For any time $t\geq\eta^j$, since $\rho \geq (1+\tilde{C_\kappa})(\frac{C_sM}{\tilde{c_\kappa}} + 4)(1+\frac1{\tilde{c}_\kappa})$, we know that
\begin{align*}
    \|u^j(t)\|_{\tilde s} &\geq \tilde{c_\kappa}\|u^j\|_s \geq \tilde{c_\kappa}\frac\rho2\geq \frac12 (\frac{\tilde{C_\kappa}C_sM}{\tilde{c_\kappa}} + 4) \geq 2 + \frac{\tilde{C_\kappa}C_s}{\tilde{c_\kappa}} \|u_0\|_{\tilde s} 
    \\
    &\geq 2 + \tilde{C_\kappa}C_s\|u_0\|_{ s} \geq 2 + \tilde{C_\kappa}\|u^j_0\|_{ s} \geq 2+ \|u^j_0\|_{\tilde s},
\end{align*}
thus $t\geq \tau_j^T$. This implies that $\tau_j^T \leq \eta^j$. Next, we state the following lemma.
\begin{lemma}[Abstract Cauchy Lemma]\label{lemma:cauchy}
For $T>0$ and $\tau_{j,k}^T$ defined in \eqref{taujk}, suppose that we have
  \begin{equation}\label{cauchy-condition-1}
      \lim_{k \rightarrow \infty} \sup_{j \geq k} \mathbb{E} \sup_{t \in\left[0, \tau_{j, k}^T\right]}\left\|u^j(t)-u^k(t)\right\|_{s}=0
  \end{equation}
and
\begin{equation}\label{cauchy-condition-2}
    \lim_{S \rightarrow 0} \sup_{j \geq J} \mathbb{P}\left[\sup_{t \in\left[0, \tau_j^T \wedge S\right]}\left\|u^j(t)\right\|_{\tilde s}>\|u_0^j\|_{\tilde s}+1\right]=0.
\end{equation}
Then there exists a stopping time $\tau$ with 
\begin{equation}\label{cauchy-result-1}
    \mathbb P(0<\tau\leq T)=1,
\end{equation}
and a predictable process $ u(\cdot \wedge \tau) \in C([0,T], \mathcal D_{s,\kappa})$ such that
\begin{equation}\label{cauchy-result-2}
    \sup_{t \in[0, \tau]}\left\|u^{j_l}-u\right\|_{s} \rightarrow 0 \quad \text { a.s. }
\end{equation}
for some subsequence $j_l\to \infty$. Moreover, 
\begin{equation}\label{cauchy-result-3}
    \sup\limits_{t\in[0,\tau]} \|u\|_{\tilde s}  \leq C_{\kappa,s}(1+ \|u_0\|_{\tilde s}) \quad \text { a.s.. }
\end{equation}
\end{lemma}

\begin{proof}
   We first construct the convergent subsequence by induction. Starting with $l=1$ and $j_1 = J$, suppose that $j_l$ has been identified, then thanks to \eqref{cauchy-condition-1} we know there exists $j_{l+1}$ such that
   \begin{align}\label{lemma:cauchy-proof-1}
       \sup\limits_{t\in[0,\tau_{j_l,j_{l+1}}^T]} \|u^{j_l}(t)-u^{j_{l+1}}(t)\|_{s} <  \frac{\tilde{c}_\kappa}{C C_s\tilde{C}_{\kappa}(M+4)}\kappa2^{-2l},
   \end{align}
   where the constant $C$ appears above is the Sobolev inequality constant such that 
   $\|f\|_{W^{2,\infty}}\leq C\|f\|_s,$ the constant $C_s$ comes from \eqref{eq:PnHs}, and the constants $\tilde{c}_\kappa$ and $\tilde{C}_\kappa$ are the ones appearing in \eqref{ctildekappa} such that 
   $\|u^{j_l}_0\|_{\tilde s} \leq \frac{\tilde{C}_\kappa C_s}{\tilde{c}_\kappa} \|u_0\|_{\tilde s}$.
   Recall that for any $u^{j_l}$ and $u^{j_{l+1}}$, they satisfy the local Rayleigh condition on $t\in[0,\tau_{j_l,j_{l+1}}^T]$: 
   \begin{align}\label{rayleigh-cauchypart}
        \kappa \leq \p_{zz} u^{j_l} \leq \frac1\kappa, \quad  \kappa \leq \p_{zz} u^{j_{l+1}} \leq \frac1\kappa.
   \end{align}
   By the Sobolev inequality and \eqref{lemma:cauchy-proof-1}, we have
   \begin{align}\label{lemma:cauchy-proof-2}
        \sup\limits_{t\in[0,\tau_{j_l,j_{l+1}}^T]} \|\p_{zz} u^{j_l}(t) - \p_{zz} u^{j_{l+1}}(t)\|_{L^\infty}  < \frac{\tilde{c}_\kappa}{ \tilde{C}_{\kappa}C_s(M+4)}\kappa2^{-2l}.
   \end{align}
   We define $\|u^{j_l}\|_{s,j_{l+1}}$ as:
   \begin{align}\label{norm-sj}
       \|u^{j_l}\|^2_{s,j_{l+1}} := \|u^{j_l}\|^2 + \sum\limits_{\substack{0\leq |\alpha| \leq s\\ D^\alpha\neq \p_x^s}} \|D^\alpha u^{j_l}\|^2 +\lnorm \frac{\p_x^s u^{j_l}}{\sqrt{\p_{zz} u^{j_{l+1}}}} \rnorm^2,
   \end{align}
   and $\|u^{j_{l+1}}\|_{s,j_{l}}$ is defined analogously. Thanks to \eqref{rayleigh-cauchypart}, we have the following equivalence between $\|\cdot\|_{s}$ and $\|\cdot\|_{s,j_{l+1}}$:
   \begin{align}\label{equivalence-sj}
       c_\kappa \|\cdot\|_{s}\leq \|\cdot\|_{s,j_{l+1}} \leq C_\kappa \|\cdot\|_{s}.
   \end{align}
   Next, we calculate that
   \begin{align*}
       &\lnorm \frac{\p_x^s u^{j_l}}{\sqrt{\p_{zz} u^{j_{l+1}}}} \rnorm^2 =  \int \frac{|\p_x^s u^{j_l}|^2 }{\p_{zz} u^{j_{l}} + (\p_{zz} u^{j_{l+1}} -\p_{zz} u^{j_{l}})} dxdz
       \\
       =& \int \frac{|\p_x^s u^{j_l}|^2 }{\p_{zz} u^{j_{l}}} \frac{\p_{zz} u^{j_{l}}}{(\p_{zz} u^{j_{l}}+\p_{zz} u^{j_{l+1}} -\p_{zz} u^{j_{l}})} dxdz = \int \frac{|\p_x^s u^{j_l}|^2 }{\p_{zz} u^{j_{l}}} \frac{1}{1+ \frac{\p_{zz} u^{j_{l+1}} -\p_{zz} u^{j_{l}}}{\p_{zz} u^{j_{l}}}} dxdz.
   \end{align*}
   For $t\in[0,\tau_{j_l,j_{l+1}}^T]$, using equations~ \eqref{rayleigh-cauchypart} and \eqref{lemma:cauchy-proof-2} gives
   \begin{align*}
       \frac{1}{1+ \frac{\p_{zz} u^{j_{l+1}} -\p_{zz} u^{j_{l}}}{\p_{zz} u^{j_{l}}}} \leq \frac{1}{1- \frac{\frac{\tilde{c}_\kappa \kappa}{ \tilde{C}_{\kappa}(M+4)}}{\kappa} 2^{-2l}} \leq \frac1{1-\frac{\tilde{c}_\kappa}{ \tilde{C}_{\kappa}(M+4)}2^{-2l}} \leq 1+\frac{\tilde{c}_\kappa}{ \tilde{C}_{\kappa}(M+4)}2^{-2l+1},
   \end{align*}
   and
   \begin{align*}
        \frac{1}{1+ \frac{\p_{zz} u^{j_{l+1}} -\p_{zz} u^{j_{l}}}{\p_{zz} u^{j_{l}}}} \geq \frac{1}{1+ \frac{\frac{\tilde{c}_\kappa \kappa}{ \tilde{C}_{\kappa}(M+4)}}{\kappa} 2^{-2l}}\geq \frac1{1+\frac{\tilde{c}_\kappa}{ \tilde{C}_{\kappa}(M+4)}2^{-2l}} \geq 1-\frac{\tilde{c}_\kappa}{ \tilde{C}_{\kappa}(M+4)} 2^{-2l},
   \end{align*}
   where we have used the inequalities $\frac1{1-x}\leq 1+2x$ when $0<x\leq \frac12$, and $\frac1{1-x}\geq 1+x$ for $0<x<1.$
   Therefore, for $t\in[0,\tau_{j_l,j_{l+1}}^T]$, one can use $(1+x)^{\frac12}<1+\frac12x$ to obtain
   \begin{align*}
      &(1-\frac{\tilde{c}_\kappa}{ \tilde{C}_{\kappa}C_s(M+4)}2^{-2l})\|u^{j_l}(t)\|_{\tilde s} \leq (1-\frac{\tilde{c}_\kappa}{ \tilde{C}_{\kappa}C_s(M+4)}2^{-2l})^{\frac12}\|u^{j_l}(t)\|_{\tilde s} \leq \|u^{j_l}(t)\|_{s,j_{l+1}}
      \\
      &\leq (1+ \frac{\tilde{c}_\kappa}{ \tilde{C}_{\kappa}C_s(M+4)}2^{-2l+1})^{\frac12}\|u^{j_l}(t)\|_{\tilde s} \leq  (1+ \frac{\tilde{c}_\kappa}{ \tilde{C}_{\kappa}C_s(M+4)}2^{-2l}) \|u^{j_l}(t)\|_{\tilde s}.
   \end{align*}
   Moreover, recall that on $t\in[0,\tau_{j_l,j_{l+1}}^T]$ one has $\|u^{j_l}(t)\|_{\tilde s}\leq 2+ \|u_0^{j_l}\|_{\tilde s} \leq 2+\frac{\tilde{C}_\kappa C_s}{ \tilde{c}_{\kappa}}\frac M2$, consequently,
   \begin{align}\label{lemma:cauchy-proof-3}
      \|u^{j_l}(t)\|_{\tilde s} - 2^{-2l-1}  \leq \|u^{j_l}(t)\|_{s,j_{l+1}} \leq   \|u^{j_l}(t)\|_{\tilde s} + 2^{-2l-1}, \quad \text{or } \Big|\|u^{j_l}(t)\|_{s,j_{l+1}}- \|u^{j_l}(t)\|_{\tilde s}\Big|\leq 2^{-2l-1}.
   \end{align}
   Such bounds hold for any $l\geq 1$ and any $t\in[0,\tau_{j_l,j_{l+1}}^T]$.
   
   Next, to find $\tau$, we define a new sequence of stopping time
   \begin{align}\label{tilde_tau}
       \tilde{\tau}_l := \inf \left\lbrace t\geq 0: \|u^{j_l}(t)\|_{\tilde s} \geq 1 + 2^{-l+1} + \|u_0^{j_l}\|_{\tilde s} \right\rbrace \wedge \inf \left\lbrace t \geq 0: \, \|\partial_{zz}  u^{j_l}(t) - \partial_{zz}  u^{j_l}_0 \|_{L^\infty} \geq \frac\kappa4 \right\rbrace \wedge T,
   \end{align}
   and let
   \begin{align*}
          \Omega_N := \bigcap\limits_{n=N}^\infty \left\lbrace \sup\limits_{t\in[0, \tilde{\tau}_{n}\wedge \tilde{\tau}_{n+1}]} \|u^{j_n} - u^{j_{n+1}}\|_{s,j_{n+1}} < 2^{-(n+2)}   \right\rbrace.
   \end{align*}
   Then it is clear that $\tilde{\tau}_l \leq \tau_{j_l}^T$ and $\tilde{\tau}_n \wedge  \tilde{\tau}_{n+1} \leq \tau_{j_n,j_{n+1}}^T$ where $\tau_{j_l}^T$ and $\tau_{j_n,j_{n+1}}^T$ are defined in \eqref{stopping-time-taujT} and \eqref{taujk}. By Markov inequality and thanks to \eqref{equivalence-sj} we get
   \begin{align*}
       &\mathbb P\left(  \sup\limits_{t\in[0, \tilde{\tau}_{l}\wedge \tilde{\tau}_{l+1}]} \|u^{j_l}(t)-u^{j_{l+1}}(t)\|_{s,j_{l+1}} \geq 2^{-(l+2)}  \right) \leq 2^{l+2} \mathbb E \left(  \sup\limits_{t\in[0, \tilde{\tau}_{l}\wedge \tilde{\tau}_{l+1}]} \|u^{j_l}(t)-u^{j_{l+1}}(t)\|_{s,j_{l+1}}   \right)
       \\
       \leq &C_{\kappa}2^{l+2}\mathbb E \left(  \sup\limits_{t\in[0, \tau_{j_l,j_{l+1}}^T]} \|u^{j_l}(t)-u^{j_{l+1}}(t)\|_{s}   \right)\leq C_{\kappa}\frac{1}{M+4}2^{-(l-3)},
   \end{align*}
 and thus
   \[
   \sum_{l=1}^\infty \mathbb P\left(  \sup\limits_{t\in[0, \tilde{\tau}_{l}\wedge \tilde{\tau}_{l+1}]} \|u^{j_l}(t)-u^{j_{l+1}}(t)\|_{s,j_{l+1}} \geq 2^{-(l+2)}  \right) <\infty.
   \]
   By the Borel-Cantelli lemma, we infer that
   \begin{align*}
      \mathbb P  \left( \bigcap_{N=1}^\infty \bigcup_{n=N}^\infty \left\lbrace \sup\limits_{t\in[0, \tilde{\tau}_{n}\wedge \tilde{\tau}_{n+1}]} \|u^{j_n} - u^{j_{n+1}}\|_{s,j_{n+1}} \geq 2^{-(n+2)} \right\rbrace  \right)=0,
   \end{align*} 
   and therefore $\tilde\Omega = \cup_{N=1}^\infty \Omega_N \subseteq\Omega$ is a set of full measure. 
   
   In order to show that $\{\tilde{\tau}_l\}$ converges, we will establish
   \begin{align}\label{proof-cauchy-decrease}
     \tilde{\tau}_{l+1}(\omega)\leq \tilde{\tau}_l(\omega) \quad \text{for every } l\geq N, \, \omega\in\Omega_N.
   \end{align}
   Given $N$ and $l\geq N$ we consider the set $\lbrace \tilde{\tau}_{l+1} > \tilde{\tau}_{l}\rbrace\cap \Omega_N$. By continuity in time and the definition of $\tilde{\tau}_l$, we infer that
   \[
   \sup\limits_{t\in[0, \tilde{\tau}_{l}]} \|u^{j_l}\|_{\tilde s} = 1+2^{-l+1} +\|u^{j_l}_0\|_{\tilde s}.
   \]
   On $\Omega_N$, note that $\|u^{j_{l+1}}\|_{s,j_{l+1}} = \|u^{j_{l+1}}\|_{\tilde s}$, by the triangle inequality one has
   \[
    \sup\limits_{t\in[0, \tilde{\tau}_{l}\wedge \tilde{\tau}_{l+1}]} \|u^{j_l}\|_{s,j_{l+1}} -  \sup\limits_{t\in[0, \tilde{\tau}_{l}\wedge \tilde{\tau}_{l+1}]} \|u^{j_{l+1}}\|_{\tilde s} < 2^{-(l+2)},
   \]
   and
   \[
     \|u^{j_{l+1}}_0\|_{\tilde s}- \|u^{j_l}_0\|_{s,j_{l+1}}  < 2^{-(l+2)}.
   \]
   By the observations above, thanks to \eqref{lemma:cauchy-proof-3}, one has
   \begin{align}\label{lemma:cauchy-proof-4}
       \sup\limits_{t\in[0, \tilde{\tau}_{l}\wedge \tilde{\tau}_{l+1}]} \|u^{j_{l+1}}\|_{\tilde s} > & \sup\limits_{t\in[0, \tilde{\tau}_{l}\wedge \tilde{\tau}_{l+1}]} \|u^{j_l}\|_{s,j_{l+1}} - 2^{-(l+2)} \geq  \sup\limits_{t\in[0, \tilde{\tau}_{l}\wedge \tilde{\tau}_{l+1}]}  \|u^{j_l}\|_{\tilde s} - 2^{-2l-1}  - 2^{-(l+2)}
       \\
       = &\sup\limits_{t\in[0, \tilde{\tau}_{l}]}\|u^{j_l}(t)\|_{\tilde s} - 2^{-2l-1} - 2^{-(l+2)}
      = 1+2^{-l+1} +\|u^{j_l}_0\|_{\tilde s}- 2^{-2l-1} - 2^{-(l+2)} 
      \\
      \geq & 1+2^{-l+1} + \|u^{j_l}_0\|_{s,j_{l+1}} - 2\cdot 2^{-2l-1} - 2^{-(l+2)}
      \\
      >& 1+2^{-l+1} +  \|u^{j_{l+1}}_0\|_{\tilde s} - 2^{-2l} - 2\cdot 2^{-(l+2)}
      \\
      =& 1  + 2^{-(l+1)+1} + 2^{-l-1} - 2^{-2l} + \|u^{j_{l+1}}_0\|_{\tilde s} \geq 1  + 2^{-(l+1)+1} + \|u^{j_{l+1}}_0\|_{\tilde s}.
   \end{align}
   on $\lbrace \tilde{\tau}_{l+1} > \tilde{\tau}_{l}\rbrace\cap \Omega_N$. 
   On the other hand, on $\lbrace \tilde{\tau}_{l+1} > \tilde{\tau}_{l}\rbrace\cap \Omega_N$ we have
   \begin{align}\label{lemma:cauchy-proof-5}
       \sup\limits_{t\in[0, \tilde{\tau}_{l}\wedge \tilde{\tau}_{l+1}]} \|u^{j_{l+1}}\|_{\tilde s} \leq \sup\limits_{t\in[0,  \tilde{\tau}_{l+1}]} \|u^{j_{l+1}}\|_{\tilde s} \leq 1 + 2^{-(l+1)+1}  + \|u^{j_{l+1}}_0\|_{\tilde s}.
   \end{align}
   Therefore \eqref{lemma:cauchy-proof-4} and \eqref{lemma:cauchy-proof-5} yield that $\lbrace \tilde{\tau}_{l+1} > \tilde{\tau}_{l}\rbrace\cap \Omega_N$ is a empty set.  By the discussion above we know \eqref{proof-cauchy-decrease} is true, thus we can define $\tau(\omega)= \lim_{l\to \infty} \tilde{\tau}_l(\omega)$ for $\omega\in\Omega_N$. As $\tilde\Omega = \cup_{N=1}^\infty \Omega_N$ is a set of full measure, then for almost every $\omega \in \Omega$ there exists some $N$ such that $\omega\in\Omega_N$, and therefore we can define
   \begin{align*}
      \tau= \lim_{l\to \infty} \tilde{\tau}_l \quad a.s..
   \end{align*}
   
   Next, we show that $\tau>0$ a.s.. For each $\varepsilon>0$ with $T>\varepsilon>0$, we have
   \begin{align*}
       \lbrace \tilde{\tau}_l < \varepsilon \rbrace\subseteq \left\lbrace \sup\limits_{t\in[0, \tilde{\tau}_{l}\wedge \varepsilon]} \|u^{j_{l}}\|_{\tilde s} = \|u^{j_{l}}_0\|_{\tilde s} +1 + 2^{-l+1}\right\rbrace 
       \subseteq \left\lbrace \sup\limits_{t\in[0, \tilde{\tau}_{l}\wedge \varepsilon]} \|u^{j_{l}}\|_{\tilde s} > \|u^{j_{l}}_0\|_{\tilde s} +1 \right\rbrace. 
   \end{align*}
   Since
   \[
   \lbrace \tau < \varepsilon \rbrace = \bigcup_{l=1}^\infty \bigcap_{k=l}^\infty \lbrace \tilde\tau_k < \varepsilon\rbrace,
   \]
   then as $\tilde{\tau}_l \leq \tau_{j_l}^T$, we have
   \begin{align*}
       &\mathbb P \left( \tau < \varepsilon\right) = \mathbb P \left( \bigcup_{l=1}^\infty \bigcap_{k=l}^\infty \lbrace \tilde\tau_k < \varepsilon\rbrace\right) \leq \limsup\limits_{l\to\infty} \mathbb P \left( \tilde\tau_l < \varepsilon\right)
       \\
       \leq &\sup_{l\geq 1} \mathbb P \left(\left\lbrace \sup\limits_{t\in[0, \tilde{\tau}_{l}\wedge \varepsilon]} \|u^{j_{l}}\|_{\tilde s} > \|u^{j_{l}}_0\|_{\tilde s} +1 \right\rbrace\right) \leq \sup_{l\geq 1} \mathbb P \left(\left\lbrace \sup\limits_{t\in[0, \tau_{j_l}^T\wedge \varepsilon]} \|u^{j_{l}}\|_{\tilde s} > \|u^{j_{l}}_0\|_{\tilde s} +1 \right\rbrace\right).
   \end{align*}
   Then by condition \eqref{cauchy-condition-2}, we know 
   \begin{align*}
       \mathbb P (\tau=0) = \mathbb P (\cap_{\varepsilon>0} \lbrace \tau<\varepsilon\rbrace) = \lim_{\varepsilon\to 0} \mathbb P(\tau<\varepsilon)=0.
   \end{align*}
   By construction $\tau\leq T$, thus \eqref{cauchy-result-1} follows.
   
   Next, we establish \eqref{cauchy-result-2}. Recall that for each $\omega\in \tilde\Omega$, there exists $N=N(\omega)$ such that $\omega\in\Omega_N$ and $\tau(\omega)\leq \tilde\tau_{l+1}(\omega)\leq \tilde\tau_{l}(\omega)$ provided that $l\geq N(\omega).$ Therefore,
   \begin{align*}
       \sup\limits_{t\in[0, \tau(\omega)]} \|u^{j_l} - u^{j_{l+1}}\|_{s} \leq C_{\kappa}\sup\limits_{t\in[0, \tau(\omega)]} \|u^{j_l} - u^{j_{l+1}}\|_{s,j_{l+1}} \leq C_{\kappa}\sup\limits_{t\in[0, \tilde{\tau}_{l}\wedge\tilde{\tau}_{l+1} ]} \|u^{j_l} - u^{j_{l+1}}\|_{s,j_{l+1}} < C_{\kappa} 2^{-(l+2)}.
   \end{align*}
   This implies that $\{u^{j_l}(\cdot\wedge\tau)\}$ is a Cauchy sequence in $C([0,T],\mathcal D_s)$, thus there exists a predictable process $u(\cdot\wedge\tau)\in C([0,T],\mathcal D_s)$ such that \eqref{cauchy-result-2} holds.
   From \eqref{cauchy-result-2} one can infer that
   \[
   \sup_{t\in[0,\tau]} \|\p_{zz} u^{j_l} - \p_{zz} u\|_{L^\infty} \to 0 \quad a.s..
   \]
   As \eqref{rayleigh-cauchypart} holds for any $l$, this implies that for $t\in[0,\tau]$ we also have \eqref{rayleigh-cauchypart} holds for $u$. Thus $u(\cdot\wedge\tau)\in C([0,T],\mathcal D_{s,\kappa})$.
 
 Finally, to establish \eqref{cauchy-result-3}, for any $\omega\in\Omega_N$ and $l\geq N$, since $\tau(\omega)\leq \tilde\tau_{l}(\omega)$ we have
   \begin{align*}
       & \sup\limits_{t\in[0,\tau(\omega)]} \|u^{j_l}\|_{s}\leq C_{\kappa} \sup\limits_{t\in[0,\tau(\omega)]} \|u^{j_l}\|_{\tilde s}\leq C_{\kappa}(1+ 2^{-l+1} + \|u^{j_{l}}_0\|_{\tilde s})\leq C_{\kappa,s}(1+ \|u_0\|_{\tilde s}).
   \end{align*}
   Notice that the above bound is uniform in $\omega\in \tilde\Omega$. Now using \eqref{cauchy-result-2}, we obtain 
   \begin{align*}
       \sup\limits_{t\in[0,\tau]} \|u\|_{\tilde s} \leq C_{\kappa}\sup\limits_{t\in[0,\tau]} \|u\|_{s} \leq C_{\kappa,s}(1+ \|u_0\|_{\tilde s}) \quad a.s..
   \end{align*}
   This finishes the proof.
\end{proof}

\subsubsection{Verification of condition~\eqref{cauchy-condition-1}}
\quad

We first establish \eqref{cauchy-condition-1}. As $(u^j,\eta^j)$ and $(u^k,\eta^k)$ are pathwise solutions to the original system \eqref{PE-inviscid-system}, from \eqref{PE-inviscid-1} we have
\[
 du^j + (u^j \p_x u^j + w^j \p_z u^j_z + \p_x p^j) dt = \sigma(u^j) dW, 
\]
\[
 du^k + (u^k \p_x u^k + w^k \p_z u^k_z + \p_x p^k) dt = \sigma(u^k) dW.
\]
Denote by $U=u^k-u^j$, $W=w^k-w^j$ and $P=p^k-p^j$ with $U_0 = u_0^k - u_0^j$, then
\begin{equation}\label{eqn:U}
    \begin{split}
        dU + (U \p_x u^k + u^j \p_x U + W \p_z u^k + w^j \p_z U + \p_x P)dt = (\sigma(u^k)-\sigma(u^j)) dW.
    \end{split}
\end{equation}
Denote by $v^j=\p_z u^j$, $v^k=\p_z u^k$, and $V=\p_z U = v^k-v^j$, one has
\begin{equation}\label{eqn:V}
    \begin{split}
        dV + (U \p_x v^k + u^j \p_x V + W \p_z v^k + w^j \p_z V)dt = \p_z(\sigma(u^k)-\sigma(u^j)) dW.
    \end{split}
\end{equation}
Notice that here we are not able to estimate $\|U\|_{\tilde s}$ as $U$ may not satisfy the local Rayleigh condition. Similar to \eqref{norm-sj} we define $\|\cdot\|_{s,j}$ by 
\begin{align}
    \|U\|_{s,j}^2 := \|U\|^2+\sum\limits_{\substack{|\alpha|\leq s,\\ D^\alpha \neq \p_x^{s}}} \|D^\alpha V\|^2 + \lnorm \frac{\p_x^{s} V}{\sqrt{\p_z v^j}} \rnorm^2.
\end{align}
Recall that $\tau^T_{j,k}\leq \eta^j\wedge\eta^k$, for each $k,j\geq J$, so for  $t\in[0,\tau^T_{j,k}]$ we have $\|u^j(t)\|_{s}<\frac\rho2$ and $\|u^k(t)\|_{s}<\frac\rho2$. 
 Thanks to \eqref{uj-regularity}, $\|U\|_{s,j}\sim_{\kappa} \|U\|_s$ is equivalent up to some constant depending on $\kappa$ but independent of $j$. The independence of $j$ in the equivalence is crucial in the analysis below.

\noindent\underline{{\bf Estimates of $\|U\|$ and $\|D^\alpha V\|$ with $D^\alpha \neq \p_x^s$:}} By It\^o's formula, we have
\begin{equation}\label{I123}
    \begin{split}
        &d(\|U\|^2 + \sum\limits_{|\alpha|\leq s, D^\alpha\neq \p_x^s} \|D^\alpha V\|^2) 
        \\
        =& -2\Big(\left\langle U \p_x u^k + u^j \p_x U + W \p_z u^k + w^j \p_z U + \p_x P, U  \right\rangle 
        \\
        &+ \sum\limits_{|\alpha|\leq s, D^\alpha\neq \p_x^s} \left\langle D^\alpha (U \p_x v^k + u^j \p_x V + W \p_z v^k + w^j \p_z V) , D^\alpha V  \right\rangle \Big) dt
        \\
        &+ \Big(\|\sigma(u^k) -\sigma(u^j)\|_{L^2(\mathscr U, L^2)}^2 + \sum\limits_{|\alpha|\leq s, D^\alpha\neq \p_x^s} \|D^\alpha\p_z(\sigma(u^k) -\sigma(u^j))\|_{L^2(\mathscr U, L^2)}^2\Big) dt
        \\
        &+ 2\Big(\left\langle  \sigma(u^k) -\sigma(u^j), U  \right\rangle + \sum\limits_{|\alpha|\leq s, D^\alpha\neq \p_x^s} \left\langle D^\alpha\p_z (\sigma(u^k) -\sigma(u^j)), D^\alpha V  \right\rangle \Big) dW 
        \\
        =: & I_1 dt + I_2 dt + I_3 dW.
    \end{split}
\end{equation}
The estimates can be performed similarly as in previous sections. Notice that we are doing estimates for $t\in[0,\tau_{j,k}^T]$.
For the nonlinear terms in $I_1$, one can get
\begin{align*}
    |I_1| \leq C_{\rho,\kappa} \|U\|_{s,j}^2 + C_{\rho,\kappa}\|u^k\|_{s+1}^2\|U\|_{s-1}^2,
\end{align*}
where the second term on the right-hand side is troublesome and will be treated at the end. For terms in $I_2$, using Assumption~\eqref{noise-ine}, we have
\begin{align*}
    |I_2| \leq C_{\kappa}\|U\|_{s,j}^2.
\end{align*}
The $I_3dW$ term is treated by the BDG inequality and using \eqref{noise-ine},
\begin{align*}
    \mathbb E\left( \sup\limits_{t\in [0,\tau_{j,k}^T]} \left|\int_0^t I_3 dW \right| \right) \leq \frac14 \mathbb E\sup\limits_{t\in [0,\tau_{j,k}^T]} \|U\|_{s,j}^2 + C_{\kappa}\mathbb E\int_0^{\tau_{j,k}^T} \|U\|_{s,j}^2 dt.
\end{align*}

\noindent\underline{{\bf Estimate of $\lnorm\frac{\p_x^{s} V}{\sqrt{\p_z v^j}}\rnorm$:}} We first compute
\begin{equation*}
\begin{split}
    d \partial_x^{s} V &= -\Big[\p_x^{s}(U \p_x v^k + u^j \p_x V + W \p_z v^k + w^j \p_z V   )\Big] dt +\Big[\p_x^{s}\p_z(\sigma (u^k) -\sigma (u^j))\Big] dW
    \\
    &= -\Big[\p_x^{s}(U \p_x v^k + u^j \p_x V + W \p_z v^j + w^k \p_z V   )\Big] dt +\Big[\p_x^{s}\p_z(\sigma (u^k) -\sigma (u^j))\Big] dW
    =:A_1 dt + A_2dW,
\end{split}
\end{equation*}
and
\begin{equation*}
\begin{split}
      &d \p_z v^j = -(\p_z u^j \p_x v^j + u^j \p_{xz} v^j + \p_z w^j \p_z v^j + w^j\p_{zz} v^j)dt +\partial_{zz} \sigma (u^j) dW=: B_1 dt + B_2 dW.
\end{split}
\end{equation*}
By It\^o's formula, we have
\begin{equation}\label{I456}
    \begin{split}
        d\lnorm\frac{\p_x^{s} V}{\sqrt{\p_z v^j}}\rnorm^2 = &\left(\left\langle 2A_1, \frac{\p_x^{s} V}{\p_z v^j} \right\rangle - \left\langle B_1,\frac{|\p_x^{s} V|^2}{|\p_z v^j|^2} \right\rangle\right)dt
        \\
        &+\frac12\left(\left\langle 2A_2^2,\frac1{\p_z v^j} \right\rangle - \left\langle 4\frac{\p_x^{s} V}{|\p_z v^j|^2} ,A_2 B_2 \right\rangle +  \left\langle 2\frac{|\p_x^{s} V|^2}{|\p_z v^j|^3} ,B_2^2 \right\rangle \right) dt
        \\
        &+ \left(\left\langle 2\frac{\p_x^{s} V}{\p_z v^j}, A_2 \right\rangle - \left\langle \frac{|\p_x^{s} V|^2}{|\p_z v^j|^2}, B_2 \right\rangle  \right)dW =:I_4dt + I_5 dt + I_6 dW.
    \end{split}
\end{equation}
For terms in $I_4$, a direct estimate will give a bound including $\|u^j\|_{s+1}^2\|U\|_{s-1}^2$, for which we are not able to control at the end.
In order to get rid of $\|u^j\|_{s+1}^2\|U\|_{s-1}^2$, we shall discuss the highest order terms in $I_4$ as they are crucial in getting the correct bound. First, by integration by parts and thanks to the property of stopping time $\tau_{j,k}^T$, for $t\in[0,\tau_{j,k}^T]$ one has
\[
\left|\langle \p_x^s(U\p_x v^k + u^j \p_x V + w^k \p_z V), \frac{\p_x^{s} V}{\p_z v^j} \rangle \right|\leq C_{\rho,\kappa} \|U\|_{s,j}^2 + C_{\rho,\kappa}\|u^k\|_{s+1}^2 \|U\|_{s-1}^2.
\]
Next, the worst term vanishes due to
\begin{align*}
    \int \p_x^s W \p_z v^{j}  \frac{\p_x^{s} V}{\p_z v^j} dxdz = \int \p_x^s W \p_x^{s} V dxdz = 0 \text{ by integration by parts.} 
\end{align*}
On the other hand, the term involving the highest order on $v^j$ can be handled as:
\begin{align*}
    &\left|\int  W \p_x^s \p_z v^{j} \frac{\p_x^{s} V}{\p_z v^j} dxdz\right| = \left|\int  W \p_x^s \p_z (v^k-V) \frac{\p_x^{s} V}{\p_z v^j} dxdz\right|
    \\
    \leq &\left|\int  W \p_x^s \p_z v^k \frac{\p_x^{s} V}{\p_z v^j} dxdz\right| + \left|\int  W \p_x^s \p_z V \frac{\p_x^{s} V}{\p_z v^j} dxdz\right|
    \\
    \leq &  C_{\rho,\kappa} \|U\|_{s,j}^2 + C_{\rho,\kappa}\|u^k\|_{s+1}^2 \|U\|_{s-1}^2,
\end{align*}
where for the second term we have used integration by parts. The lower-order terms can be handled readily. By doing these we successfully get rid of $\|u^j\|_{s+1}^2\|U\|_{s-1}^2$, and deduce
\begin{align*}
    |I_4|\leq C_{\rho,\kappa} \|U\|_{s,j}^2 + C_{\rho,\kappa}\|u^k\|_{s+1}^2 \|U\|_{s-1}^2.
\end{align*}
Again the second term on the right-hand side is troublesome and will be treated at the end. For $I_5$, by Assumption~\eqref{noise-ine}, one has
\begin{align*}
    |I_5|\leq C_{\kappa}\|U\|_{s,j}^2.
\end{align*}
Then the BDG inequality and \eqref{noise-ine} give that
\begin{align*}
    \mathbb E\Big( \sup\limits_{t\in [0,\tau_{j,k}^T]} \left|\int_0^t I_6 dW \right| \Big) \leq \frac14 \mathbb E\sup\limits_{t\in [0,\tau_{j,k}^T]} \|U\|_{s,j}^2 + C_{\rho,\kappa}\mathbb E\int_0^{\tau_{j,k}^T} \|U\|_{s,j}^2 dt.
\end{align*}

\noindent\underline{{\bf Combining the estimates:}} The estimates of $I_1$ to $I_6$ together brings
\begin{align}\label{step5-combine}
    \mathbb E\left( \sup\limits_{t\in [0,\tau_{j,k}^T]} \|U\|_{s,j}^2\right) \leq 2\mathbb E \|U_0\|_{s,k}^2 +& C_{\rho,\kappa}\mathbb E\int_0^{\tau_{j,k}^T} \|U\|_{s,j}^2 dt \nonumber\\
    &
    + C_{\rho,\kappa}\mathbb E\int_0^{\tau_{j,k}^T}  \|u^k\|_{s+1}^2\|U\|_{s-1}^2 dt.
\end{align}
By the Gr\"onwall lemma,
\begin{equation}\label{trouble-term-1}
    \begin{split}
       \mathbb E\left( \sup\limits_{t\in [0,\tau_{j,k}^T]}\|u^k-u^j\|_{s}^2\right) &\leq C_{\kappa}\mathbb E\left( \sup\limits_{t\in [0,\tau_{j,k}^T]}\|u^k-u^j\|_{s,j}^2\right)
    \\&= C_{\kappa}\mathbb E\left( \sup\limits_{t\in [0,\tau_{j,k}^T]} \|U\|_{s,j}^2\right)
    \\
    &\leq C_{\rho,\kappa,T} \mathbb E\left( \|U_0\|_{s,j}^2 + \sup\limits_{t\in [0,\tau_{j,k}^T]} \|u^k\|_{s+1}^2\|U\|_{s-1}^2 \right)
    \\
    &\leq C_{\rho,\kappa,T} \mathbb E\left(\|u^k_0-u^j_0\|_{s}^2 + \sup\limits_{t\in [0,\tau_{j,k}^T]} \|u^k\|_{s+1}^2 \|U\|_{s-1}^2 \right). 
    \end{split}
\end{equation}
Thanks to \eqref{eq:PnconvHs} and since $u_0 \in \mathcal D_s$, we have $\sup_{j\geq k} \|u^k_0-u^j_0\|_{s}^2 \to 0$ as $k\to \infty$. Then by the dominated convergence theorem, we have 
\[
 \sup_{j\geq k} \mathbb E \|u^k_0-u^j_0\|_{s}^2 \leq \mathbb E  \sup_{j\geq k}  \|u^k_0-u^j_0\|_{s}^2 \to 0 \quad \text{as } k\to \infty.
\]
Therefore, \eqref{cauchy-condition-1} is proved once we can establish 
\begin{equation}\label{cauchy-condition-1-additional}
    \lim\limits_{k\to \infty} \sup_{j\geq k} \mathbb E \sup\limits_{t\in [0,\tau_{j,k}^T]} \|u^k\|_{s+1}^2\|U\|_{s-1}^2=0.
\end{equation}

We now prove \eqref{cauchy-condition-1-additional}, where we first compute $d(\|u^k\|_{\widetilde{s+1}}^2 \|U\|_{s-1,j}^2)$. Using \eqref{I123} and \eqref{I456} gives
\begin{equation}\label{I16}
    d\|U\|_{s-1,j}^2 = (I'_1+I'_2+I'_4+I'_5) dt + (I'_3+I'_6)dW,
\end{equation}
where $I_i'$ is obtained by replacing the index $s$ with $s-1$ in $I_i$ appearing in \eqref{I123} and \eqref{I456}.
To find $d\|u^k\|_{\widetilde{s+1}}^2$, similar as \eqref{I123} one can get
\begin{equation}
\begin{split}
    &d(\|u^k\|^2 + \sum\limits_{|\alpha|\leq s+1, D^\alpha\neq \p_x^{s+1}} \|D^\alpha v^k\|^2) 
    \\
    =& -2\Big(\langle u^k\p_x u^k + w^k \p_z u^k + \p_x p^k, u^k\rangle + \sum\limits_{|\alpha|\leq s+1, D^\alpha\neq \p_x^{s+1}} \langle D^\alpha(u^k\p_x v^k + w^k \p_z v^k), D^\alpha v^k \rangle \Big)dt
    \\
    &+ \Big(\|\sigma(u^k)\|_{L^2(\mathscr U, L^2)}^2 + \sum\limits_{|\alpha|\leq s+1, D^\alpha\neq \p_x^{s+1}} \|D^\alpha \p_z \sigma(u^k)\|_{L^2(\mathscr U, L^2)}^2\Big)dt
    \\
    &+ 2\Big( \langle \sigma(u^k), u^k \rangle + \sum\limits_{|\alpha|\leq s+1, D^\alpha\neq \p_x^{s+1}} \langle D^\alpha\p_z\sigma(u^k), D^\alpha v^k \rangle \Big)dW =: J_1dt +J_2 dt + J_3 dW.
\end{split}
\end{equation}
Note that
\begin{equation*}
    d \p_x^{s+1} v^k = -\left[\p_x^{s+1}(u^k \p_x v^k + w^k \p_z v^k) \right]dt + \left[\p_x^{s+1}\p_z\sigma(u^k) \right]dW =: A_1dt +A_2 dW.
\end{equation*}
\begin{equation*}
\begin{split}
      &d \p_z v^k = -(\p_z u^k \p_x v^k + u^k \p_{xz} v^k + \p_z w^k \p_z v^k + w^k\p_{zz} v^k)dt +\partial_{zz} \sigma (u^k) dW=: B_1 dt + B_2 dW.
\end{split}
\end{equation*}
By It\^o's formula, similar to \eqref{I456}, we get
\begin{equation}
    \begin{split}
          d\lnorm\frac{\p_x^{s+1} v^k}{\sqrt{\p_z v^k}}\rnorm^2 = &\left(\left\langle 2A_1, \frac{\p_x^{s+1} v^k}{\p_z v^k} \right\rangle - \left\langle B_1,\frac{|\p_x^{s+1} v^k|^2}{|\p_z v^k|^2} \right\rangle\right)dt
        \\
        &+\frac12\left(\left\langle 2A_2^2,\frac1{\p_z v^k} \right\rangle - \left\langle 4\frac{\p_x^{s+1} v^k}{|\p_z v^k|^2} ,A_2 B_2 \right\rangle +  \left\langle 2\frac{|\p_x^{s+1} v^k|^2}{|\p_z v^k|^3} ,B_2^2 \right\rangle \right) dt
        \\
        &+ \left(\left\langle 2\frac{\p_x^{s+1} v^k}{\p_z v^k}, A_2 \right\rangle - \left\langle \frac{|\p_x^{s+1} v^k|^2}{|\p_z v^k|^2}, B_2 \right\rangle  \right)dW =:J_4dt + J_5 dt + J_6 dW.
    \end{split}
\end{equation}
Therefore,  
\begin{equation}\label{j16}
    d\|u^k\|_{\widetilde{s+1}}^2 = (J_1+J_2+J_4+J_5)dt + (J_3+J_6) dW.
\end{equation}
By an application of the It\^o product rule, we obtain
\begin{equation}\label{ito-product}
    \begin{split}
        &d(\|u^k\|_{\widetilde{s+1}}^2 \|U\|_{s-1,j}^2) 
        \\
        =&\|u^k\|_{\widetilde{s+1}}^2 d\|U\|_{s-1,j}^2 + \|U\|_{s-1,j}^2 d\|u^k\|_{\widetilde{s+1}}^2 + d\|u^k\|_{\widetilde{s+1}}^2 d\|U\|_{s-1,j}^2
        \\
        =& \Big( \|u^k\|_{\widetilde{s+1}}^2(I'_1+I'_2+I'_4+I'_5) +  \|U\|_{s-1,j}^2(J_1+J_2+J_4+J_5) + (I'_3+I'_6)(J_3+J_6)\Big) dt
        \\
        &+ \Big( \|u^k\|_{\widetilde{s+1}}^2 (I'_3+I'_6) + \|U\|_{s-1,j}^2 (J_3+J_6) \Big)dW,
    \end{split}
\end{equation}
where $(I'_3+I'_6)(J_3+J_6)$ is the term arising from $(I'_3+I'_6)dW (J_3+J_6) dW $. The estimates for \eqref{ito-product} follow similarly as in previous steps and in previous sections, thus we mainly list out the final results.

Similar to the estimates for $I_i$, we obtain
\[
|I'_1+I'_2+I'_4+I'_5|\leq C_{\rho,\kappa} \|U\|_{s-1,j}^2 +  C_{\rho,\kappa}\|u^k\|_{s}^2 \|U\|_{s-2}^2 \leq C_{\rho,\kappa} \|U\|_{s-1,j}^2.
\]
Therefore,
\[
\|u^k\|_{\widetilde{s+1}}^2(I'_1+I'_2+I'_4+I'_5) \leq C_{\rho,\kappa} \|U\|_{s-1,j}^2 \|u^k\|_{\widetilde{s+1}}^2.
\]
Following a similar derivation as in Proposition \ref{proposition:estimate-inviscid} gives
\[
\|U\|_{s-1,j}^2(J_1+J_2+J_4+J_5) \leq C_{\rho,\kappa} \|U\|_{s-1,j}^2 \|u^k\|_{\widetilde{s+1}}^2.
\]
Thanks to Assumption~\eqref{noise-ine}, for the last drift term we have
\[
(I'_3+I'_6)(J_3+J_6) \leq C_{\rho,\kappa} \|U\|_{s-1,j}^2 (1+\|u^k\|_{\widetilde{s+1}}^2).
\]
Next, for the stochastic terms, the BDG inequality implies
\begin{align*}
   &\mathbb E \left( \sup\limits_{t\in[0,\tau_{j,k}^T]} \left|\int_0^t \Big(\|u^k\|_{\widetilde{s+1}}^2 (I'_3+I'_6) + \|U\|_{s-1,j}^2 (J_3+J_6)\Big)dW \right| \right) 
   \\
   \leq &\frac12 \mathbb E \sup\limits_{t\in[0,\tau_{j,k}^T]} (\|u^k\|_{\widetilde{s+1}}^2 \|U\|_{s-1,j}^2) + C_{\rho,\kappa}\mathbb E \int_0^{\tau_{j,k}^T}  \|U\|_{s-1,j}^2 (1+\|u^k\|_{\widetilde{s+1}}^2) dt.
\end{align*}
Summarizing the above estimates brings
\begin{align*}
     \mathbb E\left( \sup\limits_{t\in [0,\tau_{j,k}^T]} \|U\|_{s-1,j}^2\|u^k\|_{\widetilde{s+1}}^2\right) \leq &2\mathbb E \left(\|U_0\|_{s-1,j}^2\|u^k_0\|_{\widetilde{s+1}}^2\right)
     \\
     &+ C_{\rho,\kappa}\mathbb E\int_0^{\tau_{j,k}^T} \left(\|U\|_{s-1,j}^2 \|u^k_0\|_{\widetilde{s+1}}^2 + \|U\|_{s-1,j}^2\right)  dt.
\end{align*}
Then the Gr\"onwall inequality yields 
\begin{align*}
    \mathbb E\left( \sup\limits_{t\in [0,\tau_{j,k}^T]} \|U\|_{s-1,j}^2\|u^k\|_{\widetilde{s+1}}^2\right) \leq &C_{\rho,\kappa,T} \mathbb E \left(\|u^k_0-u^j_0\|_{s-1,j}^2\|u^k_0\|_{\widetilde{s+1}}^2\right)
    \\
    &+ C_{\rho,\kappa,T} \mathbb E\left(\sup\limits_{t\in [0,\tau_{j,k}^T]}\|u^k(t)-u^j(t)\|_{s-1,j}^2\right) .
\end{align*}
Regarding the first term in the above inequality, thanks to the equivalence of norms  \eqref{ctildekappa} and \eqref{equivalence-sj}, using \eqref{eq:PnHs1} and \eqref{eq:PnconvHsminus1} brings,
\begin{align*}
 &\sup_{j\geq k}\|u^k_0-u^j_0\|_{s-1,j}^2\|u^k_0\|_{\widetilde{s+1}}^2 \leq C_{\kappa} \sup_{j\geq k} \|u^k_0-u^j_0\|_{s-1}^2\|u^k_0\|_{s+1}^2
 \\
  \leq &C_{\kappa,s} \sup_{j\geq k}\|u^k_0-u^j_0\|_{s-1}^2 k^2 \|u_0\|_{s}^2 
 \\
 \leq &  C_{\kappa,s} \|u_0\|_{s}^2  \left(k^2\|u_0 - u^k_0\|_{s-1}^2 + \sup_{j\geq k} j^2\|u_0 - u^j_0\|_{s-1}^2 \right) \to 0 \quad \text{as } k\to \infty.
\end{align*}
By the dominant convergence theorem, 
\begin{align}\label{term1}
    \lim_{k\to \infty} \sup_{j\geq k} \mathbb E \left(\|u^k_0-u^j_0\|_{s-1,j}^2\|u^k_0\|_{\widetilde{s+1}}^2\right)\leq\lim_{k\to \infty} \mathbb E \sup_{j\geq k}  \left(\|u^k_0-u^j_0\|_{s-1,j}^2\|u^k_0\|_{\widetilde{s+1}}^2\right) =0.
\end{align}
Concerning the second term, one can repeat the estimate for $\|U\|_{s,j}^2$ and get \eqref{step5-combine} with the index $s$ replaced by $s-1$:
\begin{align*}
    \mathbb E\left( \sup\limits_{t\in [0,\tau_{j,k}^T]} \|U\|_{s-1,j}^2\right) \leq &2\mathbb E \|U_0\|_{s-1,j}^2 + C_{\rho,\kappa}\mathbb E\int_0^{\tau_{j,k}^T} \|U\|_{s-1,j}^2 dt
    + C\mathbb E\int_0^{\tau_{j,k}^T}  (\|u^k\|_{s}^2 + \|u^j\|_{s}^2)\|U\|_{s-2}^2 dt.
    \\
    \leq &2\mathbb E \|U_0\|_{s-1,j}^2 + C_{\rho,\kappa}\mathbb E\int_0^{\tau_{j,k}^T} \|U\|_{s-1,j}^2 dt,
\end{align*}
where the trouble term originally appears in \eqref{trouble-term-1} disappears since we can control $\|u^k\|_{s}^2$ under stopping time $\tau_{j,k}^T.$ Then by the Gr\"onwall inequality and follow similarly as in the estimate of \eqref{trouble-term-1} one can eventually get
\begin{align}\label{term2}
    \lim_{k\to \infty} \sup_{j\geq k} \mathbb E \left(\sup\limits_{t\in [0,\tau_{j,k}^T]}\|u^k(t)-u^j(t)\|_{s-1,j}^2\right)=0.
\end{align}
Combining \eqref{term1} and \eqref{term2}, we now obtain
\begin{align}
   \lim_{k\to \infty} \sup_{j\geq k} \mathbb E  \left(\sup\limits_{t\in [0,\tau_{j,k}^T]} \|u^k\|_{s+1}^2 \|U\|_{s-1}^2 \right) \leq C_{\kappa} \lim_{k\to \infty} \sup_{j\geq k} \mathbb E  \left(\sup\limits_{t\in [0,\tau_{j,k}^T]} \|u^k\|_{\widetilde{s+1}}^2 \|U\|_{s-1,j}^2 \right) = 0,
\end{align}
which gives \eqref{cauchy-condition-1-additional}, and this completes the proof of \eqref{cauchy-condition-1}.

\subsubsection{Verification of condition~\eqref{cauchy-condition-2}}
First recall that since $\tau^T_{j}\leq \eta^j$, for each $j\geq J$ and $t\in[0,\tau^T_{j}]$ we have $\|u^j(t)\|_{s}<\frac\rho2$ and $u^j(t)\in \mathcal D_{s+1,\kappa}$ a.s.. 

Following the estimates in Proposition \ref{proposition:estimate-inviscid}, when $p=2$ we obtain
\begin{align*}
   \sup\limits_{t\in[0,\tau^T_j \wedge S]} \|u^j\|_{\tilde s}^2 
    \leq  \|u^j_0\|_{\tilde s}^2 + C_{\rho,\kappa} \int_0^{\tau^T_j \wedge S}(1+ \|u^j\|_{\tilde s}^2) dt + C_{\rho,\kappa}\sup\limits_{t\in[0,\tau^T_j \wedge S]} \left| \int_0^t A dW \right|,
\end{align*}
where
\begin{align*}
    A:= \left\langle  \sigma(u^j), u^j \right\rangle  + \sum_{\substack{0\leq |\alpha|\leq s \\ D^\alpha\neq \p_x^s}} \left\langle D^\alpha \p_z\sigma(u^j), D^\alpha v^j \right\rangle 
     +  \left\langle \frac {\p_x^s v^j}{\p_z v^j} , \p_x^s\p_z \sigma(u^j) \right\rangle + \frac12 \left\langle -\left(\frac{\p_x^s v}{\p_z v^j}\right)^2 , \p_{zz} \sigma(u^j) \right\rangle .
\end{align*}
Therefore,
\begin{align*}
    &\mathbb P \left( \sup_{t \in\left[0, \tau^T_j \wedge S\right]}\left\|u^j(t)\right\|^2_{\tilde s}>\left\|u^j_0\right\|^2_{\tilde s}+1 \right) =  \mathbb P \left( \sup_{t \in\left[0, \tau^T_j \wedge S\right]}\left\|u^j(t)\right\|^2_{\tilde s}- \left\|u^j_0\right\|^2_{\tilde s}>1 \right)
    \\
    \leq &\mathbb P \left( C_{\rho,\kappa} \int_0^{\tau^T_j \wedge S} (1+\|u^j\|_{\tilde s}^2) dt + C_{\rho,\kappa}\sup\limits_{t\in[0,\tau^T_j \wedge S]} \left| \int_0^t A dW \right| > 1 \right).
\end{align*}
By Markov inequality and BDG inequality, we have
\begin{align*}
   &\mathbb P \left( C_{\rho,\kappa} \int_0^{\tau^T_j \wedge S} (1+\|u^j\|_{\tilde s}^2) dt + C_{\rho,\kappa}\sup\limits_{t\in[0,\tau^T_j \wedge S]} \left| \int_0^t A dW \right| > 1 \right)
   \\
   \leq &C_{\rho,\kappa}\mathbb E \int_0^{\tau^T_j \wedge S} C_{\rho,\kappa} (1+\|u^j\|_{\tilde s}^2) dt + C_{\rho,\kappa}\mathbb E \sup\limits_{t\in[0,\tau^T_j \wedge S]} \left| \int_0^t A dW \right|
   \\
   \leq &C_{\rho,\kappa} S + C_{\rho,\kappa}\mathbb E \left(\int_0^{\tau^T_j \wedge S} (1+\|u^j\|_{\tilde s}^4) dt \right)^{\frac12} \leq C_{\rho,\kappa} S + C_{\rho,\kappa} S^{\frac12}\to 0 \quad \text{ as } S \to 0.
\end{align*}
Therefore the condition \eqref{cauchy-condition-2} is valid.

\subsection{Step 3: Local pathwise solution with \texorpdfstring{$L^2(\Omega)$}{Lg} initial data by localization}\label{sec:proof-localization}

First, using the estimate \eqref{cauchy-result-3} in Lemma~\ref{lemma:cauchy}, we know that 
\[
\mathbb E \sup\limits_{t\in[0,T]} \|u(\cdot\wedge\tau)\|_{\tilde s}^2  = \mathbb E \sup\limits_{t\in[0,\tau]} \|u\|_{\tilde s}^2 \leq C_\kappa \mathbb E(1+ \|u_0\|^2_{\tilde s}) \leq C_{\kappa}(1+\frac M2) < \infty,
\]
thus $u(\cdot\wedge\tau)\in L^2(\Omega; C([0,T],\mathcal D_{s,\kappa}))$. Till now, we obtain an unique pathwise solution $u(\cdot\wedge\tau)\in L^2(\Omega; C([0,T],\mathcal D_{s,\kappa}))$ provided that \eqref{ic-requirement} holds. Next, we extend this result to the case when $u_0\in L^2(\Omega; \mathcal D_{s,2\kappa})$.

For each $k\in \mathbb N$, denote by 
\[
\Omega_k:=\left\lbrace k-1 \leq \Vert u_0 \Vert_{\tilde s} < k \right\rbrace \subseteq \Omega, \qquad u_{0,k} = \mathds{1}_{\Omega_k} u_0,
\]
then one has $\|u_{0,k}\|_{\tilde s} < k$ \,a.s.. For each $k\in \mathbb N$, consider the modified system \eqref{PE-modified-system} with $\rho = (1+\tilde{C_\kappa})(\frac{2kC_s}{\tilde{c_\kappa}} + 4)(1+\frac1{\tilde{c_k}})$, where the choice of $\rho$ is inspired by \eqref{cauchy-rho}.
Then by Section~\ref{sec:step1} and \ref{sec:step2} there exists an unique local pathwise solutions $(u_k,\tau_k)$ to the original system \eqref{PE-inviscid-system} with initial data $u_{0,k}$. Define
\[
u = \sum\limits_{k=1}^\infty u_k \mathds{1}_{\Omega_k}, \quad \tau = \sum\limits_{k=1}^\infty \tau_k \mathds{1}_{\Omega_k}.
\]
As $0<\tau_k\leq T$ a.s., we know that $0<\tau\leq T$ a.s.. Moreover, since the filtration $\mathbb F$ is right-continuous, $\tau$ is a stopping time (see \cite[page 6--7]{karatzas1991brownian}).
Using \eqref{cauchy-result-3} again, we
know that $\sup\limits_{t\in[0,\tau_k]}\|u_k(t)\|_{\tilde s} \leq C_{\kappa}(1 + \|u_{0,k}\|_{\tilde s})$ a.s.. Then one can compute
\begin{align}\label{L2-bdd}
    \mathbb E \sup\limits_{t\in[0,T]} \|u(\cdot\wedge \tau)\|_{\tilde s}^2 = \mathbb E \sup\limits_{t\in[0,\tau]} \|u\|_{\tilde s}^2 = & \mathbb E \sum\limits_{k=1}^\infty \mathds{1}_{\Omega_k} \sup\limits_{t\in[0,\tau_k]} \|u_k\|_{\tilde s}^2 \nonumber
    \\
    \leq &C_{\kappa}\mathbb E \sum\limits_{k=1}^\infty \mathds{1}_{\Omega_k} (1+ \|u_{0,k}\|_{\tilde s})^2\leq C_{\kappa}(1+ \mathbb E \|u_0\|_{\tilde s}^2) < \infty.
\end{align}
The fact that $u_k(\cdot\wedge \tau_k) \in L^2(\Omega; C([0,T],\mathcal D_{s,\kappa}))$ together with \eqref{L2-bdd} imply that $u(\cdot\wedge \tau)\in L^2(\Omega; C([0,T], \mathcal D_{s,\kappa}))$, and thus $(u,\tau)$ is a local pathwise solution to the original system \eqref{PE-inviscid-system} corresponding to the initial data $u_0\in L^2(\Omega; \mathcal D_{s,2\kappa})$.

Finally, the extension to the maximal pathwise solution $(u, \{\eta_n\}_{n\in\mathbb N},\xi)$ follows the standard process, see, for example, \cite{glatt2009strong,glatt2014local,crisan2019solution}. Note that the condition $\|\partial_{zz} u - \partial_{zz} u_0\|_{L^\infty} = \frac{\kappa}4$  on the set $\{\xi<\infty\}$ is due to \eqref{tilde_tau} and the definition of $\tau.$ This concludes the proof of Theorem \ref{thm:main-1}.

\section{Concluding Remarks}\label{sec:rmk}
We establish the local in time existence and uniqueness of maximal pathwise solutions in Sobolev spaces to the 2D stochastic hydrostatic Euler equations with multiplicative noise. In the deterministic setting, this model is known to be ill-posed in Sobolev spaces; we address this issue by imposing a local Rayleigh condition on the initial data.

Unlike many other SPDEs (e.g., \cite{debussche2011local,brzezniak2021well}), the Galerkin system is not suitable as an approximation scheme here since the key cancelation \eqref{eqn:cancellation} is not valid for the Galerkin system (see Remark \ref{rmk:cancellation} for more details). We overcome this difficulty by considering the horizontally viscous PEs as the approximation scheme. In order to obtain a pathwise solution that is continuous in time with desired regularity, we first smooth the initial data by projecting it onto the spaces with finite Fourier modes, to obtain a sequence of smooth solutions. Then we develop an abstract Cauchy theorem to prove that this sequence will converge to a solution with desired regularity with an a.s.\ positive stopping time. Our Cauchy theorem is analogous to the one that has been established in some previous works \cite{glatt2009strong,glatt2014local}. However, our functional spaces are more complicated due to the involvement of the local Rayleigh condition, and we need to perform more delicate analysis to prove our abstract Cauchy theorem.

This work gives the first result concerning the existence and uniqueness of solutions to the stochastic hydrostatic Euler equations in Sobolev spaces, and the first result
on the existence of pathwise solutions to this model. 

\section*{Acknowledgement}
R.H. was partially supported by the Simons Foundation (MP-TSM-00002783), the ONR grant under \#N00014-24-1-2432, and the 
NSF grant DMS-2420988. Q.L. was partially supported by the Simons Foundation (SFI-MPS-TSM-00013384).

\section*{Data Availability}
Data sharing not applicable to this article as no datasets were generated or analyzed during the current study.

\bibliographystyle{plain}
\bibliography{Reference}

\end{document}